\documentclass[a4paper, 11pt]{article}
\usepackage{graphicx}
\usepackage{amssymb}
\usepackage{epstopdf}
\usepackage{fancyhdr}
\usepackage{amsmath}
\usepackage{amsthm}
\usepackage{hyperref}
\usepackage{mathdesign}
\usepackage{color}
\usepackage{soul}
\usepackage{wrapfig}

\usepackage[top=2.2 cm,bottom=2.2 cm,left=2.5 cm, right=2.5 cm]{geometry}

\newcommand{\mL}{\mathcal{L}}
\newcommand{\T}{\mathcal{T}}

\newcommand{\cE}{{\cal E}}
\newcommand{\cI}{{\cal I}}

\newcommand{\cT}{{\cal T}}
\newcommand{\cW}{{\cal W}}

\newcommand{\be}{{\bf e}}\newcommand{\bff}{{\bf f}}

\newcommand{\bu}{{\bf u}}

\newcommand{\bA}{{\bf A}}

\newcommand{\bT}{{\bf T}}

\newtheorem{theorem}{Theorem}[section]
\newtheorem{lemma}[theorem]{Lemma}

\newtheorem{definition}[theorem]{Definition}

\newtheorem{remark}[theorem]{Remark}

\newtheorem{assumption}{Assumption}

\newcommand{\bbE}{\mathbb{E}}

\newcommand{\ve}{\varepsilon}

\def\XXint#1#2#3{{\setbox0=\hbox{$#1{#2#3}{\int}$ }
\vcenter{\hbox{$#2#3$ }}\kern-.58\wd0}}

\def\XXint#1#2#3{{\setbox0=\hbox{$#1{#2#3}{\int}$}
\vcenter{\hbox{$#2#3$}}\kern-.87\wd0}}

\def\XXiint#1#2#3{{\setbox0=\hbox{$#1{#2#3}{\int}$}
\vcenter{\hbox{$#2#3$}}\kern-1.05\wd0}}

\def\XXintt#1#2#3{{\setbox0=\hbox{$#1{#2#3}{\int}$}
\vcenter{\hbox{$#2#3$}}\kern-.72\wd0}}

\def\Xinttt#1{\mathchoice
{\XXinttt\displaystyle\textstyle{#1}}%
{\XXinttt\textstyle\scriptstyle{#1}}%
{\XXinttt\scriptstyle\scriptscriptstyle{#1}}%
{\XXinttt\scriptscriptstyle\scriptscriptstyle{#1}}%
\!\int}
\def\XXinttt#1#2#3{{\setbox0=\hbox{$#1{#2#3}{\int}$}
\vcenter{\hbox{$#2#3$}}\kern-.52\wd0}}

\def\XXintttr#1#2#3{{\setbox0=\hbox{$#1{#2#3}{\int}$}
\vcenter{\hbox{$#2#3$}}\kern-.6\wd0}}

\def\XXintttt#1#2#3{{\setbox0=\hbox{$#1{#2#3}{\int}$}
\vcenter{\hbox{$#2#3$}}\kern-.78\wd0}}

\def\sqr#1#2{{\vcenter{\vbox{\hrule height.#2pt\hbox{\vrule width.#2pt height#1pt \kern#1pt\vrule width.#2pt}\hrule height.#2pt}}}}

\def\ddashinttt{\Xinttt-}

\title{Locally periodic homogenization of multiscale models for plant tissue  biomechanics}

\author{Mariya Ptashnyk\medskip\\
\small  Maxwell Institute for Mathematical Sciences, Department of Mathematics, School of\\  \small  Mathematical and Computer Sciences, Heriot-Watt University, Edinburgh, Scotland, UK}
\date{}

\begin{document}

\maketitle

\begin{abstract}
In this work, the derivation and multiscale analysis of a mathematical model for plant tissue biomechanics is considered. The microscopic model consists of a coupled systems of  equations of linear elasticity, describing mechanical deformation, and reaction-diffusion and ordinary differential equations modelling the dynamics of load-bearing cross-links in cell walls and middle lamella, connecting individual cells in a plant tissue.  It takes into account a two-way coupling between mechanical deformation and  chemical processes in  cell walls and middle lamella, where elastic properties  depend on the density of cross-links and reactions terms, defining the stretching and breakage of the cross-links, depend on the deformation gradient. The existence and uniqueness results for  the strongly coupled problem are shown by applying the  Banach fixed point theorem  and using the Alikakos iteration algorithm in the derivation of a priori estimates and the contraction property. The nonlinear dependence of solutions of the ordinary differential equation on the displacement gradient and of the boundary conditions in reaction-diffusion problem  on the displacement induces novel approaches in the derivation of a priori estimates. Using homogenization techniques of locally periodic (l-p) two-scale convergence, a macroscopic model for plant tissue biomechanics is derived, representing  the first result in the multiscale analysis of a biomechanical model for  tissues with non-periodic fibrous microstructure. For multiscale analysis and derivation of the macroscopic model, the non-periodic microstructure of plant cell walls, characterised by rotated planes of parallel aligned microfibrils,  is approximated by the corresponding locally periodic microstructure.  The   two-way coupling between equations of linear elasticity and reaction-diffusion and ordinary differential  equations requires the proof of the strong l-p two-scale convergence for the deformation gradient and for the density of stretched cross-links, in order to pass to the limit in the nonlinear  functions.

\end{abstract}

\section{Introduction}
The growth and development of plants involve many interconnected processes, that occur at different
spatial and temporal scales, from cell walls, with their highly dynamic biochemical processes and
complex microstructure, to different types of tissues and organs. Understanding  the influence of microscopic structure and  molecular interactions on the macroscopic properties of plant tissues is important, especially in the context of climate change and impact of environmental conditions on plant growth. The mechanical properties and the growth of plant tissues are  strongly determined by the structure of the cell walls and the adhesion between cells. The cell walls  consist mainly of oriented cellulose microfibrils (MF), pectin, hemicellulose, structural proteins and water, whereas the cross-linked pectin network is the main composite of the middle lamella, joining individual cells  together.  The orientation of microfibrils   and   the distribution of load-bearing calcium-pectin cross-links    strongly affect  the mechanical properties of plant cell walls~\cite{P2011,P2010,Wolf}. There are a number of models describing plant cell walls, each focusing on different aspects of their structure. Mathematical models of the  cellulose-hemicellulose network were proposed in \cite{DBJ,PF}.  The account of the microstructure of a cell wall has been addressed by considering the anisotropic yield stresses or by distinguishing between the free energies related to the elasticity of (i) macromolecules and hydrogen bonds or (ii) the matrix and microfibrils \cite{Dumais,VC}. The influence of the microfibril orientation and the external torque on the expansion process has been considered in~\cite{DJ}. The effects of changes in the chemical configuration of  pectins (methylesterified and demethylesterified) and  the calcium concentration on the viscous behaviour of the cell wall in a pollen tube has been analysed  in~\cite{KZG,RHD}. The derivation, multiscale analysis, and simulation of  models for plant cell wall biomechanics, accounting for   the two-way coupling between the mechanical deformation and the chemistry of calcium-pectin cross-links, were presented in~\cite{PS_1_2016, PS_2_2016, PS_2017}.

In this work, the results of~\cite{PS_1_2016} are extended by considering a more realistic dynamics of calcium-pectin cross-links and  distinguishing between  newly formed and stretched  cross-links, as well as by considering the non-periodic distribution of microfibrils in cell walls, characterised by rotated planes of parallel-aligned microfibrils. The non-periodic microstructure considered here has been observed experimentally in plant root cells~\cite{Anderson}, presumably as a result of the reorientation of microfibrils in the principal direction of growth during elongation. This reorientation may be one of the mechanisms underlying plant cell maturation and growth cessation.  We model a part of  plant tissue as a collection of cell walls joined by the middle lamella, and assume that cell wall matrix  and middle lamella are isotropic and   linearly elastic, whereas microfibrils are modelled as an anisotropic,  linearly elastic material.  The interplay between mechanics and cross-link dynamics comes in by assuming that the elastic  properties of the wall matrix and middle lamella depend on the density of the cross-links and that stresses within the cell wall can stretch and ultimately  break calcium-pectin cross-links.   The deformation-dependent opening of calcium channels in the cell plasma membrane is addressed in the flux boundary conditions for calcium ions. The dynamics of newly formed cross-links is described by reaction-diffusion equations, whereas the stretched cross-links are modelled by an ordinary-differential equation (ODE), assuming that they are not diffusing in the cell wall matrix,  with a local dependence of the reaction terms on the deformation gradient, which improves the modelling approach considered in~\cite{PS_1_2016}. The existence of a weak solution of the microscopic equations is shown using a classical approach and applying the Banach fixed-point theorem. However, due to quadratic nonlinearities in the reaction terms, the proof of a priori estimates and contraction inequality is not standard and relies on delicate estimates for the $L^\infty$-norm of  solutions of the reaction-diffusion and ODE system in terms of the $L^2$-norm of the displacement gradient, by applying  the Alikakos iteration technique~\cite{Alikakos}. To pass to the limit in the nonlinear function of the displacement in the boundary condition for calcium concentration, we prove the equicontinuity property with respect of the time variable for the displacement gradient.  Using homogenization techniques, i.e.~the locally periodic (l-p) two-scale convergence and l-p unfolding method~\cite{Ptashnyk_2013, Ptashnyk_2015},  we derive the macroscopic model for  plant tissue biomechanics. The main mathematical difficulty in the derivation of the macroscopic problem arises from the strong coupling between the equations of elasticity and the system of reaction-diffusion and ordinary differential equations. In order to pass to the limit in the nonlinear reaction terms, we prove the strong l-p two-scale convergence for the displacement gradient and for the density of stretched cross-links.

 There are  several mathematical results on  homogenization in locally periodic and  fibrous media.  The homogenization of a heat-conductivity problem defined in locally periodic and  non-periodic  domains consisting of spherical  balls   was studied in \cite{Briane3} using the Murat-Tartar $H-$convergence method~\cite{Murat}.
The notion of $\theta-2$ convergence, motivated by  the homogenization of  elliptic problems in a domain  with a microstructure of non-periodically  distributed  spherical balls, was introduced in~\cite{Alexandre}.  The Young measure was used in~\cite{Mascarenhas} to extend  the concept of periodic two-scale convergence~\cite{Allaire, Nguetseng}, and to define the so-called  scale convergence,   motivated by the derivation of the $\Gamma$-limit for a sequence of nonlinear energy functionals involving non-periodic oscillations. In~\cite{Mascarenhas},
as an example of non-periodic oscillations, a domain with perforations given by a transformation of the centres of balls was considered. The asymptotic expansion method was used  in~\cite{Mikelic} to derive macroscopic equations for a filtration problem  through a locally periodic fibrous medium and in~\cite{AdrianTycho2}  to derive a macroscopic model  for convection-diffusion equations defined  in  domains with  locally periodic perforations, i.e.~domains consisting of  periodic cells with  smoothly changing perforations.
Macroscopic models for elliptic equations defined in non-periodic fibrous materials, characterised by gradually rotated  planes of parallel-aligned fibres, were presented in~\cite{Briane2} and derived in~\cite{Briane1}.    By applying the $H-$convergence method for a locally periodic approximation  of the non-periodic microstructure the effective homogenized matrix was derived.
In this work, we consider a similar approximation of the non-periodic microstructure by a locally periodic one
 for a strongly coupled system of linear elasticity, reaction–diffusion, and ordinary differential equations. Using l-p two-scale convergence and l-p unfolding methods, we derive the macroscopic model and obtain new results for the multiscale analysis of nonlinear PDEs in domains with non-periodic microstructure.

The paper is organised as follows. In Section~\ref{micro_model} we derive a microscopic model for plant tissue biomechanics, which is then analysed in Section~\ref{analysis}. Derivation of the macroscopic problem is presented in Section~\ref{macro_model}, followed by  concluding remarks in Section~\ref{conclusion}.

\section{Derivation of microscopic  model for plant  tissue biomechanics}\label{micro_model}
In the microscopic model,  we represent a part of  plant tissue,  a bounded Lipschitz domain $\Omega \subset (x_{1, {\rm min}}, x_{1, {\rm max}}) \times (x_{2, {\rm min}}, x_{2, {\rm max}}) \times (x_{3, {\rm min}}, x_{3, {\rm max}})$, for $x_{j, {\rm min}}, x_{j, {\rm max}} \in \mathbb R$, $j=1,2,3$,  as a collection of cell walls joined together by the middle lamella, with
 the effect of the cell interior modelled through the internal turgor pressure acting on the cell membranes.  We distinguish between mechanical properties of cell wall matrix, middle lamella, and cellulose microfibrils, and as one of the main chemical processes influencing mechanical properties of plant tissues we consider the calcium-pectin chemistry.  Pectin is deposited into the cell walls in a highly methylestrified state and is modified by the wall enzyme pectin-methylesterase (PME), which removes methyl groups \cite{WG}.   The demethylesterified pectin  interacts with calcium ions to produce load-bearing cross-links, which then reduce the cell wall expansion~\cite{WHH}. It is   assumed that mechanical forces can stretch and eventually break the cross-links.

For mechanical deformation we consider the momentum balance equations
$$
\text{div}\,\bT  =0 \qquad \text{in}\ \Omega,
$$
and, assuming the  constitutive law of linear elasticity, see e.g.~\cite{Ciarlet},  we have
$$\bT= \big(\bbE_{M_w}(b)\chi_{\Omega_{WM}} + \bbE_{M_l}(b)\chi_{\Omega_{ML}}  + \bbE_F \chi_{\Omega_F}\big)\be(\bu),
$$
where  $\bbE_{M_w}(b)$, $\bbE_{M_l}(b)$, and $\bbE_F$ are elasticity tensors for cell wall matrix, the middle lamella,  and microfibrils, respectively,  $\bu$ denotes the displacement,  $\be(\bu)=(\nabla\bu+\nabla\bu^{\rm T})/2$, and   $\chi_{\Omega_{WM}}$, $\chi_{\Omega_{ML}}$, and $\chi_{\Omega_F}$ are the characteristic functions of  domains occupied by wall matrix $\Omega_{WM}$,  middle lamella $\Omega_{ML}$, and  microfibrils $\Omega_F$, respectively.
We assume that the elastic properties of the cell wall matrix and the middle lamella depend on the density of calcium–pectin cross-links, while the microfibrils are modelled as a transversely isotropic material~\cite{DMKM}, see Appendix for details.  The  forces acting on the plant tissue include the turgor pressure~$P$  within the cells and  traction forces~$\bff$ exerted by neighbouring cells or external  forces.

In modelling the dynamics of  calcium-pectin cross-links,   we  account for the random movement of PME $p_E$,   methylesterified $p_e$ and  demethylesterified  $p_d$ pectin, calcium ions $c$, and newly formed calcium-pectin cross-links $b$, \cite{HSP}, while assuming  that stretched cross-links $b_s$ do not diffuse.
 For the chemical interactions we consider changes in the methylesterified properties of pectin and formation,  decay, stretching  and  breakage of  calcium-pectin cross-links.
We assume that the binding of PME  to and  dissociation  from  a pectin acid group are very fast and it is not used up during the demethylesterification process. PME and methylesterified pectin are  produced by  cells and  transported into the cell walls through the boundary  $\Gamma_\cI \subset \partial \Omega$, representing the cell plasma membranes, and  can also leave the cell walls through $\Gamma_\cI$ to degrade~\cite{WG}.  To account for mechanisms  controlling the amount of pectin and PME in the cell walls, we assume that the  inflow of new methylestrified pectin  decreases with an increasing   amount of methylestrified and demethylesterified  pectins, and calcium-pectin cross-links  in the  walls and  the  inflow  of PME  into  the cell wall  depends on the amount of methylestrified pectin within the  wall  and high density of the demethylesterified pectin  and calcium-pectin cross-links inhibits  the inflow of new PME.
Due to the high calcium concentration in plant cell walls,  we assume saturation kinetics for formation of calcium-pectin cross-links. Experimental observations indicate that, during  growth, cellulose microfibrils  rotate across the  cell wall thickness, resulting in a non-periodic microstructure of rotated planes of parallel-aligned microfibrils~\cite{Anderson}.

To define the microscopic structure of the cell walls we consider   $Y= (-\frac 12, \frac 12)^3$ and  $Y^\ast = Y \setminus \overline Y_0$, where   $Y_0 = \{ y \in \overline Y: \, |\hat y|< a\}$ with $a<1/2$,  $\hat y = (y_2, y_3)$, $\Gamma = \partial Y_0$,  and characteristic function $\vartheta(y) = \chi_{Y_0}(y)$,
extended $Y$-periodically to $\mathbb R^3$.

\begin{wrapfigure}[18]{r}{5.4cm}
\vspace{-0.9 cm}
\vspace{-0.1cm}\includegraphics[width=0.325\textwidth]{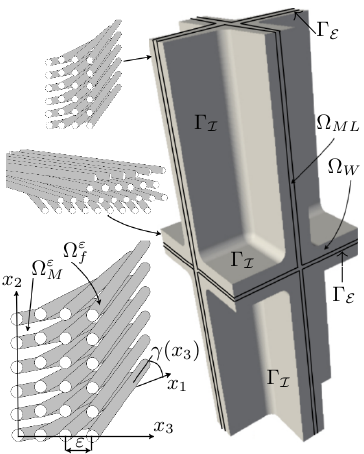}
\caption{\small A sketch of geometry and non-periodic microstructure of a part of plant root tissue, based on scanned images of plant root cells in  elongation zone~\cite{PS_2_2016}.}\label{fig_1}
\end{wrapfigure}
For $\xi\in \mathbb Z^3$ define   $ x_{\xi}^\ve=\ve R_{x_\xi^\ve}  \xi$, with $R_{x}:= R(\gamma(x_3))$,  where  $\gamma \in C^2(\mathbb R)$ with  $-\pi/2\leq \gamma(x) \leq \pi/2$ for  $x\in \mathbb R$, and rotation matrix $R(\alpha)$     around the  $x_3$-axis  and  rotation angle~$\alpha$ with the $x_1$-axis
$$
R(\alpha)=\begin{pmatrix}
 \cos(\alpha) &-\sin(\alpha)& 0\\
\sin(\alpha) &\phantom{-}\cos(\alpha)& 0\\
 0 & 0 & 1
\end{pmatrix}.
$$
Then the  characteristic function of rotated layers of parallel-aligned microfibrils is given by
\begin{equation}\label{elastnonper}
\chi_{\Omega_f^\ve} (x)
=\chi_{\Omega_W}(x)\sum\limits_{\xi\in \Xi^\ve}  \vartheta_\ve\big(R_{x^\ve_\xi}^{-1}( x- x_{\xi}^\ve)\big),
\end{equation}
where  $ \vartheta_\ve(x)= \vartheta( x/\varepsilon)$, with $\ve$ denoting  the ration between the diameter of the microfibrils and
 the cell wall thickness, $\Xi^\ve = \{ \xi \in \mathbb Z^3  : \, \ve R_{x^\ve_\xi}(\overline Y + \xi) \subset \Omega_W \}$,  and $\Omega_W =\Omega \setminus \overline{\Omega}_{ML}$ denotes the part of the domain occupied by cell walls,   since  the middle lamella  $\Omega_{ML}$ does  not contain  microfibrils.
The domain of both the cell wall matrix and middle lamella is given by
$\Omega^\ve_M = \Omega \setminus \overline{\Omega}^\ve_f$, where $\Omega_f^\ve = \bigcup_{\xi \in \Xi^\ve} \ve R_{x_\xi^\ve}(Y_0 + \xi)$, see Figure~\ref{fig_1}.  We also define  $\tilde Y^\ast_{x} = R_x Y^\ast$, $\tilde Y_x = R_x Y$.

The external  boundaries are defined by
$\Gamma_{\cE} = \{ x \in \partial\Omega \, : \,  x_1 = x_{1, {\rm max}}, x_1 = x_{1, {\rm min}}, x_2 = x_{2, {\rm max}}, x_2 = x_{2, {\rm min}}, x_3 = x_{3, {\rm max}}, \text{ or } x_3 = x_{3, {\rm min}}\}$ and the internal boundaries $\Gamma_{\cI} = \partial \Omega \setminus \Gamma_{\cE}$ correspond to the cell plasma membranes.

The microscopic model for elastic deformations is given by
\begin{equation}\label{EQST}
\begin{aligned}
-\text{div}  \left(\bbE^\ve(x, b^\ve)\,\be(\bu^\varepsilon)\right)&=0  &&  \text{ in } \Omega, \\
\bbE^\ve(x, b^\ve)\,\be(\bu^\varepsilon)\nu&= - P\, \nu  && \text{ on } \Gamma_\cI, \\
\bbE^\ve(x, b^\ve)\,\be(\bu^\varepsilon)\nu&=  \bff   && \text{ on }  \Gamma_\cE,
\end{aligned}
\end{equation}
where
$\bbE^\ve(x, b^\ve)= \bbE_f\chi_{\Omega_f^\ve}(x) +  \bbE_M(x, b^\ve)\big(1-\chi_{\Omega_f^\ve}(x))$,\,
 $\bbE_M(x, b^\ve)= \bbE_{M_w}(b^\ve)\chi_{W}(x) + \bbE_{M_l}(b^\ve)\chi_{ML}(x)$, and $\chi_W$, $\chi_{ML}$ are the characteristic functions of the cell walls $\Omega_W$ and middle lamella $\Omega_{ML}$, respectively.

 The calcium-pectin dynamics is defined in~$(0,T)\times\Omega^\ve_M$  by
\begin{eqnarray}\label{main_2}
\begin{aligned}
\partial_t p_E^\ve& =\text{div}(D_E\nabla p^\ve_E),\qquad\qquad \quad
\partial_t p^\ve_e =\text{div}(D_e\nabla p_e^\ve) - R_{e}(p^\ve_e,p^\ve_E),\\
\partial_t p^\ve_d& =\text{div}(D_d\nabla p^\ve_d) - 2 R_{d}(p^\ve_d,c^\ve) + 2 R_s(b^\ve, b^\ve_s,\be(\bu^\ve))+ R_{e}(p^\ve_e,p^\ve_E),    \\
\partial_t c^\ve &= \text{div}(D_c\nabla c^\ve) -    \;\, R_{d}(p^\ve_d,c^\ve) + \; R_s(b^\ve, b^\ve_s, \be(\bu^\ve)) ,  \\
\partial_t b^\ve&= \text{div}(D_b\nabla b^\ve) + \; R_{d}(p^\ve_d, c^\ve) - R_b(b^\ve,  \be(\bu^\ve)) - d_b b^\ve ,\\
\partial_t b^\ve_s & = \phantom{\text{div}(D_b\nabla b^\ve)  } \; \; \;   R_{b}(b^\ve,  \be(\bu^\ve)) -R_{s}(b^\ve, b^\ve_s, \be(\bu^\ve))  - d_s b_s^\ve,
\end{aligned}
\end{eqnarray}
with   boundary conditions
\begin{equation} \label{BC}
 \begin{aligned}
 D_\alpha \nabla p^\ve_\alpha\cdot\nu=&J_\alpha(p_e^\ve, p^\ve_d, b^\ve)- \gamma_{\alpha} p_\alpha^\ve,    &&  D_c\nabla c^\ve\cdot\nu=  J_c(\bu^\ve) - \gamma_{c} c^\ve   &&  \text{on }  \Gamma_\cI,\\
  D_\alpha \nabla p^\ve_\alpha\cdot\nu= & 0, \qquad  D_c\nabla c^\ve\cdot\nu = 0    &&  \text{on }  \Gamma_{\cE}, \qquad
  D_b\nabla b^\ve\cdot\nu=0    &&  \text{on }   \partial \Omega,
\end{aligned}
\end{equation}
for $\alpha = E,e,d$ and  $J_d((p_e^\ve, p^\ve_d, b^\ve) \equiv 0$, and initial conditions
\begin{equation} \label{IC}
   p_\alpha(0) = p_{\alpha,0}, \quad  c^\ve(0) = c_{0}, \quad b^\ve(0) = b_{0}, \quad
  b^\ve_s(0) = b_{s,0}.
\end{equation}
 Here
$$ R_{b}(b^\ve,  \be(\bu^\ve))  =    r_b(b^\ve )  Q_b(b^\ve, \be(\bu^\ve)),\;\;
R_s(b^\ve, b_s^\ve, \be(\bu^\ve))= r_s(b_s^\ve) Q_s(b^\ve,  \be(\bu^\ve)),
$$
where, e.g., $r_b(b^\ve) =  k_b  b^\ve$, $r_s(b^\ve_s) =  k_s  b^\ve_s$ and
$Q_b, Q_s:  \mathbb R\times \mathbb R^{3\times 3} \to \mathbb R$ are  non-negative  and locally Lipschitz continuous functions, and $Q_b$ is bounded,
\begin{equation*} 
Q_b (b,  \be(\bu)) =\frac{\big(\text{tr } \bbE(x, b) \be(\bu)\big)^{+}}{K_b+\big(\text{tr } \bbE(x, b) \be(\bu)\big)^{+}},  \quad
Q_s(b, \be(\bu)) =\big(\text{tr } \tilde \bbE(x, b) \be(\bu)\big)^{+},
\end{equation*}
for  $k_b, k_s >0$ and $K_b>0$,  assuming that there is a maximum cross-links stretching   rate for large forces and $ \tilde \bbE^\ve(x, b)$ is uniformly bounded in $b$.  In the case of the diffusion of  $b_s$, the boundedness assumption on $Q_b$ can be relaxed~\cite{PS_2_2016}.
  Possible, biologically motivated, examples for reaction terms and  boundary conditions are
\begin{equation*}
\begin{aligned}
& R_e(p^\ve_E, p^\ve_e) = \frac{k_{e1} p_e^\ve p_E^\ve}{1+ k_{e2}  p_e^\ve},
&& R_{d}(p^\ve_d,c^\ve)= \frac{k_{c1} p_d^\ve c^\ve}{k_{c2} + c^\ve}, \qquad J_c(\bu) = \beta_c |\bu|,\\
& J_E(p_e, p_d, b)= \frac{\beta_E p_e}{1+ \zeta_E (p_d+  b)}, && J_e(p_e, p_d, b)= \frac{\beta_e}{1+ \zeta_e p_e +\zeta_d p_d+ \zeta_b b },
\end{aligned}
\end{equation*}
where $k_{li}>0$, for $l=e,c$ and  $i=1,2$, $\beta_j>0$, for $j=E,e,c$, and $\zeta_m>0$, for $m=E,e, d,b$.

\begin{remark} For  simplicity of presentation,  we assume that in all cell walls within a part of plant tissue the layers of microfibrils are rotated around $x_3$-axis. In general, the   layers of  microfibrils are rotated around the axis orthogonal to the cell wall. However, the simplified scenario considered here is not restrictive for the modelling and  multiscale analysis, since  for the more realistic microstructure the analysis  follows the same lines  by considering  different rotation matrices $R(\alpha)$ in different parts of the cell walls.
\end{remark}

\section{Well-posedness results for microscopic  model \eqref{EQST} -- \eqref{IC}}\label{analysis}
To define a weak solution of microscopic model~\eqref{EQST}--\eqref{IC},
we consider the  space
$$
\cW(\Omega) =\{\bu\in H^{1}(\Omega)^3\ \big | \int_\Omega\!\!\bu  dx=0,\;  \int_\Omega \!\!\big[(\nabla\bu)_{ij}- (\nabla \bu)_{ji}\big] dx=0, \;  i,j=1,2,3, \, i \neq j\},
$$
and  the following assumptions
\begin{assumption}\label{assumptions}
\begin{enumerate}
\item The diffusion coefficients satisfy   $D_{\alpha}\geq d>0$, with $\alpha=e,E,d,c, b$.\\
The decay rates satisfy  $\gamma_e, \gamma_E>0$, $\gamma_d, \gamma_c \geq 0$,  $d_b \geq 0$, and $d_s > 0$.
\item  $R_{e} \in C^1(\mathbb R^2)$ and  $R_{e}(\xi, \eta)$ is positive for  positive $\xi$,  $\eta$, with    $R_{e}(0, \eta)= 0$  for  $\eta \in \mathbb R_+$.\\
  $R_{d}:\mathbb R^2 \to \mathbb R$  is locally  Lipschitz continuous,
$R_{d}(\xi, \eta)$ is positive for positive $\xi$, $\eta$, with
$R_{d}(\xi, 0)=0$, $R_{d}(0, \eta)=0$ for $\xi , \eta\in \mathbb R_+$,  and   $R_{d}(\xi, \eta) \leq \beta_{d}(1+ \xi)$.
\item $J_e, \, J_E: \mathbb R^3\to \mathbb R$  are continuously  differentiable in $(- \beta, + \infty)^3$  with $\beta>0$,   and  $0\leq J_E(\xi, \eta, \zeta)\leq \beta_E(1+ \xi)$, \  $0\leq J_e(\xi,\eta, \zeta)\leq \beta_e$ for $\beta_E, \beta_e >0$ and  $\xi, \eta, \zeta \in \mathbb R_+$. \\
$J_c:\mathbb R^3 \to \mathbb R$ is Lipschitz continuous, $0 \leq J_c({\bf v} )\leq \beta_c(1+ |{\bf v}|)$ for $\beta_c >0$ and ${\bf v} \in \mathbb R^3$.
\item Functions $r_b, r_s : \mathbb R \to \mathbb R$  are  Lipschitz continuous,  $r_l(\xi)\geq 0$  for $\xi\geq 0$,  $r_l(0) = 0$,  and $(r_l(\xi_1) - r_l(\xi_2))(\xi_1 - \xi_2) \geq \rho_l |\xi_1-\xi_2|^2$ for $\xi_1, \xi_2 \in \mathbb R_+$ and   $\rho_l \geq 0$, with  $l=b,s$.\\
  $Q_l: \mathbb R\times \mathbb R^{3\times 3} \to \mathbb R$ is continuous, $l=b,s$,  and  $0 \leq Q_b(\xi, \bA) \leq \beta_{b}$,   $0 \leq Q_s(\xi, \bA) \leq \beta_{s}(1+|\bA|)$  for   $\beta_{b}, \beta_{s}>0$ and all $\bA \in  \mathbb R^{3\times 3}$,  $\xi \in \mathbb R_+$,  and
  $$
  \begin{aligned}
  & |Q_l(\xi, \bA_1)-  Q_l(\xi,  \bA_2)| \leq \beta_1(1+  \xi)|\bA_1 - \bA_2|,  \\
  & |Q_l(\xi_1,  \bA)-  Q_l(\xi_2,  \bA)| \leq \beta_2(1+ |\bA|)|\xi_1 - \xi_2|,
  \end{aligned}
  $$
  for $l=b,s$, some
  $\beta_1, \beta_2 >0$ and  all $\xi,  \xi_1, \xi_2 \in \mathbb R_+$ and  $\bA, \bA_1, \bA_2 \in \mathbb R^{3\times 3}$.
  \item Elasticity tensors  $\mathbb E_{M_w}, \mathbb E_{M_l} \in C^1(\mathbb R)$, and  $\mathbb E_F$, $\mathbb E_{M_w}, \mathbb E_{M_l}$ possess major and minor symmetries,
  and $\mathbb E_F \bA \cdot \bA \geq \omega_E |\bA|^2$,   $\mathbb E_{M_j}(\xi) \bA \cdot \bA \geq \omega_E |\bA|^2$  for  $\omega_E >0$ and  all symmetric $\bA \in \mathbb R^{3\times 3}$, $\xi \in \mathbb R_+$,  and $|\mathbb E_{M_j}(\xi)| \leq \gamma_{M} (1+\xi)$  for  $\gamma_M>0$ and all  $\xi \in \mathbb R_+$, with $j=w,l$.
\item Initial conditions $b_0 \in W^{1,4}(\Omega)\cap L^\infty(\Omega)$,   $b_{s,0} \in L^\infty(\Omega)$,  $p_{\alpha0}, c_{0}\in H^1(\Omega)\cap L^\infty(\Omega)$,  are non-negative,   $\mathbf{f} \in C^{0, \sigma}([0,T]; L^2(\partial \Omega))^3$,  $P \in C^{0,\sigma}([0,T]; L^2(\Gamma_{\cI}))$ for some $\sigma \in [1/18, 1]$, and $\alpha= e,E,d$.
\end{enumerate}
\end{assumption}
\begin{remark}
Assumption~\ref{assumptions},  required for the proof of  the well-posedness of  model~\eqref{EQST}--\eqref{IC} and  for the derivation of the macroscopic equations,  are biologically relevant,~\cite{PS_1_2016}, and typical reaction functions considered in Section~\ref{micro_model} will satisfy those assumptions.
 For simplicity and biological relevance, we consider constant diffusion coefficients,  although, more general uniformly elliptic diffusion matrices  can also be considered.
 Assumptions~\ref{assumptions}.2~and~\ref{assumptions}.3 on the reaction terms $R_e$ and $R_d$  and functions $J_\alpha$, for $\alpha = e, E,c$, in the boundary conditions,   ensure the nonnegativity of solutions of the systems of reaction-diffusion  equations and ODE. The sublinearity assumptions on $R_d$, $J_E$, $J_c$  are required to prove a priori estimates for solutions of~\eqref{EQST}--\eqref{IC}.
 The boundedness of $ J_e$ is motivated by the underlying biological processes, however for the analysis it is sufficient to assume the sublinearity of $J_e$ with respect to  $\xi$ and $\eta$. Also it is possible to consider that $J_E$ is sublinear with respect to both $\xi$ and $\eta$. The nonnegativity and monotonicity of $r_l$, for $l=b,s$, and the boundedness and nonnegativity of $Q_b$ in Assumption~\ref{assumptions}.4 are biologically relevant and also essential for the proof of the nonnegativity of solutions and the derivation of a priori estimates. Assumptions~\ref{assumptions}.5 are the standard assumptions on the elasticity tensor.  The assumption that $\mathbb E_{M_l}$ and $\mathbb E_{M_w}$ are sublinear with respect to the cross-links density, used in the derivation of  a priori estimates, is biologically relevant and motivated by experimental results~\cite{PS_2_2016, ZMSR}. The regularity assumptions on the initial data in Assumption~\ref{assumptions}.6 are required for the proof of the boundedness of solutions of model~\eqref{EQST}--\eqref{IC}, whereas the regularity $b_0 \in W^{1,4}(\Omega)$ is needed to show equicontinuity with respect to the time variable of $b^\ve$  in~$L^\infty((0,T)\times\Omega_M^\ve)$ and then of~$\bu^\ve$ in $L^2(0,T;\cW(\Omega))$.
\end{remark}

\begin{definition}
A weak solution of~\eqref{EQST}--\eqref{IC} are functions $(p_\alpha^\ve, c, b^\ve, b^\ve_s, \bu^\ve)$,   such that
$\bu^\ve \in L^2(0,T;\cW(\Omega))$,
  $b^\ve_s \in H^1(0,T; L^2(\Omega^\ve_M))$, $p^\ve_\alpha, c^\ve, b^\ve \in L^2(0,T; H^1(\Omega^\ve_M))$, with  $\partial_t p^\ve_\alpha, \partial_t c^\ve, \partial_t b^\ve \in L^2(0,T; H^1(\Omega^\ve_M)^\prime)$, for $\alpha=e,E, d$,    and
satisfy the equations
\begin{eqnarray}\label{weak_sol_n1}
\begin{aligned}
\langle \partial_t p^\ve_E, \phi_E \rangle_{(H^1)^\prime, H^1}  + \langle D_E \nabla p^\ve_E, \nabla \phi_E \rangle_{\Omega^\ve_{M}}\!&=\langle J_E(p_e^\ve,  p_d^\ve, b^\ve) - \gamma_E p^\ve_E, \phi_E \rangle_{ \Gamma_{\cI}},\\
\langle \partial_t p^\ve_e, \phi_e \rangle_{(H^1)^\prime, H^1} + \langle D_e \nabla p^\ve_e, \nabla \phi_e \rangle_{\Omega^\ve_{M}}
\!&= -\langle  R_{e}(p^\ve_e, p^\ve_E), \phi_e \rangle_{\Omega^\ve_{M}}
\; + \langle J_e(p^\ve_e, p^\ve_d, b^\ve) - \gamma_e p^\ve_e, \phi_e \rangle_{\Gamma_{\cI}},\\
\langle \partial_t p^\ve_d, \phi_d \rangle_{(H^1)^\prime, H^1}  + \langle D_d \nabla p^\ve_d, \nabla \phi_d \rangle_{\Omega^\ve_{M}}
\!& =   \langle R_{e}(p^\ve_e, p^\ve_E), \phi_d \rangle_{\Omega^\ve_{M}} - \langle \gamma_d p_d^\ve, \phi_d \rangle_{\Gamma_{\cI}}\\
& \quad + 2\langle  R_s(b^\ve, b_s^\ve, \be(\bu^\ve)) - R_{d}(p_d^\ve, c^\ve), \phi_d \rangle_{\Omega^\ve_{M}}, \\
\langle \partial_t c^\ve, \phi_c \rangle_{(H^1)^\prime, H^1}  + \langle D_c \nabla c^\ve, \nabla \phi_c \rangle_{\Omega^\ve_{M}}
\!& =   \langle J_c(\bu^\ve) -\gamma_c c^\ve, \phi_c \rangle_{\Gamma_{\cI}}\\
& \quad +  \langle  R_s(b^\ve, b^\ve_s,\be(\bu^\ve)) -R_{d}(p_d^\ve, c^\ve), \phi_c \rangle_{\Omega^\ve_{M}}, \\
\langle \partial_t b^\ve, \phi_b \rangle_{(H^1)^\prime, H^1}  + \langle D_b \nabla b^\ve, \nabla \phi_b \rangle_{\Omega^\ve_{M}}\! & =  -\langle d_b b^\ve  , \phi_b \rangle_{\Omega^\ve_{M}}\\
& \quad + \langle  R_{d}(p_d^\ve,  c^\ve) - R_b(b^\ve, \be(\bu^\ve)) , \phi_b \rangle_{\Omega^\ve_{M}} ,
\end{aligned}
\end{eqnarray}
for a.a.~$t \in (0, T)$ and all  $\phi_\alpha\in L^2(0,T; H^1(\Omega_{M}^\ve))$,     where $\alpha=e,E, d,c, b$,
\begin{equation}\label{weak_sol_bs}
 \partial_t b^\ve_s =  R_{b}(b^\ve, \be(\bu^\ve)) - R_s(b^\ve, b^\ve_s, \be(\bu^\ve)) - d_s b^\ve_s  \quad \text{a.e. in }  \Omega^\ve_{M,T},
\end{equation}
\begin{equation}\label{weak_sol_u}
\langle \bbE^\ve(x,b^\ve)\be(\bu^\ve), \be(\boldsymbol{\psi}) \rangle_{\Omega} = -\langle P \nu, \boldsymbol{\psi} \rangle_{\Gamma_{\cI}} +\langle \bff, \boldsymbol{\psi} \rangle_{\Gamma_{\cE}},
\end{equation}
for a.a.~$t \in (0, T)$ and all $\boldsymbol{\psi} \in L^2(0,T; \cW(\Omega))$. Furthermore, $p_\alpha^\ve, c, b^\ve, b^\ve_s$ satisfy the  initial conditions \eqref{IC} in $L^2(\Omega_M^\ve)$,  where  $\alpha=e,E,d$.
\end{definition}
Here $\langle \psi, \phi \rangle_{(H^1)^\prime, H^1}$ denotes the dual product for $\phi \in  L^2(0,T; H^1(\Omega_{M}^\ve))$, $\psi \in  L^2(0,T; H^1(\Omega_{M}^\ve)^\prime)$ and $\langle \phi, \psi \rangle_{\Omega} = \int_\Omega \phi \psi dx$ for $\phi \in L^q(\Omega)$,  $\psi \in L^p(\Omega)$ with $1/p+ 1/q =1$, and similar for $\langle \cdot, \cdot\rangle_{\Omega_M^\ve}$, $\langle \cdot, \cdot\rangle_{\Omega_T}$,  $\langle \cdot, \cdot\rangle_{\Omega^\ve_{M,T}}$, $\langle \cdot, \cdot\rangle_{\Gamma_\cI}$, and  $\langle \cdot, \cdot\rangle_{\Gamma_{\cI, T}}$, where $\Omega_T = (0,T)\times \Omega$, $\Omega^\ve_{M,T}= (0,T)\times \Omega^\ve_M$, and $\Gamma_{\cI, T} = (0,T)\times \Gamma_{\cI}$.

Since $\cW(\Omega) \cap \mathcal R(\Omega) = \{ {\bf 0}\}$, where $\mathcal R(\Omega)$ denotes the space of rigid displacements of $\Omega$, by Korn's second inequality~\cite{CC} we have
\begin{equation}\label{Korn_in}
 \| \bu\|_{\cW(\Omega)} \leq  C \|\be(\bu)\|_{L^2(\Omega)} \qquad \text{ for all }\,  \bu \in \cW(\Omega).
\end{equation}

In the derivation of the a priori estimates we will use the following trace and Gagliardo-Nirenberg inequalities, see e.g.~\cite{Brezis, Galdi},
\begin{equation}\label{trace_GN_ineq}
 \begin{aligned}
&  \|u\|_{L^2(\partial \Omega)} \leq C \big[\|u\|_{L^2(\Omega)} +  \|u\|_{L^2(\Omega)}^{\frac 12}\|\nabla u\|^{\frac 12}_{L^2(\Omega)}\big]  \leq C \|u\|_{L^2(\Omega)} + \delta \|\nabla u\|_{L^2(\Omega)} , \\
&  \|u\|_{L^r(\partial \Omega)} \leq C \big[\|u\|_{L^2(\Omega)}^{(1-\lambda)}\|\nabla u\|^{\lambda}_{L^2(\Omega)} +  \|u\|_{L^2(\Omega)}^{(1-\frac 1r)(1-\lambda)}\|\nabla u\|^{\frac 1r + \lambda(1-\frac 1 r)}_{L^2(\Omega)}\big]   , \\
   & \|u\|_{L^q(\Omega)} \leq C  \big[ \|\nabla u\|_{L^2(\Omega)}^{\lambda_1} \|u\|_{L^2(\Omega)}^{1-\lambda_1} + \|u\|_{L^1(\Omega)}\big], \\
&    \|u\|_{L^q(\Omega)} \leq C \big[  \|\nabla u\|_{L^2(\Omega)}^{\lambda_2} \|u\|_{L^1(\Omega)}^{1-\lambda_2} + \|u\|_{L^1(\Omega)}\big],
 \end{aligned}
\end{equation}
for $u \in H^1(\Omega)$ and $\delta>0$, where $\lambda= d(r-2)/(2(r-1))$ for $2\leq r \leq 2(d-1)/(d-2)$,  $\lambda_1 = d (1/2 - 1/q)$ for $2\leq q \leq 2d/(d-2)$ and $\lambda_2 = 2 d(q-1)/(q(d+2))$ for $1\leq q \leq  2d/(d-2)$ for $d> 2$, and $d={\rm dim}(\Omega)$.

To obtain estimates uniform in~$\varepsilon$, we will also use  the extension results, ensured by the assumptions on  the  microstructure of $\Omega^\ve_M$ and the regularity of rotation matrix $R$.
\begin{lemma}\label{extension_local} \cite{Ptashnyk_2017}
For $x_j^\ve \in \Omega_W$,   and $w \in W^{1,q}(\tilde Y^\ast_{x_j^\ve})$, with $q\in [1, +\infty)$ and $j \in \Xi^\ve$,  there exists an extension  $\tilde w \in  W^{1,q}(\tilde Y_{x_j^\ve})$ from $\tilde Y^\ast_{x_j^\ve}$ into $\tilde Y_{x_j^\ve}$ such that
\begin{equation}\label{extension_local_1}
\|\tilde w\|_{L^q(\tilde Y_{x_j^\ve})} \leq \mu \|w\|_{L^q( \tilde Y_{x_j^\ve}^\ast)}, \qquad \|\nabla \tilde w\|_{L^q(\tilde Y_{x_j^\ve})} \leq \mu \|\nabla w\|_{L^q(\tilde Y_{x_j^\ve}^\ast)},
\end{equation}
where $\mu$ depends on $Y_1$, $Y_0$, and $R$ and is independent of $\ve$ and $j \in \Xi^\ve$.

For $w \in W^{1,q}(\Omega^{\ve}_{MW})$  we have an extension $\tilde w \in W^{1,q}(\Omega_W)$ from $\Omega_{MW}^{\ve}$ into $\Omega_W$,
\begin{equation}\label{extension_local_2}
\|\tilde w\|_{L^q(\Omega_W)} \leq \mu \|w\|_{L^q(\Omega^{\ve}_{MW})},
\qquad \|\nabla \tilde w\|_{L^q(\Omega_W)} \leq \mu \|\nabla w\|_{L^q(\Omega^{\ve}_{MW})}\; ,
\end{equation}
where  $\mu$ depends on $Y_1$, $Y_0$, and $R$,  and is independent of $\ve$, and $\Omega_{MW}^\ve =\Omega^{\ve}_M\cap \Omega_W$.
\end{lemma}
Using the extension for $b^\ve$ we can define the extension for $b^\ve_s$ from $\Omega^\ve_M$ into $\Omega$ as a solution of~\eqref{weak_sol_bs}, where $b^\ve$ is replaced by its extension.

The properties of the extension~\eqref{extension_local_2} and inequalities~\eqref{trace_GN_ineq} imply,  for $w^\ve \in H^1(\Omega^\ve_M)$,
\begin{equation}\label{trace_GN_ineq_extension}
 \begin{aligned}
 & \|w^\ve\|_{L^2(\partial \Omega)} \! \leq \! C \|\tilde w^\ve\|_{L^2(\Omega)}\! + \delta \|\nabla \tilde w^\ve\|_{L^2(\Omega)}
\!\leq \! C \|w^\ve\|_{L^2(\Omega^\ve_M)}\!\! + \delta \|\nabla  w^\ve\|_{L^2(\Omega^\ve_M)}, \\
 & \|w^\ve\|_{L^r(\partial \Omega)} \!\leq \! C \big[\|\tilde w^\ve\|_{L^2(\Omega)}^{(1-\lambda)}\|\nabla \tilde w^\ve\|^{\lambda}_{L^2(\Omega)}\! +  \|\tilde w^\ve\|_{L^2(\Omega)}^{(1-\frac 1r)(1-\lambda)}\|\nabla \tilde w^\ve\|^{\frac 1r + \lambda(1-\frac 1 r)}_{L^2(\Omega)}\big] \\
 & \qquad \quad \leq \! C \big[\|w^\ve\|_{L^2(\Omega_M^\ve)}^{(1-\lambda)}\|\nabla  w^\ve\|^{\lambda}_{L^2(\Omega_M^\ve)} \!+  \| w^\ve\|_{L^2(\Omega_M^\ve)}^{(1-\frac 1r)(1-\lambda)}\|\nabla  w^\ve\|^{\frac 1r + \lambda(1-\frac 1 r)}_{L^2(\Omega_M^\ve)}\big], \\
 & \|w^\ve\|_{L^q(\Omega^\ve_M)} \! \leq \|\tilde w^\ve\|_{L^q(\Omega)} \!\leq C   \big[\|\nabla \tilde w^\ve\|_{L^2(\Omega)}^{\lambda_1} \|\tilde w^\ve\|_{L^2(\Omega)}^{1-\lambda_1} + \|\tilde w^\ve\|_{L^1(\Omega)}\big]\\
 & \qquad \quad  \leq \! C  \big[ \|\nabla w^\ve\|_{L^2(\Omega^\ve_M)}^{\lambda_1} \|w^\ve\|_{L^2(\Omega^\ve_M)}^{1-\lambda_1} + \|w^\ve\|_{L^1(\Omega^\ve_M)} \big], \\
 &  \|w^\ve\|_{L^q(\Omega_M^\ve)} \! \leq \|\tilde w^\ve\|_{L^q(\Omega)} \!\leq C  \big[ \|\nabla \tilde w^\ve\|_{L^2(\Omega)}^{\lambda_2} \|\tilde w^\ve\|_{L^1(\Omega)}^{1-\lambda_2} + \|\tilde w^\ve\|_{L^1(\Omega)}\big]\\
& \qquad \quad \leq \! C  \big[ \|\nabla  w^\ve\|_{L^2(\Omega^\ve_M)}^{\lambda_2} \| w^\ve\|_{L^1(\Omega^\ve_M)}^{1-\lambda_2} + \|w^\ve\|_{L^1(\Omega^\ve_M)}\big],
 \end{aligned}
\end{equation}
for any fixed $\delta >0$ and the constant $C$ depends on $\Omega$, $Y_1$, $Y_0$, $R$, and is independent of~$\ve$. 

\begin{lemma}\label{estim:apriori}
Under Assumption~\ref{assumptions}, solutions of  microscopic model~\eqref{EQST}--\eqref{IC} satisfy
\begin{equation}\label{estim_apriori}
\begin{aligned}
&  \| \bu^\ve\|_{L^\infty(0,T;\cW(\Omega))}  \leq C,
\quad \|\theta_h \bu^\ve - \bu^\ve\|_{L^2(0,T-h; \cW(\Omega))} \leq C \big(h^{\frac \varsigma{2(\varsigma +1)}}+  h^{\frac{1-9 \zeta} {1+ \zeta}} \big), \\
 &\|p^\ve_\alpha\|_{L^\infty(\Omega_{M, T}^\ve)} + \|\nabla p^\ve_\alpha\|_{L^2( \Omega_{M,T}^\ve)}   + \|w^\ve\|_{L^\infty(\Omega_{M, T}^\ve)} + \|\nabla w^\ve\|_{L^2( \Omega_{M,T}^\ve)}    \leq C, \\
  &\| \theta_h p^\ve_\alpha - p^\ve_\alpha\|_{L^2( \Omega_{M, T-h}^\ve)} +\| \theta_h w^\ve - w^\ve\|_{L^2(\Omega_{M, T-h}^\ve)}  \leq C h^{\frac 1 4} ,\\
   &\|b^\ve_s\|_{L^\infty(0,T; L^q(\Omega_M^\ve))} + \|b^\ve_s\|_{L^\infty(\Omega_{M, T}^\ve)} +  \|\partial_t  b^\ve_s\|_{L^2(\Omega_{M, T}^\ve)}  \leq C,
 \end{aligned}
\end{equation}
where  $w^\ve=c^\ve,b^\ve$ and  $\alpha=e,E,d$,  and  $\theta_h v(t,x) = v(t+h, x)$ for $t \in [0, T-h]$, $x\in \Omega_M^\ve$,  $0<h\leq 1$,  any  $\varsigma \in (0, 1/9)$ and   $q \in [2, \infty)$, and the  constant $C$ is independent of $\ve$.
\end{lemma}
\begin{proof}
The proof follows similar ideas as in~\cite{PS_1_2016}.
The  first estimate for $\bu^\ve$ follows directly by considering $\bu^\ve$ as a test function in \eqref{weak_sol_u} and using the ellipticity properties of~$\bbE$,  Assumption~\ref{assumptions}.5, the second Korn inequality, see~\eqref{Korn_in} and e.g.~\cite{Ciarlet}, and Assumption~\ref{assumptions}.6 on $\bff$ and $P$.

 Using the method of positive invariant regions, e.g.~\cite{Redlinger, Smoller}, and assumptions on the non-linear functions  yields that $p^\ve_\alpha, c^\ve, b^\ve, b_s^\ve$ are non-negative,  where $\alpha=E,e,d$.

Considering $p_e^\ve$ and $p_E^\ve$ as test functions in the equations for  $p_e^\ve$ and $p_E^\ve$, respectively, and using the non-negativity of $p_\alpha^\ve$ and $b^\ve$, with $\alpha = e,E,d$, sublinearity of $J_E$, boundedness of $J_e$,  and non-negativity of $R_e$, see Assumptions~\ref{assumptions}.2~and~\ref{assumptions}.3,   yield
$$
\begin{aligned}
 \|p_E^\ve(\tau)\|^2_{2} \! + \|p_e^\ve(\tau)\|^2_{2} \!   +  2\! \int_0^\tau \!\!\! d\big[\|\nabla p_E^\ve(t)\|^2_{2} \! + \|\nabla p_e^\ve(t)\|^2_{2}\big]
 + \tilde \gamma \big[ \|p_E^\ve(t) \|^2_{L^2(\Gamma_\cI)} +  \|p_e^\ve(t) \|^2_{L^2(\Gamma_\cI)}\big]  dt
 \\
  \leq \|p_{E0}\|^2_{2} + \|p_{e0}\|^2_{2}  +  \tilde \beta  \!\int_0^\tau\!\! \!\big[1+ \|p_E^\ve(t)\|^2_{L^2(\Gamma_\cI)} + \|p_e^\ve(t)\|^2_{L^2(\Gamma_\cI)}\big] dt,
\end{aligned}
$$
for $\tau \in (0, T]$, where $\tilde \gamma = 2\min\{ \gamma_e, \gamma_E\}$, $\tilde \beta = 2\max\{\beta_e, \beta_E\}$,  and we use the short notation $\|\cdot \|_{q} = \|\cdot \|_{L^q(\Omega^\ve_M)}$.
Using the regularity of initial data, see Assumption~\ref{assumptions}.6, and the trace inequality~\eqref{trace_GN_ineq_extension},  applied to an extension of $p_e^\ve$ and $p_E^\ve$ from $\Omega^\ve_M$ to $\Omega$,  and the Gr\"onwall inequalities, we obtain
the estimates for   $p_e^\ve$ and $p_E^\ve$ in $L^\infty(0,T; L^2(\Omega_M^\ve))$ and $L^2(0,T; H^1(\Omega_M^\ve))$.
Considering $(p_e^\ve-M_e)_+$ and $(p_E^\ve-M_E)_+$, with some  $M_e, M_E >0$,   as test functions in the equations for  $p_e^\ve$ and $p_E^\ve$, respectively,  we obtain
$$
\begin{aligned}
\|(p_e^\ve(\tau)-M_e)_+\|^2_{L^2(\Omega^\ve_M)} + 2d  \|\nabla (p_e^\ve-M_e)_+\|^2_{L^2(\Omega_{M,\tau}^\ve)} +2 \gamma_e \langle p_e^\ve, (p_e^\ve-M_e )_+ \rangle_{\Gamma_{\cI, \tau}}
\\ \leq \|(p_{e0}-M_e)_+\|^2_{L^2(\Omega^\ve_M)} +  2\beta_e \langle 1, (p_e^\ve-M_e)_+ \rangle_{\Gamma_{\cI, \tau}},
\\
 \|(p_E^\ve(\tau) - M_E)_+\|^2_{L^2(\Omega^\ve_M)}   + 2d  \|\nabla (p_E^\ve-M_E)_+\|^2_{L^2(\Omega_{M,\tau}^\ve)}
 +   2 \gamma_E\langle p_E^\ve , (p_E^\ve - M_E)_+ \rangle_{\Gamma_{\cI, \tau}}
 \\
 \leq \|(p_{E0}-M_E)_+\|^2_{L^2(\Omega^\ve_M)}+ 2\beta_E \langle (1+ p_e^\ve), (p_E^\ve - M_E)_+ \rangle_{\Gamma_{\cI, \tau}}.
\end{aligned}
$$
Choosing  $M_e \geq \max\{\beta_e/\gamma_e, \|p_{e0}\|_{L^\infty(\Omega)}\}$,  $M_E \geq \max\{\beta_E(1+ M_e)/\gamma_E, \|p_{E0}\|_{L^\infty(\Omega)}\}$,  yields $p_e^\ve(t,x) \leq M_e$ and $p_E^\ve(t,x) \leq M_E$, for $(t,x)  \in (0,T)\times \Omega_M^\ve$.

Considering $p_d^\ve$, $c^\ve$, and $b^\ve$ as test functions in the corresponding equations in~\eqref{weak_sol_n1} and using  boundedness of $p_e$ and $p_E$ and  Assumptions~\ref{assumptions}.2~and~\ref{assumptions}.4  on $R_e$, $R_d$, $r_l$ and $Q_l$, for $l=b,s$,  imply
\begin{equation}\label{estim_L2}
\begin{aligned}
 &  \|p_d^\ve (\tau)\|^2_{2} + \|c^\ve (\tau)\|^2_{2} + \|b^\ve (\tau)\|^2_{2} +  2d \! \int_0^\tau \!\!\Big[\|\nabla p_d^\ve\|^2_{2} +   \|\nabla c^\ve\|^2_{2}   + \|\nabla b^\ve\|^2_{2}\Big] dt
\\
& \qquad \leq  \|p_{d0} \|^2_{2} + \|c_0\|^2_{2} +  \|b_0\|^2_{2} + 2\!\int_0^\tau\!\! \big\langle |J_c(\bu^\ve)|, |c^\ve| \big\rangle_{\Gamma_\cI} dt
\\
&\;  +   C \!\int_0^\tau\!\!\Big[1+ \|p_d^\ve\|^2_{2}
 +   \|b^\ve\|^2_{2}+ \big(1+\|\be(\bu^\ve)\|_{2}\big)\|b_s^\ve\|_{4}\big[\|p_d^\ve\|_{4} + \|c^\ve\|_{4}\big]  \Big]dt,
 \end{aligned}
\end{equation}
where the constant  $C= 2\max\{ 2\beta_s L_{r_s}, M_{R_e},\beta_d\}$, with $L_{r_s}$ denoting the Lipschitz constant of $r_s$ and $M_{R_e} = \!\! \max\limits_{0\leq \xi\leq M_e, 0 \leq \eta\leq M_E}\!\! R_e(\xi,\eta)$.
Using Assumption~\ref{assumptions}.3 on $J_c$ and  applying the trace inequalities in \eqref{trace_GN_ineq} and~\eqref{trace_GN_ineq_extension}, yield
\begin{equation}\label{estim_L2_bt}
\begin{aligned}
 \int_0^\tau\!\! \big\langle |J_c(\bu^\ve)|, |c^\ve| \big\rangle_{\Gamma_\cI} dt &\leq
 C_{\delta_1} \big(1+ \|\bu^\ve\|^2_{L^2((0,\tau)\times \Gamma_\cI)}\big) + \delta_1 \|c^\ve\|^2_{L^2((0,\tau)\times \Gamma_\cI)}
 \\
 &\leq C \big(1+ \|\bu^\ve\|^2_{L^2(0,\tau; \cW(\Omega))}\big) + \delta \|c^\ve\|^2_{L^2(0,\tau; H^1(\Omega_M^\ve))},
\end{aligned}
\end{equation}
for any fixed $\delta >0$ and the constant $C>0$ depends on $\delta, \beta_c$ and the constant in the trace  inequalities, and is independent of $\ve$.

Testing equation~\eqref{weak_sol_bs} with $b_s^\ve |b_s^\ve|^{q-2}$, for  $q\geq 2$, and using Assumption~\ref{assumptions}.4 on  $r_l$ and $Q_l$, for $l=b,s$,  yield
\begin{equation*}
\begin{aligned}
&\|b^\ve_s (\tau)\|^q_{q} + q d_s\int_0^\tau \!\!\! \|b^\ve_s(t)\|^q_{q} dt + q \rho_s\int_0^\tau \!\!\!\int_{\Omega_M^\ve} \!\! \! Q_s(b^\ve, \be(\bu^\ve)) |b^\ve_s|^q dxdt\\
& \quad \leq
 \|b_{s0}\|^q_{q}  + \beta_b L_{r_b} q \!\int_0^\tau\!\!\! \|b^\ve (t)\|_{q}\|b^\ve_s (t)\|^{q-1}_{q}  dt \leq
  \frac{\beta_b L_{r_b}}{\delta_s^{q-1}} \Big[\int_0^\tau\!\!\! \|b^\ve (t)\|^{\frac q 2}_{q}  dt\Big]^2\\
&\quad\qquad    +  \|b_{s0}\|^q_{q}+  (q-2)\beta_b L_{r_b}\delta_s\! \int_0^\tau\!\! \|b^\ve_s (t)\|^q_{q} dt
 + \beta_b L_{r_b}\delta_s\!\! \!\sup\limits_{t \in(0,\tau)} \|b^\ve_s (t)\|^q_{q},
\end{aligned}
\end{equation*}
for  $\tau\in (0,T]$ and any fixed $\delta >0$, where $L_{r_b}$ is the Lipschitz constant of $r_b$ and we used that, by applying the H\"older and Young inequalities, the following holds
$$
\begin{aligned}
& q \!\int_0^\tau\!\!\|b^\ve \|_{q}\|b^\ve_s \|^{q-1}_{q}  dt \leq
q\Big[\!\int_0^\tau\!\! \|b^\ve \|_{q}^{\frac q2} dt \Big]^{\frac 2 q}\Big[ \int_0^\tau \!\!\|b^\ve_s \|^{\frac{q(q-1)}{q-2}}_{q} \! dt \Big]^{\frac{q-2} q}
\\
& \quad \leq
(q-1) \delta_s  \Big[\int_0^\tau\!\|b^\ve_s \|^q_{q} \,  dt \Big]^{\frac{q-2} {q-1}} \!\sup\limits_{(0,\tau)}\|b^\ve_s \|^{\frac q{q-1}}_{q}
 + \frac 1 {\delta_s^{q-1}} \Big[\int_0^\tau\!\! \|b^\ve \|_{q}^{\frac q2} dt \Big]^2
\\
& \quad \leq
 \frac 1 {\delta_s^{q-1}} \Big[\int_0^\tau\!\! \|b^\ve \|_{q}^{\frac q2} dt \Big]^2
+  (q-2) \delta_s \!\int_0^\tau \!\!\|b^\ve_s \|^q_{q} \,  dt + \delta_s \!\sup\limits_{(0,\tau)}\|b^\ve_s\|^q_{q}.
\end{aligned}
$$
 Thus, using  that $d_s>0$ and choosing $\delta_s =\min\{ d_s, 1\}/(2\beta_b L_{r_b})$, imply
\begin{equation}\label{bs_Lp2}
\begin{aligned}
 \sup\limits_{(0,\tau)}\|b^\ve_s \|^q_{L^q} + q\! \!\int_0^\tau\! \!\|b^\ve_s\|^q_{L^q} dt +  q \!\int_0^\tau\!\!\!\int_{\Omega_M^\ve} \! \!  Q_s(b^\ve, \be(\bu^\ve)) |b^\ve_s|^q dx dt
\leq C_s\Big[1+ \frac 1{\delta_s^{q}}\Big[ \int_0^\tau \!\!\|b^\ve\|^{\frac q 2}_{L^q} dt \Big]^2 \Big],
\end{aligned}
\end{equation}
where $C_s>0$ depends on $d_s, \beta_b, L_{r_b}$, and is independent of $\ve$.
Considering estimate~\eqref{bs_Lp2} with $q=4$ and applying the Gagliardo-Nirenberg inequality~\eqref{trace_GN_ineq_extension} yield
\begin{equation}\label{estim_GN_22}
\begin{aligned}
& \int_0^\tau \!\! \|\be(\bu^\ve)\|_{2}\|b_s^\ve\|_{4}\big[\|p_d^\ve\|_{4} + \|c^\ve\|_{4}\big] dt
\leq C \Big[1+ \Big[ \int_0^\tau \!\!\|b^\ve\|^{2}_{4} dt \Big]^2 \Big]^{\frac 14} \int_0^\tau\!\!\big(\|p_d^\ve\|_{4} + \|c^\ve\|_{4}\big) dt \\
&\quad  \leq
C \Big[1+ \int_0^\tau \!\!\big(\|b^\ve\|^{ 2}_{4} + \|p_d^\ve\|^2_{4} + \|c^\ve\|^2_{4}\big) dt\Big]
\\
&\quad  \leq \delta  \int_0^\tau \!\! \big(\|\nabla p_d^\ve\|^2_{2} +   \|\nabla c^\ve\|^2_{2}  + \|\nabla  b^\ve\|^2_{2}\big)dt  +
C_\delta\Big[ 1 + \int_0^\tau \!\! \big(\|p_d^\ve\|^2_{2} +   \| c^\ve\|^2_{2}   + \| b^\ve\|^2_{2}\big) dt\Big],
\end{aligned}
\end{equation}
for any fixed $\delta >0$, where $C_\delta$ depends on $\delta$, $\delta_s$, $C_s$, on the constant in the Gagliardo-Nirenberg inequality~\eqref{trace_GN_ineq_extension}, and on $\|\be(\bu^\ve)\|_{L^\infty(0,T; L^2(\Omega))}$, which is bounded in terms of~$\bff$ and~$P$,  and is independent of~$\ve$.
Using estimates \eqref{estim_L2_bt} and \eqref{estim_GN_22} in  \eqref{estim_L2}, together with the regularity of initial  conditions and the Gr\"onwall inequality,  yields the estimates for $p_d^\ve$, $c^\ve$ and $b^\ve$ in $L^\infty(0,T; L^2(\Omega_M^\ve))$ and $L^2(0,T; H^1(\Omega_M^\ve))$ and for $b_s^\ve$ in $L^\infty(0,T; L^4(\Omega_M^\ve))$.

Taking  $p_d^\ve |p_d^\ve|^{q-2}$,  $c^\ve |c^\ve|^{q-2}$, and  $b^\ve |b^\ve|^{q-2}$ as test functions in the corresponding equations in~\eqref{weak_sol_n1} yields
\begin{equation}\label{estim_Lq_11}
\begin{aligned}
 &  \|p_d^\ve (\tau)\|^q_{q} + \|c^\ve (\tau)\|^q_{q} +  \|b^\ve (\tau)\|^q_{q} + \frac{4(q-1)}q d\int_0^\tau \!\!\Big[\|\nabla |p_d^\ve|^{\frac q 2}\|^2_{2} +   \|\nabla |c^\ve|^{\frac q 2}\|^2_{2} + \|\nabla |b^\ve|^{\frac q 2}\|^2_{2}\Big] dt \\
 & \leq \|p_{d0}\|^q_{q} + \|c_0 \|^q_{q} +  \|b_0\|^q_{q} + q\!\int_0^\tau\!\! \big\langle |J_c(\bu^\ve)|, |c^\ve|^{q-1} \big\rangle_{\Gamma_\cI} dt  \\
 & \quad  +  C \!\int_0^\tau\!\! \Big[1+ q\|p_d^\ve\|^q_{q} +(q-1)   \|b^\ve\|^q_{q} +
q\big(1+\|\be(\bu^\ve)\|_{2}\big)\|b_s^\ve\|_{{2q}}\big(\|p_d^\ve\|_{{2q}}^{q-1} + \|c^\ve\|_{{2q}}^{q-1}\big)  \Big]dt,
 \end{aligned}
\end{equation}
where $C= 2\max\{\beta_s L_{r_s}, M_{R_e},\beta_d\}$.
 Using the trace inequality~\eqref{trace_GN_ineq_extension} and Assumption~\ref{assumptions}.3 on $J_c$,   the boundary integral is estimated as
\begin{equation}\label{estim_LqGamma}
\begin{aligned}
& q\!\int_0^\tau\! \!\big\langle |J_c(\bu^\ve)|, |c^\ve|^{q-1} \big\rangle_{\Gamma_\cI}   dt\leq  C
 q\!\int_0^\tau\!\!\|\bu^\ve\|_{L^4(\Gamma_{\cI})} \||c^\ve|^{\frac q2}\|^{2\frac{q-1}{q}}_{L^{\frac 83}(\Gamma_{\cI})}\! dt \\
&  \leq C \!\int_0^\tau\!\! \Big[
(q-1)\big[  \|\nabla |c^\ve|^{\frac q 2}\|^{\frac 65}_{2}\||c^\ve|^{\frac q 2}\|^{\frac 4 5}_{2} + \|\nabla |c^\ve|^{\frac q 2}\|^{\frac 3 2}_{2}\||c^\ve|^{\frac q 2}\|^{\frac 1 2}_{2}\big]
\\
& \; + \|\bu^\ve\|^q_{\cW(\Omega)}\Big] dt  \leq   C\!\int_0^\tau\!\!\Big[\|\bu^\ve\|_{\cW(\Omega)}^q  + \delta \||\nabla c^\ve|^{\frac q 2}\|^{2}_{2} +
(1 + q^4)\||c^\ve|^{\frac q 2}\|^2_{2} \Big] dt.
\end{aligned}
\end{equation}
Applying the Gagliardo-Nirenberg inequality~\eqref{trace_GN_ineq_extension} to the last term in~\eqref{estim_LqGamma} we obtain
$$
\begin{aligned}
(1+q^4)\||c^\ve|^{\frac q 2}\|^2_{2}& \leq
C q^4  \big[\|\nabla |c^\ve|^{\frac q 2}\|^{\frac 6 5}_{2} \| |c^\ve|^{\frac q 2}\|^{\frac 4 5}_{L^1}\! + \| |c^\ve|^{\frac q 2}\|^2_{L^1} \big]
 \leq \delta  \|\nabla |c^\ve|^{\frac q 2}\|^2_{2}  + C_\delta  q^{10}  \| |c^\ve|^{\frac q 2}\|^2_{L^1},
\end{aligned}
$$
for any fixed $\delta >0$. Similar calculations imply
\begin{equation}\label{estim_Lq_22}
\begin{aligned}
& q\big[\|p_d^\ve\|^q_{q} + \|b^\ve\|^q_{q}\big] =
q\big[\||p_d^\ve|^{\frac q 2}\|^2_{2} + \||b^\ve|^{\frac q 2}\|^2_{2}\big]
\\
& \quad \leq \delta \big[ \|\nabla |p_d^\ve|^{\frac q 2}\|^2_{2} +\|\nabla |b^\ve|^{\frac q 2}\|^2_{2}\big]
 + C_\delta  q^{\frac 52} \big[ \| |p_d^\ve|^{\frac q 2}\|^2_{L^1}+ \| |b^\ve|^{\frac q 2}\|^2_{L^1}\big],
\end{aligned}
\end{equation}
where  $C_\delta>0$ depends on $\delta$ and the constant in the Gagliardo-Nirenberg inequality~\eqref{trace_GN_ineq_extension} and is independent of $\ve$.
Using the estimate for $b^\ve_s$ in~\eqref{bs_Lp2} and the Gagliardo-Nirenberg inequality gives
\begin{equation*}
\begin{aligned}
& q\!\int_0^\tau \!\!\!\big(1+\|\be(\bu^\ve)\|_{2}\big)\|b_s^\ve\|_{{2q}}\big[\|p_d^\ve\|^{q-1}_{{2q}} + \|c^\ve\|^{q-1}_{{2q}}\big] dt\\
&  \leq q\big[1+ \|\be(\bu^\ve)\|_{L^\infty(0,T;L^2)}\big]
 \! \int_0^\tau\! \!\Big[ C_s+ \frac {C_s}{\delta_s^{2q}}\Big[\int_0^\tau \!\!\! \!\!\|b^\ve\|_{{2q}}^{q}dt \Big]^{2}\Big]^{\frac{1}{2q}} \big[\|p_d^\ve\|^{q-1}_{{2q}} \! + \|c^\ve\|^{q-1}_{{2q}}\big] dt\\
&   \leq \frac {C_u}{\delta_s} \Big[C_s^{\frac 12}\tau\big(1 + \||b^\ve|^{\frac q2}\|^2_{L^2(0,\tau;L^4)}\big)  + (q-1)\big(\||p_d^\ve|^{\frac q2}\|^2_{L^2(0,\tau;L^4)} \! + \||c^\ve|^{\frac q2}\|^2_{L^2(0,\tau;L^4)}\big)\Big]  \\
&\leq  C_\delta\Big[1 + \||b^\ve|^{\frac q2}\|^2_{L^2(0,\tau;L^1)}  + q^{10}\big(\||p_d^\ve|^{\frac q2}\|^2_{L^2(0,\tau;L^1)}  + \||c^\ve|^{\frac q2}\|^2_{L^2(0,\tau;L^1)}\big)\Big]
\\ & \quad +
\delta\big[   \|\nabla|p_d^\ve|^{\frac q2}\|^2_{L^2(\Omega_{M,\tau}^\ve)}  + \|\nabla|c^\ve|^{\frac q2}\|^2_{L^2(\Omega_{M,\tau}^\ve)}+ \|\nabla|b^\ve|^{\frac q2}\|^2_{L^2(\Omega_{M,\tau}^\ve)}\big],
\end{aligned}
\end{equation*}
for $\tau \in 0,T]$ and any fixed $\delta >0$, where $C_u = 1+ \|\be(\bu^\ve)\|_{L^\infty(0,T;L^2(\Omega))}$.
 Combining the last estimate, \eqref{estim_LqGamma}, and \eqref{estim_Lq_22} 
 in the estimate~\eqref{estim_Lq_11}, choosing $\delta>0$ sufficiently small and subtracting the terms $\delta \||\nabla p_d^\ve|^{\frac q 2}\|^2_{L^2}$,  $\delta\||\nabla c^\ve|^{\frac q 2}\|^2_{L^2}$, and $\delta\||\nabla b^\ve|^{\frac q 2}\|^2_{L^2}$ on the left-hand side, yield
\begin{equation}
\begin{aligned}
&   \int_0^\tau\! \! \Big[\||\nabla p_d^\ve|^{\frac q 2}\|^2_{2} \!+ \||\nabla c^\ve|^{\frac q 2}\|^2_{2} \!  + \||\nabla b^\ve|^{\frac q 2}\|^2_{2}\Big]dt  + \|p_d^\ve (\tau)\|^q_{q} + \|c^\ve (\tau)\|^q_{q} + \|b^\ve (\tau)\|^q_{q}\\
& \leq  C_1\big[1+  \|\bu^\ve\|_{L^\infty(0,T; \cW(\Omega))}^q\big] + C_0^q +C_2 q^{10}\!\! \int_0^\tau \!\!\Big[\||b^\ve|^{\frac q 2} \|^2_{L^{1}} + \||p_d^\ve|^{\frac q 2}\|^2_{L^{1}} + \||c^\ve|^{\frac q 2}\|^2_{L^{1}}\Big]dt,
 \end{aligned}
\end{equation}
where $C_0^q = \|p_{d0}\|^q_{L^q} + \|c_0 \|^q_{L^q} +  \|b_0\|^q_{L^q}$.
Thus
$$
\begin{aligned}
  \|p_d^\ve (\tau)\|^q_{q} + \|c^\ve (\tau)\|^q_{q} + \|b^\ve (\tau)\|^q_{q} & \leq C_0^q + C_1 \big[1+ \|\bu^\ve\|_{L^\infty(0,T; \cW(\Omega))}^q\big] \\
 & + \tilde C_2 q^{10} \Big(\sup\limits_{(0,\tau)} \big[ \|p_d^\ve \|^{\frac q2}_{{\frac q2}} + \|c^\ve \|^{\frac q2}_{{\frac q2}} + \|b^\ve \|^{\frac q2}_{{\frac q2}}\big]\Big)^2,
 \end{aligned}
$$
for $\tau \in 0,T]$, where $\tilde C_2 = \max\{ 1, T C_2\}$. Considering $q=2^k$,  for $k=1, 2,\ldots$, and  using similar recursive iterations as in~\cite[Lemma 3.2]{Alikakos}, applied to
$w_q(\tau) =\|p_d^\ve (\tau)\|^q_{q} + \|c^\ve (\tau)\|^q_{q} + \|b^\ve (\tau)\|^q_{q}$, satisfying
$$
w_q (\tau) \leq  \tilde C_2 q^{10} \big[\sup\limits_{(0,\tau)} w_{q/2}\big]^2 + C_3 \big[1+\|p_{d0}\|_{q} + \|c_0 \|_{q} +  \|b_0\|_{q} +\|\bu^\ve\|_{L^\infty(0,T; \cW(\Omega))} \big]^q ,
$$
for $\tau \in (0,T]$ and $C_3 = \max\{ 1, C_1\}$,  we obtain
\begin{equation}\label{ineq:q_iter_1}
\begin{aligned}
& \big(\|p_d^\ve (\tau)\|_{q} + \|c^\ve (\tau)\|_{q} + \|b^\ve (\tau)\|_{q}\big)^q
\leq 2^{q-1}( \|p_d^\ve (\tau)\|^q_{q} + \|c^\ve (\tau)\|^q_{q}
 + \|b^\ve (\tau)\|^q_{q})
 \\
& \quad  \leq 2^{q-1} \tilde C_2^{q-1} 2^{20q} 2^{\frac q2-1}
 \big[ \|p_d^\ve \|_{L^\infty(0,\tau;L^1)} + \|c^\ve \|_{L^\infty(0,\tau; L^1)}
  + \|b^\ve \|_{L^\infty(0,\tau; L^1)} \big]^q
\\
& \qquad + 2^{q-1} \tilde C_2^{\frac q2-1}C_3^{\frac q2}2^{20 q} 2^{q-2} \big[1+\|p_{d0}\|_{q}
  + \|c_0 \|_{q} +  \|b_0\|_{q} +\|\bu^\ve\|_{L^\infty(0,T; \cW(\Omega))} \big]^q ,
 \end{aligned}
\end{equation}
for $\tilde C_2, C_3\geq 1$.
Applying the $q$th root and taking $q \to \infty$
imply estimates  for $p^\ve_d, c^\ve$ and $b^\ve$ in $L^\infty((0,T)\times \Omega_M^\ve)$. Then considering~\eqref{bs_Lp2}, we also obtain the estimate  for $b_s^\ve$ in~$L^\infty((0,T)\times \Omega_M^\ve)$. Notice that by using the extension properties, see Lemma~\ref{extension_local}, and applying the embedding inequalities to extended functions, all constants in the estimates above are independent of~$\ve$.

To show the equicontinuity of $p_\alpha^\ve$,   $c^\ve$, and $b^\ve$  with respect to  the time variable,  we consider $\phi_j(t,x)= \int_{t-h}^t  v_j(s, x) \kappa(s) ds$, with $\kappa(s) = 1$ for $s \in (0, T - h)$ and $\kappa(s) = 0$ for $s \in [-h, 0]\cup [T-h,T]$ and  $j=e,E, d,c,b$, as   test functions in~\eqref{weak_sol_n1}, where $v_\alpha = \widetilde p^\ve_\alpha = \theta_h p^\ve_\alpha - p^\ve_\alpha$, for $\alpha = e,E,d$,   $v_c = \widetilde c^\ve =\theta_h c^\ve - c^\ve$, and  $v_b=\widetilde b^\ve=\theta_h b^\ve - b^\ve$, respectively, and obtain
 \begin{equation*}
\begin{aligned}
& \sum_{\alpha=e,E,d}\!\!\!\!\| \widetilde p^\ve_\alpha \|^2_{2} +  \| \widetilde c^\ve \|^2_{2} + \| \widetilde b^\ve \|^2_{2}\leq
D \Big[\!\!\!\!\sum_{\alpha=e,E,d}\! \!\!\!\Big\langle\! \int_t^{t+h}\! \!\!\!\!\!|\nabla p_\alpha^\ve(s)|   ds, |\nabla \widetilde p_\alpha^\ve| \Big \rangle
 +  \Big\langle \!\int_t^{t+h}\! \!\!\!\!\!|\nabla c^\ve(s)|  ds,  |\nabla \widetilde c^\ve| \Big \rangle \\
&  \quad +  \Big\langle\! \int_t^{t+h}\!\!\! \!\!\!|\nabla b^\ve(s)|  ds,  |\nabla \widetilde b^\ve| \Big \rangle \Big]
  +
\Big \langle \!\int_t^{t+h}\!\!\! \big[R_{d}( p_d^\ve, c^\ve) +  R_s(b^\ve, b_s^\ve, \be(\bu^\ve))\big]  ds,  2|\widetilde p_d^\ve| +  |\widetilde c^\ve|   \Big\rangle
\\ &
\quad +  \Big\langle \!\int_t^{t+h}\!\!\! R_{e}(p_e^\ve, p_E^\ve) ds, |\widetilde p^\ve_e| + |\widetilde p^\ve_d| \Big\rangle
    +  \Big \langle \!\int_t^{t+h}\!\!\!\!\! \big[R_{d}( p_d^\ve, c^\ve) +  R_b(b^\ve,\be(\bu^\ve) ) + d_b b^\ve \big] ds,  |\widetilde b^\ve| \Big\rangle \\
& \quad  +  \sum_{\alpha=e,E,d}\!\!\!\Big\langle\! \int_t^{t+h}\!\!\! \!\! \big[J_\alpha(p_e^\ve, p_d^\ve, b^\ve)  + \gamma_\alpha p^\ve_\alpha \big]ds,
    |\tilde p^\ve_\alpha|  \Big\rangle_{\Gamma_{\cI, \tau} }  +    \Big\langle\! \int_t^{t+h} \!\!\!\!\!\big[ J_c(\bu^\ve) + \gamma_c c^\ve(s,x) \big] ds,
    |\tilde c^\ve|  \Big\rangle_{\Gamma_{\cI, \tau} },
\end{aligned}
\end{equation*}
for  $\tau \in (0, T-h]$ and any $h>0$, where $D=\max\{D_e, D_E, D_d, D_c, D_b\}$, $J_d =0$, and we  omit $\Omega^\ve_{M,\tau}$ in $\langle \cdot, \cdot\rangle$. Then  the boundedness of $p_e^\ve$, $p_E^\ve$, $b^\ve_s$, and the  estimates for $\bu^\ve$, $p^\ve_\alpha$, $c^\ve$, and $b^\ve$, with $\alpha=e,E,d$, imply the result. Equation~\eqref{weak_sol_bs}, together with Assumption~\ref{assumptions}.4 on $r_l$ and $Q_l$, for $l=b,s$,  yields
$$
\begin{aligned}
\|\partial_t b^\ve_s\|^2_{L^2(\Omega_{M,T}^\ve)} \leq C \big[1+  \|\be(\bu^\ve)\|^2_{L^2(\Omega_{M,T}^\ve)}\|b^\ve_s\|^2_{L^\infty(\Omega_{M,T}^\ve)} + \|b^\ve\|^2_{L^2(\Omega_{M,T}^\ve)} + \|b^\ve_s\|^2_{L^2(\Omega_{M,T}^\ve)} \big] \leq C.
\end{aligned}
$$
To prove equicontinuity of $\bu^\ve$ with respect to the time variable, we consider the difference of~\eqref{weak_sol_u} for $t+h$ and $t$. Taking $\bu^\ve(t+h,x) - \bu^\ve(t,x)$ as a test function,   using Assumption~\ref{assumptions}.6 on $P$ and $\bff$, and applying the trace and the Korn inequalities,  yield
\begin{equation}\label{estim_h_u}
 \begin{aligned}
 & \|\bu^\ve(t+h) -   \bu^\ve(t) \|^2_{\cW(\Omega)}  \leq C\big[\|\bff(t+h) - \bff(t)\|^2_{L^2(\Gamma_\cI)}\\
 & \quad + \|P(t+h) - P(t)\|^2_{L^2(\Gamma_\cI)}\! +  \|b^\ve(t+h) - b^\ve(t)\|_{L^\infty(\Omega_M^\ve)}^2\|\be(\bu^\ve(t)) \|^2_{2}\big]\\
  & \leq C \big[h^{2\sigma} + \|b^\ve(t+h) - b^\ve(t)\|_{L^\infty(\Omega_M^\ve)}^2 \big],
 \end{aligned}
\end{equation}
for  $t\in [0, T-h]$. To estimate $\|b^\ve(t+h) - b^\ve(t)\|_{L^\infty(\Omega_M^\ve)}$, we consider
$\widetilde b^\ve |\widetilde b^\ve|^{q-2}$  as a test function in the difference of equation for $b^\ve$ in \eqref{weak_sol_n1} for $t+h$ and $t$, with $t \in (0, T-h]$,  where $\widetilde b^\ve(t,x) = b^\ve(t+h, x) - b^\ve(t,x)$, and using assumptions on $R_d$, $r_b$, and $Q_b$ in Assumptions~\ref{assumptions}.2 and~\ref{assumptions}.4, obtain
\begin{equation}\label{estim_tilde_q1}
 \begin{aligned}
   \frac d{dt}  \|\widetilde b^\ve \|^q_{q}   + \|\nabla |\widetilde b^\ve|^{\frac q 2}\|^2_{2}
 \leq C\big( q\big[\|\be(\widetilde\bu^\ve)\|_{2}+   \|\widetilde p_d^\ve\|_{2}  +  \|\widetilde c^\ve\|_{2}  \big] \||\widetilde b^\ve|^{q-1} \|_{2}  \qquad
 \\ +  q\big[1+\|\be(\bu^\ve) \|_{2}\big] \||\widetilde b^\ve|^{q} \|_{2} \big) = C( I_1 + I_2),
 \end{aligned}
 \end{equation}
 where $\widetilde \bu^\ve(t,x) = \bu^\ve(t+h, x) - \bu^\ve(t, x)$.
Using $1<\frac {q(1+\varsigma)}{q \varsigma + 1} \leq 1+ \frac 1 \varsigma$ for some  $0<\varsigma <1$ and $q\geq 2$, implies
 \begin{equation} \label{estim_diff_11}
\begin{aligned}
 & \int_0^\tau \!\!\! I_1 d t
\leq q \Big[\! \int_0^\tau \!\!\! \!\big[ \|\widetilde p_d^\ve \|^{\kappa}_{2}+ \|\widetilde c^\ve \|^{\kappa}_{2}+ \|\be(\widetilde \bu^\ve) \|^{\kappa}_{2}\big] dt
\Big]^{\frac 1 \kappa}
\Big[\! \int_0^\tau  \!\!\!\! \!\| |\widetilde b^\ve|^{\frac q 2}\|^{2(1+\varsigma)}_{4}\! dt \Big]^{\frac {q-1}{(1+\varsigma) q}}
\\
& \leq \Big[\|\widetilde p_d^\ve \|_{L^{\tilde \kappa}(0,\tau; L^2)}^q \! +  \|\widetilde c^\ve \|_{L^{\tilde \kappa}(0,\tau; L^2)}^q \! +  \|\be(\widetilde \bu^\ve) \|_{L^{\tilde \kappa}(0,\tau; L^2)}^q \Big]
  +  (q-1)
\Big[ \int_0^\tau\!\!\! \| |\widetilde b^\ve|^{\frac q 2}\|^{2(1+\varsigma)}_{4} \! dt \Big]^{\frac 1{1+\varsigma}} \;,
 \end{aligned}
\end{equation}
where $\kappa = \frac{q(1+\varsigma)}{q\varsigma +1}$ and $\tilde \kappa = 1 + \frac 1 \varsigma$.  Applying the Gagliardo-Nirenberg inequality yields
 \begin{equation}\label{etim_GN_L4}
 \| |\widetilde b^\ve|^{\frac q 2}\|^{2}_{L^4}       \leq C\big[ \| \nabla |\widetilde b^\ve|^{\frac q 2}\|^{2 a }_{L^2}
     \| |\widetilde b^\ve|^{\frac q 2}\|^{2(1-a)}_{L^1} +  \| |\widetilde b^\ve|^{\frac q 2}\|^{2}_{L^1}\big] ,
 \end{equation}
 where  $a =   9/{10} $ for a three-dimensional domain.   Considering $\varsigma$ such that $a(1+ \varsigma) < 1$ we obtain
 \begin{equation}\label{complex_in}
\begin{aligned}
 &\Big[\!\int_0^\tau \!\!\!\! \||\widetilde b^\ve|^{\frac q 2}\|^{2(1+\varsigma)}_{4} dt\Big]^{\frac 1{1+\varsigma}}\!\!
\!\leq C_1 \| \nabla |\widetilde b^\ve|^{\frac q 2}\|^{2a}_{2}
\Big[\! \int_0^\tau \!\!\! \!\| |\widetilde b^\ve|^{\frac q 2}\|^{\frac{2(1+\varsigma)(1-a)}{1- a (1+\varsigma)}}_{L^1} \! dt\Big]^{\frac{1- a (1+\varsigma)}{1+\varsigma}}\!\!\!\!\! \!\!+ C_2 \Big[\!\int_0^\tau \!\!\!\! \||\widetilde b^\ve|^{\frac q 2}\|^{2(1+\varsigma)}_{L^1} dt\Big]^{\frac 1{1+\varsigma}}
\\
&\leq  \frac \delta  q \| \nabla |\widetilde b^\ve|^{\frac q 2}\|^{2}_{L^2}
 + C_\delta \, q^{\frac{ a}{(1-a)}}\Big[ \int_0^\tau\!\!\!  \| |\widetilde b^\ve|^{\frac q 2}\|^{\frac{2(1+\varsigma)(1-a)}{1-a(1+\varsigma)}}_{L^1} dt \Big]^{\frac {1- a(1+\varsigma)}{(1-a)(1+\varsigma)}} + \Big[\int_0^\tau \!\!\!\! \||\widetilde b^\ve|^{\frac q 2}\|^{2(1+\varsigma)}_{L^1} dt\Big]^{\frac 1{1+\varsigma}},
 \end{aligned}
\end{equation}
for  arbitrary fixed $\delta>0$,  and hence
\begin{equation*}
\begin{aligned}
 \int_0^\tau \!\!\!\! I_1 dt
   \leq   \Big[\|\be(\widetilde \bu^\ve) \|_{L^{1+\frac 1 \varsigma}(0,\tau; L^2)}^q+ \|\widetilde p_d^\ve \|_{L^{1+\frac 1 \varsigma}(0,\tau; L^2)}^q+\|\widetilde c^\ve \|_{L^{1+\frac 1 \varsigma}(0,\tau; L^2)}^q \Big]\\
   +
 \delta  \frac{ q-1}{q} \| \nabla |\widetilde b^\ve|^{\frac q 2}\|^2_{L^2(0, \tau, L^2)}   +
 C q^{10} \Big[\sup_{(0,\tau)} \| |\widetilde b^\ve|^{\frac q 2}\|_{L^1}\Big]^2,
 \end{aligned}
 \end{equation*}
 for  $\tau \in (0,T-h]$.  The Gagliardo-Nirenberg inequality also ensures
 \begin{equation}\label{estim_ub_I2}
 \begin{aligned}
\int_0^\tau \!\!  I_2 dt
 \leq  C q^{10}\big[1+\|\be(\bu^\ve)\|^{9}_{L^\infty(0,\tau; L^2)}\big]
 \||\widetilde b^\ve|^{\frac q 2}\|^2_{L^2(0,\tau; L^1)}
  +  \delta   \|\nabla |\widetilde b^\ve|^{\frac q 2}\|^2_{L^2(0, \tau; L^2)}.
  \end{aligned}
\end{equation}
Combining the estimates from above and integrating~\eqref{estim_tilde_q1} over $(0, \tau)$, we obtain
 \begin{equation}\label{estim_tilde_b}
 \begin{aligned}
  \|\widetilde b^\ve(\tau) \|^q_{q} \leq \|\widetilde b^\ve(0) \|^q_{q} + C\Big[
 \|\be(\widetilde \bu^\ve) \|_{L^{1+\frac 1 \varsigma}(0,\tau; L^2)}^q \!\! + \|\widetilde p_d^\ve \|_{L^{1+\frac 1 \varsigma}(0,\tau; L^2)}^q\\
 +\|\widetilde c^\ve \|_{L^{1+\frac 1 \varsigma}(0,\tau; L^2)}^q
 \!\! +  q^{10}\big(\sup_{(0,\tau)} \| |\widetilde b^\ve|^{\frac q 2}\|_{L^1}\big)^2\Big].
 \end{aligned}
 \end{equation}
Iterations in $q$, with $q = 2^\kappa$ for  $\kappa=2,3,\ldots$,  as in~\cite[Lemma~3.2]{Alikakos} and~\eqref{ineq:q_iter_1},  imply
\begin{equation*}
\begin{aligned}
 \|\widetilde b^\ve(\tau) \|^q_{q}
  \leq    C^q 2^{22 q}  \big[
  \| \be(\widetilde \bu^\ve)\|_{L^{1+ \frac 1 \varsigma}(0,\tau;L^2)} \! +
   \|\widetilde p_d^\ve \|_{L^{1+\frac 1 \varsigma}(0,\tau; L^2)}\!\!+\|\widetilde c^\ve \|_{L^{1+\frac 1 \varsigma}(0,\tau; L^2)}\\
   + \|\widetilde b^\ve(0) \|_{q} + \|\widetilde b^\ve \|_{L^\infty(0,\tau; L^2)} \big]^q,
\end{aligned}
\end{equation*}
for $\tau \in (0, T-h]$ and $9(1+\varsigma) <10$.  Considering $q=2$ in \eqref{estim_tilde_q1} and using \eqref{estim_h_u} and
$$
\begin{aligned}
& \|\widetilde p_d^\ve\|_{L^2(\Omega_{M,\tau}^\ve)}  +  \|\widetilde c^\ve\|_{L^2(\Omega_{M,\tau}^\ve)} + \|\widetilde b^\ve \|_{L^2(\Omega_{M,\tau}^\ve)} \leq C h^{\frac 14},\\
& \|\widetilde b^\ve \|_{L^4} \leq C \big[ \| \nabla\widetilde b^\ve\|^{3/4 }_{L^2}
     \| \widetilde b^\ve\|^{1/4}_{L^2} +  \| \widetilde b^\ve\|_{L^1}\big] \leq \delta \| \nabla\widetilde b^\ve\|_{L^2}
     + C_\delta \| \widetilde b^\ve\|_{L^2},
 \end{aligned}
 $$
for any fixed $\delta>0$,  we obtain
 $$
 \|\widetilde b^\ve \|^2_{L^\infty(0, \tau; L^2(\Omega_M^\ve))} \leq  \|\widetilde b^\ve(0) \|^2_{L^2(\Omega_M^\ve)} +
 C \big(h^{\frac 12} + h^{2 \sigma} +  \|\widetilde b^\ve \|^2_{L^2(0,\tau; L^\infty(\Omega_M^\ve))}\big).
 $$
Taking the $1/q$-power and letting $q \to \infty$, and then  using \eqref{estim_h_u}, the estimates for $\widetilde p_d^\ve$, $\widetilde c^\ve$ in $L^2(\Omega_{M,T}^\ve)$, and  boundedness of $p_d^\ve$, $c^\ve$, yield
\begin{equation}\label{estim_equi_b_L_infty}
\begin{aligned}
  \|\widetilde b^\ve \|_{L^\infty(\Omega_{M,\tau}^\ve)}
  \leq   C \big[\| \be(\widetilde \bu^\ve)\|_{L^{1+ \frac 1 \varsigma}(0,\tau;L^2)}
  + \|\widetilde p_d^\ve \|_{L^{1+\frac 1 \varsigma}(0,\tau; L^2)} \qquad \qquad
\\
 +\|\widetilde c^\ve \|_{L^{1+\frac 1 \varsigma}(0,\tau; L^2)}
  +\|\widetilde b^\ve \|_{L^\infty(0,\tau; L^2)}   + \|\widetilde b^\ve(0) \|_{L^\infty(\Omega_M^\ve)}  \big]
 \\
  \leq C\big[(\tau^{\frac 12} + \tau^{\frac \varsigma{\varsigma +1}})\|\widetilde b^\ve \|_{L^\infty(\Omega_{M,\tau}^\ve)}+ \|\widetilde b^\ve(0) \|_{L^\infty(\Omega_M^\ve)} + h^{\frac \varsigma{2(\varsigma +1)}}+ h^{\frac 14} + h^\sigma \big]  ,
  \end{aligned}
\end{equation}
for $\tau \in (0, T-h]$ and $\varsigma \in (0, 1/9)$. Considering $|b^\ve(t) - b_0|^{q-2}(b^\ve(t) -b_0) $ as a test function in the equation for $b^\ve$ in \eqref{weak_sol_n1} and integrating over $(0,h)$ gives
$$
\begin{aligned}
& \|\hat b^\ve(h) \|^q_{q} +  \frac{d (q-1)} q\|\nabla |\hat b^\ve|^{\frac q2}\|_{L^2(0,h;L^2)}^2
\\
& \leq D^2_b q^2 \Big[\! \int_0^h \!\! \|\nabla b_0 \|_{4}^{2\kappa_1} dt\Big]^{\frac 1 {\kappa_1}}
\Big[\! \int_0^h \!\! \| \hat b^\ve(t)|^{\frac q2}\|_{4}^{2(1+\zeta)} d t\Big]^{\frac{q-2}{q(1+\zeta)}}
+ q \Big[\!\int_0^h \!\! \!\big(d_b^\kappa\|b_0 \|_{2}^\kappa  \\
& \quad + \|R_d(p^\ve_d, c^\ve)\|_{2}^\kappa+ \|R_b(b^\ve, \be(\bu^\ve))\|_{2}^\kappa \big)dt\Big]^{\frac 1 \kappa}
\Big[\!\int_0^h \!\!\| \hat b^\ve(t)|^{\frac q2}\|_{4}^{2(1+\zeta)} d t\Big]^{\frac{q-1}{q(1+\zeta)}}
\\
& \leq  C_1 \big[h^{\frac q {2\kappa_1}} \|\nabla b_0 \|^q_{4} + h^{\frac q \kappa} \|b_0 \|^q_{2} + \|\be(\bu^\ve) \|_{L^{\kappa}(0,h; L^2)}^q
+ \| p_d^\ve \|_{L^{\kappa}(0,h; L^2)}^q
\\
& \quad +\|c^\ve \|_{L^{\kappa}(0,h; L^2)}^q \big]
  +C_2\big[(q-2) q^{1+ \frac 2{q-2}}+  (q-1)\big] \Big[\!\int_0^h \!\!\| \hat b^\ve(t)|^{\frac q2}\|_{4}^{2(1+\zeta)} d t\Big]^{\frac 1{(1+\zeta)}},
 \end{aligned}
$$
where $\hat b^\ve(t) =b^\ve(t) -b_0$,  $\kappa_1 = (1+ \zeta) q/(q\zeta + 2)$, $\kappa = (1+ \zeta) q/(q\zeta + 1)$, $q\geq 4$, and $0<\zeta <1/9$.
Then similar calculations as  in \eqref{complex_in}, assumptions on the initial condition $b_0$,  and  a priori estimates for $\bu^\ve$, $p_d^\ve$ and $c^\ve$  yield
$$
\begin{aligned}
& \||\hat b^\ve(h)|^{\frac q 2} \|^2_{L^2} +  \|\nabla |\hat b^\ve(t)|^{\frac q2}\|_{L^2(0,h;L^2)}^2
 \\
 & \qquad \leq C \Big[ C^q\big(h^{\frac q {\kappa}}+ h^{\frac q {2\kappa_1}}\big) + (q^{10} + q^{30}) h^{\frac{1-9\varsigma}{1+\varsigma}} \big(\sup_{(0,h)} \| |\hat b^\ve|^{\frac q 2}\|_{L^1}\big)^2 \Big],
 \end{aligned}
$$
for $h\in (0,1)$. Iterating over $q=2^k$, for $k=2,3,\ldots$,  as in~\cite[Lemma~3.2]{Alikakos} and~\eqref{ineq:q_iter_1},  we obtain
$$
\|\tilde b^\ve(0) \|_{L^\infty(\Omega^\ve_{M})}=\|\hat b^\ve(h) \|_{L^\infty(\Omega^\ve_{M})}\leq C \big(h^{\frac \zeta {1+ \zeta}}+ h^{\frac \zeta {2(1+ \zeta)}}
+ h^{\frac{1-9 \zeta} {1+ \zeta}} \big).
$$
Iterating in~\eqref{estim_equi_b_L_infty} over time-intervals such that
$C\tau^{\frac \varsigma{\varsigma +1}} \leq 1/2$, we obtain
$$
\|\theta_h b^\ve - b^\ve \|_{L^\infty((0,T-h)\times \Omega_M^\ve)} \leq C \big(h^{\frac \varsigma{2(\varsigma +1)}} +  h^{\frac{1-9 \zeta} {1+ \zeta}} + h^\sigma \big).
$$
Then, estimate~\eqref{estim_h_u}  ensures
$$
\|\theta_h \bu^\ve - \bu^\ve\|_{L^2(0,T-h; \cW(\Omega))} \leq  C \big(h^{\frac \varsigma{2(\varsigma +1)}} +  h^{\frac{1-9 \zeta} {1+ \zeta}} \big) \; \text{ for }   \varsigma \in (0, 1/9) \text{ and } \sigma \geq 1/18,
$$
which proves the second estimate in the lemma.
\end{proof}

\begin{theorem}\label{th:existence_micro}
 Under Assumption~\ref{assumptions},  there exists a unique solution of problem~\eqref{EQST}--\eqref{IC}.
\end{theorem}
\begin{proof}
For a given $b^\ve \in L^\infty((0,T)\times\Omega^\ve_{M})$,  applying the Lax-Milgram theorem and using the properties of the elasticity tensor and the second Korn and trace inequalities, see e.g.~\cite{Ciarlet, Galdi},  we obtain existence of a unique solution of \eqref{EQST}  satisfying
\begin{equation} \label{contract_u}
  \| \bu^\ve_1 - \bu^\ve_2\|^2_{L^\infty(0,T;\cW(\Omega))}  \leq C \|b^\ve_1 - b^\ve_2\|^2_{L^\infty(\Omega_{M, T}^\ve)},
\end{equation}
where $\bu^\ve_1$ and $\bu^\ve_2$ are solution of \eqref{EQST} for  $b^\ve_1, b^\ve_2 \in L^\infty((0,T)\times\Omega_{M}^\ve)$ respectively.

 For given $\bu \in L^\infty(0,T; \cW(\Omega))$, using a fixed-point argument, the Galerkin method and estimates similar to those in Lemma~\ref{estim:apriori} we obtain existence of a unique non-negative weak solution of problem~\eqref{main_2}--\eqref{IC}. In the derivation of estimates in $L^\infty((0,T)\times \Omega_M^\ve)$, we first consider $|v^\ve_K|^{q-2} v^\ve_K$, where $v^\ve_K=\min\{ v^\ve, K\}$ and $\|v^\ve(0)\|_{L^\infty(\Omega^\ve_M)} \leq K$, for $v^\ve =  p_d^\ve, c^\ve, b^\ve$, as test functions in~\eqref{weak_sol_n1} and, by deriving estimates uniform in $K$, show that $|v^\ve|^{\frac q 2} \in L^2(0,T; H^1(\Omega_M^\ve))$ and $v^\ve \in L^q((0,T)\times \Omega_M^\ve)$, for any $q\geq 0$,  and hence  $|v^\ve|^{q-2} v^\ve$ are  admissible test functions in~\eqref{weak_sol_n1}.
 More specifically, using the nonnegativity of $p^\ve_d$, $c^\ve$, $b^\ve$ and considering $|p^\ve_{d,K}|^{q-2} p^\ve_{d,K}$, $|c^\ve_K|^{q-2} c^\ve_K$,  and $|b^\ve_K|^{q-2} b^\ve_K$ as test functions in the corresponding equations in \eqref{weak_sol_n1} yield
 \begin{equation}\label{estim_K_q}
\begin{aligned}
 &   \int_0^\tau \!\!\Big[\|\nabla |p_{d,K}^\ve|^{\frac q 2}\|^2_{2} \! +   \|\nabla |c_K^\ve|^{\frac q 2}\|^2_{2}  \! + \|\nabla |b_K^\ve|^{\frac q 2}\|^2_{2}\Big] dt + \|p_{d,K}^\ve(\tau)\|^q_{q} \\
 &\quad    + \|c^\ve_K (\tau)\|^q_{q} \!
  +  \|b^\ve_K (\tau)\|^q_{q} \! + q C_1\!\!\int_0^\tau \!\!  \Big[ \|p_{d,K}^\ve\|^q_{L^q(\Gamma_{\cI})} \! +\! \|c^\ve_K \|^q_{L^q(\Gamma_{\cI})}\!  +\!  \|b^\ve_K\|^q_{q} \Big] dt 
\\
& \leq  C_2\! \!\int_0^\tau\!\!\Big[1 +
q\|\be(\bu^\ve)\|_{2}\|b_s^\ve\|_{{2q}}\big(\|p_{d,K}^\ve\|_{{2q}}^{q-1} + \|c^\ve_K\|_{{2q}}^{q-1}\big) + \delta\|p_{d,K}^\ve\|^q_{q}  \\
& \qquad \qquad  + \big\langle p_d^\ve, |b_K^\ve|^{q-1}\big\rangle_{\Omega_M^\ve}\Big]dt
  + q\!\int_0^\tau\!\! \!\big\langle |J_c(\bu^\ve)|, |c^\ve_K|^{q-1} \big\rangle_{\Gamma_\cI} dt,
 \end{aligned}
\end{equation}
for any fixed $\delta >0$. Testing~\eqref{weak_sol_bs} with $|b^\ve_{s,K}|^{q-2} b^\ve_{s,K}$ implies
 \begin{equation}
  \begin{aligned}
 &\|b^\ve_{s, K} (\tau)\|^q_{q} + q d_s\int_0^\tau \!\!\!\! \|b^\ve_{s, K}\|^q_{q} dt
 + q \rho_s\int_0^\tau \!\!\! Q_s(b^\ve, \be(\bu^\ve)) \|b^\ve_{s, K}\|^q_{q}dt  \\
 & \leq  \|b^\ve_s (0)\|^q_{q}  + C_\delta \Big[\int_0^\tau  \!\! \!\!\|b^\ve\|^{\frac q 3}_{q} dt\Big]^3
+ (q-2)\delta \!\!\int_0^\tau\!\!\! \!\|b^\ve_{s, K} \|^q_{q} dt
  + \delta \!\sup\limits_{(0,\tau)} \|b^\ve_{s, K}\|^q_{q},
  \end{aligned}
 \end{equation}
for any fixed $\delta>0$.  Thus for $0<\delta \leq\min\{ d_s/2, 1/2\}$ and $q>3$ we have
$$
\begin{aligned}
 \sup\limits_{(0,\tau)} \|b^\ve_{s, K} \|^q_{q} + q \int_0^\tau \!\!\Big[ d_s\|b^\ve_{s, K}\|^q_{q}  + 2\rho_s Q_s(b^\ve, \be(\bu^\ve)) \|b^\ve_{s, K}\|^q_{q} \Big]dt
 \leq C_1 + C_2 \Big[\int_0^\tau \! \! \|b^\ve \|^{\frac q 3}_{q} dt\Big]^3.
 \end{aligned}
  $$
Then the terms on the right-hand side of \eqref{estim_K_q} are estimates as follows
$$
\begin{aligned}
&\int_0^\tau \!\! \!\big\langle p_d^\ve, |b_K^\ve|^{q-1}\big\rangle_{\Omega_M^\ve}\! dt
\leq
\|p_d^\ve\|_{L^\infty(0,\tau; L^2(\Omega_M^\ve))} \||b_K^\ve|^{\frac q 2}\|_{L^2(0,\tau; L^4(\Omega^\ve_M))}^{2 \frac{q-1} q}
\\
&\qquad \leq C \|p_d^\ve\|_{L^\infty(0,\tau; L^2(\Omega_M^\ve))}^q + \delta \big[\||b_K^\ve|^{\frac q 2}\|^2_{L^2(\Omega_{M,\tau}^\ve)} + \|\nabla |b_K^\ve|^{\frac q 2}\|^2_{L^2(\Omega_{M,\tau}^\ve)} \big],
\end{aligned}
$$
for any fixed $\delta >0$, and
$$
\begin{aligned}
& I = \int_0^\tau\!\! \|b_s^\ve\|_{{2q}}\big[\|p_{d,K}^\ve\|_{{2q}}^{q-1} + \|c^\ve_K\|_{{2q}}^{q-1}\big]  dt
\\
& \quad \leq  C\! \int_0^\tau \!\!\Big[1+ \Big[\int_0^t \!\!\|b^\ve\|^{\frac {2q}3}_{{2q}} d\zeta \Big]^{\frac 3{2q}} \Big]\big[\|p_{d,K}^\ve\|_{{2q}}^{q-1} + \|c^\ve_K\|_{{2q}}^{q-1}\big]   dt.
\end{aligned}
$$
Using the embedding theorem,  for $d\leq 3$, we obtain
$$
\int_0^t \!\!\|b^\ve\|^{\frac {2q}3}_{{2q}} d\zeta  =
\int_0^t \!\!\||b^\ve|^{\frac q3}\|_{6}^2 \, d\zeta \leq C \big[\||b^\ve|^{\frac q3}\|_{L^2(\Omega^\ve_{M,t})}^2
+ \|\nabla|b^\ve|^{\frac q3}\|_{L^2(\Omega^\ve_{M,t})}^2\big].
$$
Then,  for $q \geq 3$, combining the relation
$$
\begin{aligned}
\|\nabla|b^\ve|^{\frac q3}\|_{L^2(\Omega_{M,\tau}^\ve)}^2 & \leq \|\nabla b^\ve\|^2_{L^2(\Omega^{\ve,1}_{M,\tau})}
+ \|\nabla| b^\ve|^{\frac {q-1} 2}\|^2_{L^2(\Omega^\ve_{M,\tau} \setminus\Omega^{\ve,1}_{M,\tau})} \\ & \leq
\|\nabla b^\ve\|^2_{L^2(\Omega^{\ve}_{M,\tau})}
+ \|\nabla| b^\ve|^{\frac {q-1} 2}\|^2_{L^2(\Omega^\ve_{M,\tau})},
\end{aligned}
$$
where $\Omega^{\ve,1}_{M,\tau} =\{(t,x) \in (0,\tau)\times \Omega^\ve_M\, : \, b^\ve(t,x) \leq 1\}$,  with the estimate
$$
\begin{aligned}
 \|p_{d,K}^\ve\|_{{2q}}^{q-1} \!+ \|c^\ve_K\|_{{2q}}^{q-1}
\!\leq C\Big[\||p_{d,K}^\ve|^{\frac q2}\|_{{2}}^{\frac{q-1}{2q}}
\|\nabla |p_{d,K}^\ve|^{\frac q2}\|_{{2}}^{\frac{3(q-1)}{2q}}\! \!\!\!+
 \||c_{K}^\ve|^{\frac q2}\|_{{2}}^{\frac{q-1}{2q}}
\|\nabla |c_{K}^\ve|^{\frac q2}\|_{{2}}^{\frac{3(q-1)}{2q}}\\
+ \||p_{d,K}^\ve|^{\frac q2}\|_{2}^{\frac{2(q-1)}q} + \||c_{K}^\ve|^{\frac q2}\|_{2}^{\frac{2(q-1)}q} \Big],
\end{aligned}
$$
derived from  the Gagliardo-Nirenberg inequality, we obtain
$$
\begin{aligned}
I \leq & \delta \, \big[ \||p_{d,K}^\ve|^{\frac q2}\|_{{2}}^2
+ \|\nabla |p_{d,K}^\ve|^{\frac q2}\|_{2}^2 +
 \||c_{K}^\ve|^{\frac q2}\|_{{2}}^2 +
\|\nabla |c_{K}^\ve|^{\frac q2}\|_{{2}}^2 \big] \\
& + C_\delta
\big[1+ \|\nabla b^\ve\|^2_{2} + \||b^\ve|^{\frac{q-1} 2} \|^2_{2}
+  \|\nabla |b^\ve|^{\frac{q-1} 2} \|^2_{2}\big]^{\frac 32},
\end{aligned}
$$
for  $q \geq 3$,  $d\leq 3$, and  any fixed $\delta >0$.
The trace inequality implies
$$
\int_0^\tau \!\! \big\langle |J_c(\bu^\ve)|, |c^\ve_K|^{q-1} \big\rangle_{\Gamma_\cI} dt \leq
C \int_0^\tau \!\!\Big[\| \bu^\ve\|^q_{\mathcal W(\Omega)}
+ \delta \big(\| |c^\ve_K|^{\frac q 2}\|^2_{2}
+ \|\nabla  |c^\ve_K|^{\frac q 2}\|^2_{2} \big)\Big] d t.
$$
Combining the estimates from above and choosing appropriate $\delta>0$  yields
\begin{equation*}\label{estim_qK}
 \begin{aligned}
 \sup_{(0,T)}\! \big[ \|p_{d,K}^\ve \|^q_{q} + \|c^\ve_K \|^q_{q} +  \|b^\ve_K \|^q_{q}\big]\! + \Big[\|\nabla |p_{d,K}^\ve|^{\frac q 2}\|^2_{2} +   \|\nabla |c_K^\ve|^{\frac q 2}\|^2_{2}   + \|\nabla |b_K^\ve|^{\frac q 2}\|^2_{2}\Big] \\
  \leq C\Big[ 1+ \| \bu^\ve\|_{L^\infty(0,T; \mathcal W(\Omega))} + \|p_d^\ve\|_{L^\infty(0,T; L^2(\Omega^\ve_{M}))}^q \qquad
 \\
 + \big[\|\nabla b^\ve\|^2_{L^2(\Omega^\ve_{M,T})} + \||b^\ve|^{\frac{q-1} 2} \|^2_{L^2(\Omega^\ve_{M,T})}
+  \|\nabla |b^\ve|^{\frac{q-1} 2} \|^2_{L^2(\Omega^\ve_{M,T})}\big]^{\frac 32} \Big].
 \end{aligned}
\end{equation*}
Using estimates for $q=2$,  iterating over $q\geq 3$,  and letting $K \to \infty$,
we obtain the required regularity for $c^\ve$, $p^\ve_d$, $b^\ve$ and $b^\ve_s$.

 To show existence of a unique solution  of the coupled problem~\eqref{EQST}--\eqref{IC} we apply the Banach fixed-point theorem to $\mathcal K: L^\infty(\Omega_{M,T}^\ve)\to L^\infty(\Omega_{M,T}^\ve)$, where for given $\hat b^\ve \in L^\infty(\Omega_{M,T}^\ve)$, $b^\ve = K(\hat b^\ve)$ is a solution of \eqref{EQST}-\eqref{IC}, with $\hat b^\ve$ instead of $b^\ve$ in~\eqref{EQST}.
To derive the contraction inequality we consider $\tilde p_d^\ve = p^\ve_{d,1} - p^\ve_{d,2}$, $\tilde c^\ve = c^\ve_{1} - c^\ve_{2}$,  $\tilde b^\ve = b^\ve_{1} - b^\ve_{2}$, and
 $\tilde b_s^\ve = b^\ve_{s,1} - b^\ve_{s,2}$ as test functions in the  equations for $\tilde p_d^\ve$, $\tilde c^\ve$, $\tilde b^\ve$, and $\tilde b^\ve_s$ respectively, and using  assumptions on the nonlinear functions,  estimates in Lemma~\ref{estim:apriori} and the Gagliardo-Nirenberg inequality obtain
 \begin{equation}\label{estim_contract_bs}
 \begin{aligned}
   \|\widetilde b_s^\ve \|^2_{2} + d_s  \int_0^\tau\!\!\! \|\widetilde b_s^\ve \|^2_{2} dt + 2\rho_s \int_0^\tau\!\!\! \|\sqrt{Q_s(b^\ve, \bu_1^\ve)}\, \widetilde b_s^\ve\|^2_{2} dt
    \leq C\int_0^\tau \!\!\!\big[\| \widetilde b^\ve \|^2_{2} + \|\be( \widetilde \bu^\ve) \|^2_{2} \big]dt
   \end{aligned}
 \end{equation}
 and
\begin{equation*}\label{estim_tilde_21}
\begin{aligned}
&  \frac d{dt} \big[\|\widetilde p_{d}^\ve \|^2_{2} +  \|\widetilde c^\ve \|^2_{2}+
\|\widetilde{b}^\ve \|^2_{2} \big]
 + 2d\big[ \|\nabla \widetilde p_d^\ve\|^2_{2} + \|\nabla \widetilde c^\ve\|^2_{2} + \|\nabla \widetilde b^\ve\|^2_{2}\big]
  \!\leq \!C\Big[\|\be(\widetilde\bu^\ve)\|_{2}^2
  \\
  & +\|\be(\bu^\ve_1) \|_{2} \|\widetilde b^\ve\|^2_{4} +  \|\widetilde p_d^\ve\|^2_{2} +
\|\widetilde c^\ve \|^2_{2} + \delta \|\nabla \widetilde c^\ve\|^{2}_{2}+ \|\widetilde b^\ve \|^2_{2}
  + \big(\|\be(\widetilde\bu^\ve)\|_{2} \|b_s^\ve(1+ b^\ve)\|_{4} \\
  & \quad + \|b_s^\ve\|_{L^\infty} \|\be(\bu^\ve)\|_{2} \|\widetilde b^\ve\|_{4} + \|\be(\bu^\ve)\|_{2}^{\frac 12}\|\sqrt{Q_s(b^\ve,\bu^\ve_1)}\, \widetilde b_s^\ve\|_{2}\big)\big(\|\widetilde p_{d}^\ve\|_{4} +  \|\widetilde c^\ve \|_{4}\big)
\Big]\\
& \leq C \big[\|\be(\widetilde\bu^\ve)\|_{2}^2 + \|\widetilde p_d^\ve\|^2_{2} +
\|\widetilde c^\ve \|^2_{2}+ \|\widetilde b^\ve \|^2_{2}\big] +\delta\big[ \|\nabla \widetilde p_d^\ve\|^{2}_{2}+\|\nabla \widetilde c^\ve\|^{2}_{2}+ \|\nabla \widetilde b^\ve\|^{2}_{2} \big],
\end{aligned}
\end{equation*}
for any fixed $\delta>0$. Choosing $\delta>0$ sufficiently small and
applying  the  Gr\"onwall inequality  yield
\begin{equation}\label{estim_tilde_22}
\begin{aligned}
 &  \|\widetilde p_{d}^\ve(\tau) \|^2_{2} +  \|\widetilde c^\ve(\tau) \|^2_{2}+  \|\widetilde b^\ve(\tau) \|^2_{2}
 +  \int_0^\tau \!\!\big[\|\nabla \widetilde p_d^\ve\|^2_{2} + \|\nabla \widetilde c^\ve\|^2_{2}
 + \|\nabla \widetilde b^\ve\|^2_{2} \big]dt
\leq C  \int_0^\tau\!\!\! \|\be(\widetilde\bu^\ve)\|_{2}^2 dt,
\end{aligned}
\end{equation}
for $\tau \in (0,T]$. Considering now $\widetilde b^\ve |\widetilde b^\ve|^{q-2}$ as a test function in the equation for $b^\ve$ in~\eqref{weak_sol_n1}, we obtain
\begin{equation}\label{estim_tilde_q}
 \begin{aligned}
   \frac d{dt}  \|\widetilde b^\ve \|^q_{q}   + \|\nabla |\widetilde b^\ve|^{\frac q 2}\|^2_{2}
 \leq C\Big[ q\big(\|\be(\widetilde\bu^\ve)\|_{2}+   \|\widetilde p_d^\ve\|_{2}  +  \|\widetilde c^\ve\|_{2}  \big) \||\widetilde b^\ve|^{q-1} \|_{2} \qquad  \\ +  q\|\be(\bu^\ve_2) \|_{2} \||\widetilde b^\ve|^{q} \|_{2} \Big] = I_1 + I_2.
 \end{aligned}
 \end{equation}
 Then calculations similar to~\eqref{estim_diff_11}, \eqref{etim_GN_L4} and \eqref{complex_in}, and  estimates for $\widetilde p_d^\ve$ and $\widetilde c^\ve$  in~\eqref{estim_tilde_22},  imply
\begin{equation*}
\begin{aligned}
& \int_0^\tau \!\! I_1 dt
 \leq C \|\be(\widetilde \bu^\ve) \|_{L^{1+\frac 1\varsigma}(0,\tau; L^2)}^q \!\! +  (q-1)
\Big[ \int_0^\tau\!\! \| |\widetilde b^\ve|^{\frac q 2}\|^{2(1+\varsigma)}_{4} dt \Big]^{\frac 1{1+\varsigma}}
\\
& \; \leq  C  \|\be(\widetilde \bu^\ve) \|_{L^{1+\frac 1 \varsigma}(0,\tau; L^2)}^q
\!\!\! \!+
 \delta  \frac{ q-1}{q} \| \nabla |\widetilde b^\ve|^{\frac q 2}\|^2_{2} \!  +
 C_\delta   \tau^{\frac{1- a(1+\varsigma)}{(1-a)(1+\varsigma)}} \, q^{10} \Big[\!\sup_{(0,\tau)} \| |\widetilde b^\ve|^{\frac q 2}\|_{L^1}\Big]^2,
 \end{aligned}
\end{equation*}
for $ \tau\in (0, T]$,  any $\delta>0$, $a =   9/{10} $ for a three-dimensional domain, and   $0<\varsigma<1$ such that $a(1+ \varsigma) < 1$.
For $I_2$ we have similar estimate as in \eqref{estim_ub_I2}. Considering   iterations in $q$ as in    \cite[Lemma~3.2]{Alikakos}, with $q = 2^\kappa$ and  $\kappa=1, 2,3,\ldots,$  we obtain
\begin{equation*}
 \|\widetilde b^\ve \|_{L^\infty((0,\tau)\times \Omega_M^\ve)}
  \leq  C  \| \be(\widetilde \bu^\ve)\|_{L^{1+ \frac 1 \varsigma}(0,\tau;L^2(\Omega))}  \quad
  \text{ for } \; \varsigma \in (0, 1/9) \; \text{ and } \; \tau \in (0, T].
\end{equation*}
Thus, together with \eqref{contract_u}, for $b_1^\ve = \mathcal K(\hat b_1^\ve)$ and $b_2^\ve = \mathcal K(\hat b_2^\ve)$ we obtain
$$
\|b^\ve_1 - b^\ve_2 \|_{L^\infty((0,\tau)\times \Omega_M^\ve)} \!\leq  C \tau^{\frac \zeta{1+\zeta}}  \| \be(\widetilde \bu^\ve)\|_{L^\infty(0,\tau;L^2(\Omega))} \!\leq  C \tau^{\frac \zeta{1+\zeta}} \|\hat b^\ve_1 - \hat b^\ve_2 \|_{L^\infty((0,\tau)\times \Omega_M^\ve)},
$$
which for $C \tau^{\frac \zeta{1+\zeta}} < 1$ is a contraction inequality. Thus applying the Banach fixed point theorem and  iterating  over time-intervals,  yield  the existence of a unique weak solution of the coupled model~\eqref{EQST}-\eqref{IC}.
\end{proof}


\section{Derivation of macroscopic problem for microscopic model~\eqref{EQST}--\eqref{IC}}\label{macro_model}
To derive the macroscopic equations for model~\eqref{EQST}--\eqref{IC} defined in a domain with  non-periodic microstructure we  approximate it by a  problem defined in a domain with a locally periodic  microstructure  and apply the locally periodic (l-p) two-scale convergence and unfolding method. The  unfolding operator maps functions defined on the $\ve$-dependent perforated domains into functions defined on  the  Cartesian product of two  fixed domains, the  macroscopic domain  and the unit cell. Then, by showing the Cauchy property,  we can  prove the strong convergence of the unfolded sequence, which then implies  the strong l-p two-scale convergence of the original sequence, required to pass to the limit as $\ve \to 0$ in the nonlinear functions in the microscopic problem. In the problem considered here we apply the l-p unfolding operator to prove strong l-p two-scale convergence of $b_s^\ve$ satisfying ODE~\eqref{weak_sol_bs} in the perforated domain $\Omega^\ve_M$. Using the unfolding operator it is possible to consider $b_s^\ve$  defined on $(0,T)\times \Omega^\ve_M$, however, for the simplicity,  we apply the unfolding operator to  an extension of $b^\ve_s$ from $\Omega^\ve_M$ into $\Omega_W$, given as a solution of ODE~\eqref{weak_sol_bs} in $(0,T)\times \Omega_W$, with $b^\ve$ replaced by its extension.

To define a locally periodic approximation for a non-periodic microstructure of the cell walls,  we consider  a partition covering  $\Omega_W \subset \cup_{n=1}^{N_\ve}\overline \Omega_n^\ve$  by a family of  open non-intersecting cubes $\{\Omega_n^\ve\}_{1\leq n\leq  N_\ve}$ of side $\varepsilon^r$, with $0<r<1$,  and
 $\Omega_n^\ve \cap \Omega_W \neq \emptyset$, where $N_\ve$  is the number of  $\Omega_n^\ve$ having a non-empty intersection with $\Omega_W$. Consider also
$
\mathcal K^\ve= \Omega \setminus  \big(\bigcup_{n=1}^{\tilde N_\ve} \overline\Omega_n^\ve\big),
$
where  $\tilde N_\ve$  denotes  the number of cubes $\Omega_n^\ve$ enclosed in $\Omega_W$,  and $\mathcal K^\ve\subset \bigcup_{n=\tilde N_\ve +1}^{N_\ve}\overline \Omega_n^\ve$.
All $\Omega_n^\ve$ with $n=\tilde N_\ve+1, \ldots, N_\ve$ have the non-empty intersection with $\partial \Omega_W$ and are enclosed in a $\ve^r$-neighbourhood of $\partial \Omega_W$, and hence
$
| \mathcal K^\ve| \leq  C \ve^r,$
with some constant $C>0$. This gives $\tilde N_\ve \ve^{rd} \leq |\Omega_W|$ and $N_\ve \ve^{rd} \leq |\Omega_W| + C
$  for $0<\ve \leq 1$. The power $r$,  which determines the size of the local periodicity of the microstructure, is specified below and is defined through the approximation  of the non-periodic microstructure by a locally periodic one.

 For the non-periodic microstructure of the cell walls, described by rotated planes of parallel-aligned microfibrils,  we choose  $\kappa_n \in \mathbb Z^3$, for  $n=1,\ldots, N_\ve$,  such that for $x_{n}^\ve=\ve R_{x_{n}^\ve} \kappa_n $ we have $x_{n}^\ve \in \Omega_n^\ve$, and cover $\Omega_n^\ve$ by  parallelepipeds $Y_{x_n^\ve} = D(x_{n}^\ve) Y$, where  matrix $D(x)\in \mathbb R^3$, for $x \in \Omega_W$, is specified below, i.e.,
$$
\Omega_n^\ve\subset x_{n}^\ve +\bigcup_{j=1}^{I_n^\ve} \ve Y^j_{x_{n}^\ve}, \quad \text{  where   } \; Y^j_{x_{n}^\ve}=D(x_{n}^\ve)(Y+m_j) \; \text{ for }\;  m_j\in \mathbb Z^3.
$$
To specify matrix $D(x)$,  we use the  regularity of $\gamma$  and  the Taylor expansion for $R^{-1}$ around  $x_n^\ve$. Then   for $k_j^n = \kappa_n + m_j$ and $x_{k_j^n}^\ve= \ve R_{x_{k_{j}^{n}}^\ve} k_j^n$, with $1\leq j \leq I_n^\ve$, we obtain
\begin{equation}\label{Taylor_Approx}
\begin{aligned}
 &R^{-1}_{x_{k_j^n}}(x- x_{k_j^n}^\ve) =R^{-1}_{x_{k_j^n}} x- \ve k_j^n \\
 & = R^{-1}_{x^\ve_{n}}x +  (R^{-1}_{x^\ve_n})^\prime x_n^\ve  m_{j,3}\ve +(R^{-1}_{x^\ve_n})^\prime(x- x_{n}^\ve)
 m_{j,3}\ve+ C|m_{j,3} \ve|^2x -  \ve(\kappa_n+ m_j)
 \\ &
= R^{-1}_{x^\ve_n}(x  -x_{n}^\ve)- \tilde W_{x_{n}^\ve}m_j\ve +(R^{-1}_{x^\ve_{n}})^\prime(x- x_{n}^\ve)  m_{j,3}\ve+  C|m_{j,3} \ve|^2x,
\end{aligned}
\end{equation}
where    $\tilde W_{x_{n}^\ve}= \tilde W(x_{n}^\ve)$ with $\tilde W(x)=(I- \nabla R^{-1}(\gamma(x_3)) x)$.  The notation of the gradient is understood as  $\nabla R^{-1}(\gamma(x))x=\nabla_z(R(\gamma(z))x)|_{z=x}$. Thus for  $x\in \Omega_n^\ve$ the distance  between
  $R^{-1}_{x^\ve_n} (x -x_n^\ve)- \tilde W_{x_n^\ve} m_j\ve$ and   $R^{-1}_{x_{k_j^n}}(x-  x_{k_j^n}^\ve)$ is of order $\sup\limits_{1\leq j\leq I_n^\ve} |m_j\ve|^2 \sim \ve^{2r}$.
This  implies that the non-periodic microstructure can by approximated by locally periodic one,  comprising  $ Y_{x_n^\ve}$-periodic microstructure  in each $\Omega_n^\ve$ of side $\ve^r$,  with an appropriately  chosen  $r\in (0,1)$, where     $Y_x= D(x)Y$  with   $D(x)=R(\gamma(x_3))W(x)$ and
\begin{equation}\label{eq:W}
W(x) = \begin{pmatrix}
1 & 0 & 0\\
 0&1 & w(x)\\
 0&0&1
\end{pmatrix}, \; \text{ where } \, \, w(x)=  \gamma^\prime(x_3)\big(\cos (\gamma(x_3))x_1+\sin(\gamma(x_3)) x_2\big).
\end{equation}
Notice that since the characteristic function of the parallel-aligned microfibrils is independent of the first variable, in the definition of $W$ we consider the  shift only for the second variable.
We denote $Z_x=W(x)Y$ and  consider a $Z_x$-periodic function
\begin{equation}\label{charact_l_p}
\hat \vartheta(x,y)= \sum_{k\in \mathbb Z^3} \vartheta(y- W(x) k).
\end{equation}
Then $\{y \in \mathbb R^3 : \hat \vartheta(x,y)=1\}$ is  a set of $Z_x$-periodic cylinders of radius $a$.

Using~\eqref{Taylor_Approx},    together with the estimate below for the characteristic function of a fibre system
\begin{equation}\label{Differ_Charact}
 ||\vartheta_r(x+\tau) - \vartheta_r(x)||^2_{L^2(\Omega)} \leq  C r L |\tau|,
\end{equation}
where $L$ and $r$ denote the fibre length and radius, respectively, see \cite{Briane1} for the proof, and
 observing that in each $\Omega_n^\ve$ the fibre length  and radius are of order $\ve^r$  and $\ve$, respectively, while $N_\ve\leq C \ve^{-3r}$ and  $I_n^\ve\leq C \ve^{3(r-1)}$, we  obtain
 \begin{equation}\label{estim_non_local_1}
\sum\limits_{n=1}^{N_\ve} \int_{\Omega_n^\ve}\sum\limits_{j=1}^{I_n^\ve} \Big| \vartheta_\ve\big(R^{-1}_{x_{k_j^{n}}}(x-x_{k_j^n}^\ve) \big)-
\vartheta_\ve\big(R^{-1}_{x_{\kappa_n}} (x-x_n^\ve) - W_{x_{n}^\ve}m_j\ve \big)\Big|^2 dx\leq  C\ve^{3r-2}.
\end{equation}
Thus, estimates~\eqref{Taylor_Approx} and~\eqref{estim_non_local_1} imply  that the domain with a locally periodic
microstructure, characterised by a  $Y_{x_n^\ve}$-periodic structure   in each
cubes $\Omega_n^\ve$ of size $\ve^r$, where  $r \in (2/3,1)$ and $n=1, \ldots, N_\ve$, provides an appropriate approximation for the original domain with a non-periodic microstructure, characterised by rotated layer of size $\ve$ consisting of parallel-aligned microfibrils. The definition of the characteristic function~\eqref{charact_l_p}  ensures that the periods in the locally periodic  microstructure depend on the rotation angles,
while preserving the natural cylindrical shape of the microfibrils.

To define the l-p two-scale convergence and the unfolding operator, we introduce the notion of a locally periodic approximation of functions   $\psi \in C(\overline\Omega_W; C_{\text{per}}(Y_x))$,  defined by the relation $\psi(x,y) = \tilde \psi(x, D_x^{-1} y)$ for the  corresponding functions $\tilde \psi \in C(\overline\Omega_W; C_{\text{per}}(Y))$, where  $ Y_x=D_x Y$, $D_x = D(x)$, and  $D, D^{-1} \in {\rm Lip}(\mathbb R^d; \mathbb R^{d\times d})$.

A locally periodic approximation  $\mathcal L^\ve: C(\overline\Omega_W; C_{\text{per}}(Y_x))\to L^\infty(\Omega_W)$ of  $\psi \in C(\overline\Omega_W; C_{\text{per}}(Y_x))$  is defined as
\begin{equation}\label{lp_approx}
(\mathcal L^\ve \psi)(x)=    \sum\limits_{n=1}^{N_\ve} \tilde \psi\Big(x, \frac {D^{-1}_{x_n^\ve}(x-\tilde x_n^\ve) }\ve\Big)\chi_{\Omega_n^\ve}(x)
\quad \text{ for } x\in \Omega_W.
\end{equation}
We   consider also  the map $\mathcal L^\ve_0: C(\overline\Omega_W; C_{\text{per}}(Y_x)) \to L^\infty(\Omega_W)$ defined  for $x\in \Omega_W$ as
\[
(\mathcal L^\ve_0 \psi)(x)=\sum\limits_{n=1}^{N_\ve}  \psi\Big(x_n^\ve, \frac {x-\tilde x_n^\ve}\ve\Big)\chi_{\Omega_n^\ve}(x)=
 \sum\limits_{n=1}^{N_\ve} \tilde \psi\Big(x_n^\ve, \frac {D^{-1}_{x_n^\ve}(x-\tilde x_n^\ve)}\ve\Big)\chi_{\Omega_n^\ve}(x).
 \]
 If choosing $\tilde x_n^\ve= D_{x_n^\ve} \ve k$ for some $k \in \mathbb Z^d$, then  the periodicity of $\tilde \psi$ yields, for
 $x\in \Omega_W$,
\[
(\mathcal L^\ve \psi)(x)=
 \sum\limits_{n=1}^{N_\ve} \tilde \psi\Big(x, \frac {D^{-1}_{x_n^\ve}x }\ve\Big)\chi_{\Omega_n^\ve}(x),  \;
  (\mathcal L^\ve_0 \psi)(x)=
 \sum\limits_{n=1}^{N_\ve} \tilde \psi\Big(x_n^\ve, \frac {D^{-1}_{x_n^\ve}x }\ve\Big)\chi_{\Omega_n^\ve}(x),
\]
For $\psi$ in
  $C(\overline\Omega_W; L^q_{\text{per}}(Y_x))$ or  $L^p(\Omega_W; C_{\text{per}}(Y_x))$ the maps $\mathcal L^\ve\psi$ and  $\mathcal L^\ve_0\psi$  are defined analogously. In the proof of the  convergence theorem we shall   use a regular approximation of $\mathcal L^\ve \psi$, defined as
 \[
 (\mathcal L^\ve_{\rho} \psi)(x)= \sum\limits_{n=1}^{N_\ve} \tilde \psi\Big(x, \frac {D^{-1}_{x_n^\ve} x}\ve\Big)\phi_{\Omega_n^\ve}(x)
\quad \text{ for } x\in \Omega_W,
\]
where  $\phi_{\Omega_n^\ve}$  is an approximations of    $\chi_{\Omega_n^\ve}$
such that $\phi_{\Omega_n^\ve} \in C^\infty_0 (\Omega_n^\ve)$ and, for $0<r<\rho<1$,
\begin{equation}\label{ApproxCharactF}
 \sum\limits_{n=1}^{N_\ve}|\phi_{\Omega_n^\ve}  -\chi_{\Omega_n^\ve}| \to 0 \, \text{ in }  L^2(\Omega_W),\, \quad   \,
||\nabla^m \phi_{\Omega_n^\ve}||_{L^\infty(\mathbb R^3)}\leq C \ve^{-\rho m}.
\end{equation}
As an example for $\phi_{\Omega_n^\ve}$, we can consider $\phi_{\Omega_n^\ve}(x)= \varphi_\ve\ast \chi_{\Omega_{n,\rho}^\ve}$  for $\Omega_n^\ve \subset \Omega_W$  and $\phi_{\Omega_n^\ve}(x)=0$  for $\Omega_n^\ve\cap \partial \Omega_W\neq \emptyset$, where
$\varphi_{\ve}(x)= \frac 1{\ve^{n\rho}}\varphi(\frac x {\ve^\rho})$, with  $\varphi(x)=c\exp(-1/(1-|x|^2))$ for $|x|<1$ and $\varphi(x)=0$ for $|x|\geq 1$ and $\Omega_{n,\rho}^\ve=\{x\in \Omega_n^\ve, \text{dist}( x, \partial\Omega_n^\ve)> \ve^\rho\}$ and  $c$ is such that $\int_{\mathbb R^d} \varphi(x) dx =1$. Then $\phi_{\Omega_n^\ve} \in C^\infty_0 (\Omega_n^\ve)$  and~\eqref{ApproxCharactF} follow from the properties  of $\varphi_{\ve}$. Notice
$\|\sum_{n=1}^{N_\ve}|\phi_{\Omega_n^\ve}  -\chi_{\Omega_n^\ve}|\|_{L^2(\Omega_W)}\leq C_1 \ve^{r(d-1)} \ve^\rho N_\ve\leq C_2 \ve^{\rho-r}$. A different construction of $\phi_{\Omega_n^\ve}$ can be found in \cite{Briane1}.

\begin{definition}\label{def_two-scale}\cite{Ptashnyk_2013}
A sequence $\{u^\ve\} \subset  L^2(\Omega_W)$ is said to converge locally periodic two-scale (l-p two-scale) to $u \in L^2(\Omega_W; L^2(Y_x)) $ as $\ve \to 0$ if
\begin{eqnarray*}
\lim\limits_{\ve \to 0}\int_{\Omega_W} u^\ve(x) \mathcal L^\ve\psi(x) dx = \int_{\Omega_W}   \ddashinttt_{Y_x}  u(x,y) \psi(x, y)  dy dx,
\end{eqnarray*}
for any $\psi \in L^2(\Omega_W; C_{\text{per}}(Y_x))$, where $\mathcal L^\ve\psi $ is the locally periodic approximation of $\psi$.
\end{definition}
The  convergence results for the locally periodic approximation  $\mathcal L^\ve$  ensure  that Definition~\ref{def_two-scale} does not depend on the choice of  $x_n^\ve, \tilde x_n^\ve \in \Omega_n^\ve$, for $n=1, \ldots, N_\ve$, in the definition of $\mathcal L^\ve$, see~\cite{Ptashnyk_2013} for details.

\begin{theorem}\label{lp_covregence}\cite{Ptashnyk_2013}
Let $\{u^\ve\}$ be a bounded sequence in $L^2(\Omega_W)$. Then there exists a subsequence  of $\{u^\ve\}$, denoted again by $\{u^\ve\}$, and  $u \in L^2(\Omega_W; L^2(Y_x))$, s.t.~$u^\ve \to u$ l-p two-scale as $\ve \to 0$. \\
 Let $\{u^\ve\} $ be a bounded sequence in $H^1(\Omega_W)$ that converges weakly to $u$ in $H^1(\Omega_W)$. Then
\begin{itemize}
\item   the sequence $\{u^\ve\}$ converges l-p two-scale to $u$;
\item there exist a subsequence of $\{\nabla u^\ve \}$,  not relabeled, and  $u_1\in L^2(\Omega_W; H^1_{\text{per}}(Y_x)/\mathbb R)$  such that $\nabla u^\ve \to \nabla u + \nabla_y u_1$  l-p two-scale  as $\ve \to 0$.
 \end{itemize}
\end{theorem}

Analogously to the periodic case, see e.g.~\cite{CDG_book},  in the definition of the locally periodic unfolding operator we  consider the subdomains composed of unit cells contained in the corresponding domains
$\Omega_n^\ve$ of a partition  of $\Omega_W$,
\begin{equation}\label{Omega_hat}
 \hat \Omega^\ve= \bigcup_{n=1}^{N_\ve} \hat \Omega_n^\ve, \;  \text{ with }  \hat \Omega_n^\ve =\text{Int} \Big( \bigcup_{\xi \in \hat \Xi_n^\ve} \ve D_{x_n^\ve}(\overline{Y}+ \xi)\Big), \; \; \Lambda^\ve= \bigcup_{n=1}^{N_\ve} \Lambda^\ve_n  \cap \Omega_W,
\end{equation}
where $\Lambda^\ve_n = \Omega_n^\ve \setminus \hat \Omega_{n}^\ve$ and $\hat \Xi^\ve_n =\{\xi \in \Xi_n^\ve \; : \,  \, \ve D_{x_n^\ve}(Y+ \xi) \subset (\Omega_n^\ve \cap \Omega_W)\}$. Then for $Y^\ast = Y \setminus W(x)^{-1}\overline Y_0$ and $Y_x^\ast =D_x Y^\ast$, with $\overline Y_0 \subset Y$ and $W(x)^{-1}\overline Y_0\subset Y$ for $x \in \Omega_W$, the domain with local variations in the distribution of perforations is defined by
$$\Omega^\ast_{\ve} = \text{Int} \big(\bigcup_{n=1}^{N_\ve} \Omega_n^{\ast, \ve}\big)  \cap \Omega_W, \qquad \text{ where } \qquad
 \Omega_n^{\ast, \ve}  =  \bigcup_{\xi \in \hat \Xi_{n}^{\ve}} \ve D_{x_n^\ve}(\overline{Y^\ast} + \xi)\cup \overline \Lambda_{n}^{\ve}. $$
\begin{definition} \label{l-p-unf-oper} \cite{Ptashnyk_2015} For any  Lebesgue-measurable on $\Omega_W$  function $\psi$
the locally periodic (l-p)  unfolding operator
$\mathcal T_{\mathcal L}^\ve$ is defined as
\begin{eqnarray*}
\mathcal T^\ve_{\mathcal L} (\psi) (x,y) = & \sum\limits_{n=1}^{N_\ve}
\psi\big( \ve D_{x_n^\ve} \big[{D^{-1}_{x_n^\ve} x} /\ve \big]_{Y} +  \ve D_{x_n^\ve}  y \big) \chi_{\hat \Omega_{n}^\ve}(x)  \quad \text{ for } \, \, x \in \Omega_W \text{ and }   \, \,  y \in Y.
\end{eqnarray*}
 For any  Lebesgue-measurable on $\Omega_\ve^\ast$  function $\psi$
the l-p  unfolding operator $\mathcal T_{\mathcal L}^{\ast, \ve} $ is defined as
\begin{eqnarray*}
\mathcal T^{\ast, \ve}_{\mathcal L} (\psi) (x,y) = &
\sum\limits_{n=1}^{N_\ve} \psi\big( \ve D_{x_n^\ve} \big[ {D^{-1}_{x_n^\ve} x}/ \ve\big]_{Y} +  \ve D_{x_n^\ve}  y \big) \chi_{\hat \Omega_{n}^\ve}(x)  \quad \text{ for } \, \, x \in \Omega_W \text{ and }   \, \,  y \in Y^\ast.
\end{eqnarray*}
\end{definition}

Analogously to the periodic case \cite{CDG, CDDGZ, CDG_book},
the following compactness results hold for the l‑p unfolding operator.
\begin{theorem}\label{prop_conver_1}\cite{Ptashnyk_2015}
For a sequence $\{w^\ve \}\subset L^p(\Omega_W)$, with $p \in (1, +\infty)$,  satisfying
$$
\| w^\ve \|_{L^p(\Omega_W)} +
\ve \|\nabla w^\ve \|_{L^p(\Omega_W)} \leq C,$$
there exists $w \in L^p(\Omega_W; W^{1,p}_{\rm per}(Y_x))$ such that, upto a subsequence,
\begin{eqnarray*}
\begin{aligned}
 \T_{\mL}^\ve(w^\ve) & \; \rightharpoonup && w(\cdot, D_x \cdot) \quad\hspace{2.1 cm } && \text{ weakly in } \, L^p(\Omega_W; W^{1,p}(Y)), \\
 \ve  \T_{\mL}^\ve(\nabla w^\ve) &\; \rightharpoonup  && D_x^{-T}\nabla_y w(\cdot, D_x \cdot)  \quad && \text{ weakly in } \, L^p(\Omega_W\times Y).
\end{aligned}
\end{eqnarray*}
For   $\{w^\ve\}\subset W^{1,p}(\Omega_W)$ converging weakly to $w\in W^{1,p}(\Omega_W)$, there exists $w_1\in L^p(\Omega_W; W^{1,p}_{\rm per}(Y_x))$, such that
\begin{eqnarray*}
\begin{aligned}
 \mathcal T_{\mathcal L}^{\ve}(w^\ve) &\;  \rightharpoonup  &&  w   \quad && \text{  weakly in } \, \, L^p(\Omega_W; W^{1,p}(Y)), \\
 \mathcal T_{\mathcal L}^{\ve}(\nabla w^\ve)(\cdot, \cdot) &\;  \rightharpoonup  && \nabla_x w(\cdot) +  D_x^{-T} \nabla_y w_1(\cdot, D_x \cdot) \quad && \text{ weakly in } L^p(\Omega_W\times Y).
\end{aligned}
\end{eqnarray*}
For a sequence $\{w^\ve\} \subset W^{1,p}(\Omega^\ast_\ve)$, where $p\in (1, +\infty)$, satisfying
\begin{equation}\label{extim_w_ve}
\| w^\ve \|_{L^p(\Omega^\ast_\ve)}  + \ve \|\nabla w^\ve \|_{L^p(\Omega^\ast_\ve)} \leq C,
\end{equation}
there   exists $ w \in L^p(\Omega_W; W^{1,p}_{\rm per}(Y^\ast_x))$ such that, upto a subsequence,
\begin{equation}\label{weak_conver_1}
\begin{aligned}
\mathcal T_{\mathcal L}^{\ast, \ve} (w^\ve) &\;  \rightharpoonup &&  w(\cdot, D_x\cdot)  && \quad  \text{ weakly in }   \quad  L^p(\Omega_W; W^{1,p}(Y^\ast)), \\
  \ve \mathcal T_{\mathcal L}^{\ast, \ve} (\nabla w^\ve) & \;  \rightharpoonup && D_x^{-T} \nabla_y  w(\cdot, D_x\cdot) && \quad  \text{ weakly in }  \quad   L^p(\Omega_W; L^p(Y^\ast)).
\end{aligned}
\end{equation}
For a sequence $\{ w^\ve \} \subset W^{1, p}(\Omega^\ast_\ve)$, where $p\in (1, +\infty)$, satisfying
$$
\|w^\ve \|_{W^{1, p}(\Omega^\ast_\ve)} \leq C,
$$
there exist  $w\in W^{1,p}(\Omega_W)$ and $w_1\in L^p(\Omega_W; W^{1,p}_{\rm per}(Y^\ast_x))$ such that, upto a subsequence,
\begin{eqnarray*}
\begin{aligned}
\mathcal T^{\ast,\ve}_{\mL}(w^\ve) & \; \rightharpoonup && w \quad\qquad  &&   \text{ weakly in }   L^p(\Omega_W; W^{1,p}(Y^\ast)), \\
  \mathcal T^{\ast,\ve}_{\mL}(\nabla w^\ve)& \;  \rightharpoonup && \nabla w + D_x^{-T} \nabla_y w_1(\cdot, D_x \cdot)\quad &&  \text{ weakly in }   L^p(\Omega_W; L^p(Y^\ast)), \\
\mathcal T^{\ast,\ve}_{\mL}(w^\ve) & \; \to &&  w \quad\qquad &&  \text{ strongly in }   L^p_{\rm loc}(\Omega_W; W^{1,p}(Y^\ast)).
\end{aligned}
\end{eqnarray*}
\end{theorem}
We also have  the equivalence between the l-p two-scale convergence and  weak convergence of the unfolded sequence.
\begin{lemma}\label{l-t-s-l-p-eq}\cite{Ptashnyk_2015}
 Let $\{w^\ve\}$ be a bounded sequence in $L^p(\Omega_W)$, where  $p \in (1, +\infty)$.
 Then the following assertions  are equivalent
\begin{itemize}
\item[ (i)] $\quad \qquad \qquad w^\ve \to  w  \qquad \qquad  \; \text{ l-p two-scale},  \quad \qquad w \in L^p(\Omega_W; L^p(Y_x)),$
\item[(ii)] $ \quad\;  \T_{\mL}^{\ve}(w^\ve)(\cdot, \cdot)  \rightharpoonup  w(\cdot, D_x\cdot)  \quad \text{weakly  in } L^p(\Omega_W\times Y).$
\end{itemize}
\end{lemma}

Combining the a priori estimates from ~Lemma~\ref{estim:apriori} with compactness properties of  l-p two-scale convergence, see Theorem~\ref{lp_covregence}, and the  extensions, not relabeled, of $p^\ve_\alpha$,  $c^\ve$, $b^\ve$, and $b^\ve_s$ from $\Omega_M^\ve$ into $\Omega$,  with $\alpha = E,e, d$, given by~Lemma~\ref{extension_local},  we obtain the following convergence results.
\begin{lemma}\label{converg_2}
 There exist functions $\bu \in L^\infty(0,T; \cW(\Omega))$, $\bu_1 \in L^2(\Omega_{W,T}; H^1_{\rm per}(Y_x)/\mathbb R)^3)$,  $b_s \in L^\infty(\Omega_{ML,T})$, $b_s \in L^\infty(\Omega_{W,T}\times  Y_x)$,   $p_\alpha, c, b \in L^2(0,T; H^1(\Omega))\cap L^\infty(\Omega_T)$, and $p_{\alpha,1}, c_1, b_1 \in L^2(\Omega_{W,T}; H^1_{\rm per}(Y_x)/\mathbb R)$,  with $\alpha = e,E,d$,   such that, up to a subsequence,
 \begin{equation*}
 \begin{aligned}
 & \bu^\ve  \rightharpoonup  \bu \quad   \text{ weakly}^\ast \text{ in }  L^\infty(0,T;  \cW(\Omega)),  &&    \text{strongly in }  L^2(\Omega_T) \text{ and }  L^2((0,T)\times \partial \Omega),  \\
 &p_\alpha^\ve  \rightharpoonup  p_\alpha,\;  c^\ve  \rightharpoonup  c ,\;  b^\ve  \rightharpoonup b     &&  \text{weakly in }    L^2(0,T;H^1(\Omega)),  \\
 & &&   \text{strongly in } L^2(\Omega_T) \text{ and }  L^2((0,T)\times \partial \Omega), \\
  & \nabla \bu^\ve  \rightharpoonup \nabla \bu +  \nabla_y \bu_1,  \; \;\;\quad  \nabla b^\ve  \rightharpoonup \nabla b +  \nabla_y b_1 \hspace{-0.5 cm} &&   \hspace{2 cm}   \text{l-p two-scale},\\
 &\nabla p_\alpha^\ve  \rightharpoonup \nabla p_\alpha +  \nabla_y p_{\alpha,1}, \;
 \nabla c^\ve  \rightharpoonup \nabla c +  \nabla_y c_1   \hspace{-0.5 cm}  &&  \hspace{2 cm}   \text{l-p two-scale}, \\
 & b_s^\ve   \rightharpoonup  b_s   \qquad   \text{ weakly  in }  L^q(\Omega_{ML, T}),   &&  b_s^\ve   \rightharpoonup  b_s   \qquad   \text{l-p two-scale},
 \end{aligned}
 \end{equation*}
 as $\ve \to 0$, where $\alpha=e,E,d$, and any $1< q < \infty$.
 \end{lemma}
\begin{proof}
The a priori estimates   in Lemma~\ref{estim:apriori} and  Theorem~\ref{lp_covregence}  imply the weak and l-p two-scale convergence.
Additionally, the equicontinuity with respect to the time variable, ensured by the a priori estimates, and the compactness result in~\cite{Simon} yield the strong convergence in $L^2((0,T)\times \Omega)$ and $L^2((0,T)\times \partial \Omega)$.
\end{proof}

Using the convergence results in Lemma~\ref{converg_2}, we first derive the macroscopic model for elastic deformations of  plant tissue.
\begin{theorem}
A sequence of  solutions $\{\bu^\ve\}\subset L^\infty(0,T; \cW(\Omega))$ of microscopic model~\eqref{EQST}  converges to a   solution $\bu \in  L^\infty(0,T; \cW(\Omega))$ of the macroscopic problem
\begin{equation}\label{eq:macro_u}
\begin{aligned}
&-{\rm div} ( \mathbb E^{\rm hom}(x, b) \be(\bu)) ={\bf 0} && \text{ in } \; \Omega, \\
& \mathbb E^{\rm hom}(x, b) \be(\bu) \cdot \nu = - P(t,x) \nu && \text{ on } \; \Gamma_\cI,\\
& \mathbb E^{\rm hom}(x, b) \be(\bu) \cdot \nu =  \bff(t,x) && \text{ on } \; \Gamma_\cE,
\end{aligned}
\end{equation}
where  the  homogenized elasticity tensor is given by $\bbE^{\text{hom}}(x, b)= \bbE^{\text{hom}}_W(x,b)\chi_{\Omega_W}(x) + \bbE_{M_l}(b)\chi_{\Omega_{ML}}(x)$ and
$$
\bbE^{\text{hom}}_{W, ijkl}(x, b)
=\ddashinttt_{\hat Z_x}\!\! \big( \tilde{\mathbb B}_{ijkl}(x,\hat y, b) +\tilde{\mathbb B}(x,\hat y, b) \hat \be^R_{\hat y}(\tilde \omega_{ij})_{kl} \big) d\hat y, \quad \text{ for } 1\leq i,j,k,l \leq 3,
$$
with $\tilde{\mathbb B}(x, y, b)= \bbE_f\hat \vartheta(x, y) +\bbE_{M_w}(b)(1- \hat \vartheta(x,  y))$ and $\hat y= (y_2,  y_3)$, $\hat Z_x = Z_x \cap \{ y_1 =0\}$, and $\hat \vartheta$ is as in~\eqref{charact_l_p}. Here $\tilde \omega_{ij}$, for $1\leq i,j \leq 3$, are solutions of the cell problems~\eqref{UnitCellNonPer}.
Note that, since $\hat \vartheta$ does not depend on $y_1$, we identify $\tilde{\mathbb B}(x, y, b)$ with $\tilde{\mathbb B}(x,\hat y, b)$.
Additionally,
\begin{equation}\label{strong_conv_u}
\begin{aligned}
 \be(\bu^\ve) &\to \be(\bu) +\be_y(\bu_1) \quad && \text{strongly l-p two-scale}, \\
 \cT_\mL^\ve(\be(\bu^\ve))&\to \be(\bu) +\be_y(\bu_1) \quad && \text{strongly in } \; L^2((0,T)\times \Omega_W\times Y), \\
 \be(\bu^\ve) &\to \be(\bu) \quad && \text{strongly in } \; L^2((0,T)\times \Omega_{ML}).
 \end{aligned}
\end{equation}
\end{theorem}

\begin{proof}
To derive the macroscopic problem using the locally-period (l-p) two-scale convergence method, we  rewrite  equation~\eqref{weak_sol_u}  in the form
\begin{equation}\label{aprox_eq_u}
\begin{aligned}
\big\langle\bbE^\ve_{N^\ve}(x, b^\ve) \be(\bu^\ve), \be (\phi)\big\rangle_\Omega  + \big\langle[\bbE^\ve(x, b^\ve)- \bbE^\ve_{N^\ve}(x, b^\ve)] \be(\bu^\ve) ,\be (\phi) \big\rangle_\Omega
\\
= \langle \bff, \phi \rangle_{\Gamma_{\cE}} - \langle P \nu,  \phi \rangle_{\Gamma_{\cI}},
 \end{aligned}
\end{equation}
for $\phi \in L^2(0,T; C^1(\overline\Omega)\cap \cW(\Omega))$, where
$$\bbE^\ve_{N_\ve}(x, b^\ve) = (\mathcal L^\ve_0 \mathbb{B})(x, b^\ve)\, \chi_{\Omega_W}(x) +  \bbE_{M_l}(b^\ve) \, \chi_{\Omega_{ML}}(x)$$
and $(\mathcal L^\ve_0 \mathbb{B})(\cdot, b^\ve)$ is a locally periodic approximation of the tensor
$$\mathbb{B}(x,y, b^\ve)=  \bbE_f\hat \vartheta(x,R^{-1}_{x}y) +\bbE_{M_w}(b^\ve)(1- \hat \vartheta(x,R^{-1}_{x}y)) \quad \text{  for } x\in \Omega,  \, y\in Y_x.
$$
Here $Y_x= D(x)Y$,   with   $D(x)=R(\gamma(x_3))W(x)$ given by~\eqref{eq:W}, and $R_{x}:= R(\gamma(x_3))$. Note that assumptions on  $\bbE_{M_w}$ and $\gamma$ and the definition of the characteristic function $\hat \vartheta$ imply   that $\mathbb{B}(\cdot, \cdot, b) \in C(\overline\Omega, L^p_{\rm{per}}(Y_x))$ for $1< p < \infty$,  $B(x,y, \cdot) \in C^1(\mathbb R)$ for $x\in \Omega$,  $y\in Y_x$, and $|B(x,y,b)|\leq M_B$  for $x\in \Omega$,  $y\in Y_x$,  $b \in L^\infty((0,T)\times \Omega)$, and some $M_B>0$.

 To estimate $\bbE^\ve(x, b^\ve)- \bbE^\ve_{N^\ve}(x, b^\ve)$,  we use~\eqref{estim_non_local_1}, together with    the definition of  $\bbE^\ve$,  $\bbE^\ve_{N^\ve}$,     and  $
\chi_{\Omega_f^\ve}(x)$,  and obtain
\begin{equation}\label{estimnon-perLoc}
\int_{\Omega} \big|\big(\bbE^\ve(x, b^\ve) - \bbE^\ve_{N_\ve}(x, b^\ve)\big) \be(\bu^\ve) \be (\phi)\big| dx\leq C \ve^{\frac{3r}2 -1} \|\bu^\ve\|_{\cW(\Omega)} \|\phi \|_{W^{1,\infty}(\Omega)},
 \end{equation}
for $\phi \in C^1(\overline \Omega)$ and $t \in (0,T)$,  which, since $\{\bu^\ve\}$ is bounded in $L^\infty(0,T; \cW(\Omega))$,   for  $2/3 <r<1$ converges to zero as $\ve \to 0$.

Considering $ \psi(t,x)= \psi_1(t,x) + \ve(\mathcal L^\ve_\rho\psi_2)(t,x)$,  where  $\psi_1\in C^1_0(0,T; C^1(\overline\Omega)\cap \cW(\Omega))$ and  $\psi_2\in C^{1}_0((0,T)\times\Omega_W; C^1_{\text{per}}(Y_x))$, extended by zero into $\Omega_{ML}$,
as a  test function in~\eqref{aprox_eq_u}, we rewrite the first term as
$$
\big\langle \bbE^\ve_{N^\ve}(x, b^\ve)\,  \be(\bu^\ve),  \be(\psi_1)+
\ve \mL^\ve \be(\psi_2) + \ve \sum\limits_{n=1}^{N_\ve}\big(\be(\psi_2^{n})[\phi_{\Omega_n^\ve} -  \chi_{\Omega_n^\ve}]
+   \psi_2^{n} \odot \nabla \phi_{\Omega_n^\ve}\big) \big\rangle_{\Omega_T},
$$
 where $\psi_2^n(t,x)=\tilde\psi_2 (t,x,{R^{-1}_{x_n^\ve} x/\ve})$ and $a\odot b= \big(\frac 12(a_i b_j+ a_j b_i)\big)_{1\leq i, j\leq 3}$  for $a, b \in \mathbb R^3$.  Using the regularity of $\psi_2$, strong convergence of $b^\ve$, and  $\ve \nabla \psi_2^{n}=\ve\nabla_x \psi_2^{n}+ R^{-T}_{x_n^{\ve}}\nabla_{\tilde y} \psi_2^n$,   together with the  results for the l-p two-scale convergence, see~\cite{Ptashnyk_2013}, we obtain the strong l-p two-scale convergence of $\ve \mL^\ve \be(\psi_2)$ and $\mathcal L^\ve_0 \mathbb B(\cdot, b^\ve)$ to
$\be_y(\phi_2)$ and $\mathbb B(\cdot, \cdot, b)$, respectively. Then, since $b^\ve$ is bounded, it follows that  $\Phi^\ve(t,x)=\mathcal L^\ve_0 \mathbb B(x, b^\ve) (\be(\psi_1(t,x))+ \ve \mL^\ve \be(\psi_2)(t,x) )
$ is  bounded in  $L^q(\Omega_T)$, for any $1<q<\infty$, and converges strongly l-p  two-scale to $\mathbb B(x,y, b)(\be(\psi_1(t,x)) + \be_y(\psi_2(t,x,y)))$.
The convergence \eqref{ApproxCharactF} of  $\phi_{\Omega_n^\ve}$ in $L^2(\Omega)$, the   regularity of $\psi_2$,  the estimates for $\bu^\ve$, and  the boundedness of $\bbE^\ve_{N^\ve}(x, b^\ve)$, ensured by the  boundedness of $b^\ve$,   yield that
\begin{eqnarray*} 
 && \lim\limits_{\ve\to 0}  \ve \! \int_0^T\!\!\!\int_{\Omega }  \bbE^\ve_{N^\ve}(x, b^\ve) \,   \be(\bu^\ve)
\sum\limits_{n=1}^{N_\ve} \be(\psi_2^{n}) (\phi_{\Omega_n^\ve} - \chi_{\Omega_n^\ve}) dxdt =0.
\end{eqnarray*}
 The property  $||\nabla \phi_{\Omega_n^\ve}||_{L^\infty(\mathbb R^d)}\leq C \ve^{-\rho}$, where
$0<r<\rho<1$, ensures
\begin{equation*} 
 \lim\limits_{\ve\to 0} \ve \! \int_0^T \!\!\!\int_{\Omega }\bbE^\ve_{N^\ve}(x, b^\ve) \, \be(\bu^\ve)  \sum\limits_{n=1}^{N_\ve} \psi_2^{n} \odot\nabla \phi_{\Omega_n^\ve}  dxdt=0.
\end{equation*}
Thus,  using the l-p two-scale convergence of  $\nabla \bu^\ve$, the strong  convergence of $b^\ve$,  the strong l-p two-scale convergence of $\Phi^\ve$, and estimate~\eqref{estimnon-perLoc}, yields, as $\ve \to 0$,
\begin{equation}\label{macro_1}
\begin{aligned}
\big\langle  |Y_x|^{-1} \mathbb B(x, y, b) \big(\be(\bu) + \be_y(\bu_1)\big), \be(\psi_1)+
\be_y(\psi_2)\big\rangle_{\Omega_{W,T},  Y_x}  \qquad \\
+ \big\langle \bbE_{M_l}(b)\be(\bu), \be(\psi_1)\big\rangle_{\Omega_{ML, T}}
  = \langle \bff, \psi_1\rangle_{\Gamma_{\cE, T}}- \langle  P\nu, \psi_1\rangle_{\Gamma_{\cI, T}}.
\end{aligned}
\end{equation}
Considering $\psi_2=0$, we obtain
\begin{equation}\label{macro_13}
\begin{aligned}
 & \big\langle | Y_x|^{-1} \mathbb B(x, y, b) \big(\be(\bu) + \be_y(\bu_1)\big), \be(\psi_1)\big\rangle_{\Omega_{W,T},  Y_x}
 \\
 & +  \big\langle \bbE_{M_l}(b)\be(\bu), \be(\psi_1)\big\rangle_{\Omega_{ML,T}}
 =  \langle \bff, \psi_1\rangle_{\Gamma_{\cE,T}}- \langle  P \nu , \psi_1\rangle_{\Gamma_{\cI, T}}.
 \end{aligned}
\end{equation}
Choosing $\psi_1=0$ implies the equation for $\bu_1(t,x,y)$, for $(t,x) \in (0,T)\times \Omega_W$ and $y \in Y_x$,
\begin{equation}\label{macro_12}
 \big\langle  |Y_x|^{-1} \mathbb B(x, y, b) \big(\be(\bu) + \be_y(\bu_1)\big),
\be_y(\psi_2)\big\rangle_{\Omega_{W, T},  Y_x}=0.
\end{equation}
The  transformation $\mathcal F:  Y_x \to Z_x$, i.e. $\tilde y= \mathcal F(y)=R_{x}^{-1}y$, in \eqref{macro_12} gives
\begin{eqnarray*}
&&\big\langle |Z_x|^{-1} \tilde {\mathbb B}(x,\tilde y, b) \big(\be(\bu) + \be^R_{\tilde y}(\tilde \bu_1)
\big), \be^R_{\tilde y}(\tilde \psi_2)\big \rangle_{\Omega_{W,T}, Z_x}
= 0,
\end{eqnarray*}
with
$\tilde{\mathbb B}(x,\tilde y, b)= \bbE_f\hat \vartheta(x,\tilde y) +\bbE_{M_w}(b)(1- \hat \vartheta(x, \tilde y))$, and
\begin{equation}\label{def_trans_grad}
\be^R_{\tilde y,kl}(v)=\frac 12 \Big(\big[{R}^{-T}_{x} \nabla_{\tilde y}  v^l\big]_k
+\big[R^{-T}_{x} \nabla_{\tilde y}  v^k\big]_l \Big).
\end{equation}
Using  linearity of the problem we  conclude that  $\tilde \bu_1$  has the form
\begin{equation}\label{eq:u1}
\tilde \bu_1(t,x,\tilde y)= \frac 12\sum\limits_{i,j=1}^3 \Big( \frac{\partial \bu^i(t,x)}{\partial x_j} + \frac{\partial \bu^j(t,x)}{\partial x_i} \Big)
\tilde \omega_{ij}(t,x,\tilde y),
\end{equation}
where $\tilde \omega_{ij}(x,\tilde y)$ are solutions of the  unit cell problems
\begin{equation}\label{unit_loc_per}
\begin{aligned}
-\nabla_{\tilde y}\cdot \big(
  {R}_{x}^{-1} \tilde {\mathbb B} (x, \tilde y, b)  \be^R_{\tilde y}(\tilde \omega_{ij}) \big)
  =  \nabla_{\tilde y}\cdot \big({R}_{x}^{-1} \tilde{\mathbb B}(x,\tilde y, b) l_{ij} \big) & \quad \text{ in }  Z_x, \\
 \tilde{\omega}_{ij} \hspace{1.0 cm} \text{ periodic } & \quad \text{ in }  Z_x,
 \end{aligned}
\end{equation}
with $l_{ij}= \frac 12 (l_i \otimes l_j + l_j \otimes l_i)$,  where $\{l_i\}_{1\leq i\leq 3}$ is the canonical basis in $\mathbb R^{3}$.
Notice that $\hat\vartheta$  is independent of $\tilde y_1$ and  solutions of~\eqref{unit_loc_per} are unique up to a constant. Thus,  $\tilde\omega_{ij}$, for $i,j=1,2,3$, do not depend on $\tilde y_1$ and  the  unit cell problems~\eqref{unit_loc_per} can be reduced to the two-dimensional problems
\begin{equation}\label{UnitCellNonPer}
\begin{aligned}
 -\nabla_{\hat y}\cdot \big(
  \hat R_{x}^{-1} \tilde{\mathbb  B}(x,\hat y,b) \hat \be^R_{\hat y} (\tilde \omega_{ij}) \big)
 =  \nabla_{\hat y}\cdot \big(\hat R_{x}^{-1}  \tilde{\mathbb B}(x,\hat y,b) l_{ij} \big) &&  \text{ in } \hat Z_x,\\
\tilde{\omega}_{ij} \hspace{1.0 cm} \text{ periodic } &&   \text{ in } \hat Z_x,
\end{aligned}
\end{equation}
where  $\hat y=(\tilde y_2, \tilde y_3)$     and the matrices   $\hat \be^R_{\hat y}(v)$ and $\hat R_{x}^{-1}$  are given by
 \begin{equation*}
\hat \be_{\hat y,kl}^R(v)=\frac 12 \Big[\big(\hat {R}^{-T}_{x}
\nabla_{\hat y} v^l\big)_k +\big(\hat {R}^{-T}_{x} \nabla_{\hat y} v^k\big)_l \Big], \;\;
\hat{R}^{-1}_{x}= \begin{pmatrix}
-\sin(\gamma(x_3)) &
 \cos (\gamma(x_3))&  0 \\
 0&  0& 1
\end{pmatrix}.
\end{equation*}
Taking into account  the transformation $\mathcal F$ and the ansatz~\eqref{eq:u1} for $\tilde\bu_1$ in \eqref{macro_13} and choosing  first $\psi_1 \in C^1_0((0,T)\times \Omega)$ and then $\psi_1 \in C^1([0,T]\times \overline\Omega)$, with $\psi_1 =0$  on  $(0,T)\times \Gamma_{\cI}$ and subsequently on $(0,T)\times (\partial\Omega \setminus\Gamma_{\cI})$,  we obtain the macroscopic problem~\eqref{eq:macro_u}.
Assumption~\ref{assumptions}.5 on $\bbE_f$, $\bbE_{M_w}$, and $\bbE_{M_l}$, ensure that $\bbE^{\rm hom}$  possesses major and minor symmetries,  satisfies the ellipticity property on the set of symmetric matrices, and is bounded since  $b$ is bounded. Then, using the standard arguments, we obtain uniqueness of solutions of the macroscopic problem~\eqref{eq:macro_u}, for a given $b \in L^\infty((0,T)\times \Omega)$.

In order to pass to the limit in the nonlinear functions $Q_l(b^\ve, \be(\bu^\ve))$, for $l=b,s$,  we now prove the strong l-p two-scale convergence of $\be(\bu^\ve)$.  Using $\bu^\ve$ as a test function    in~\eqref{weak_sol_u}
and combining   the lower-semicontinuity of the norm, the weak and l-p two-scale convergence of $\bu^\ve$, the properties  of $\bbE^\ve(x, b^\ve)$, and the boundedness  and strong convergence of $b^\ve$, which imply the l-p two-scale convergence of $\big(\bbE^\ve(x, b^\ve)\big)^{\frac 12} \be(\bu^\ve)$ in $(0,T)\times \Omega_W$  and the weak convergence of
$\big(\bbE_{M_l}(b^\ve)\big)^{\frac 12}\be(\bu^\ve)$ in $(0,T)\times \Omega_{ML}$,  we obtain
\begin{equation}
\begin{aligned}
 \big\langle |Y_x|^{-1}  \mathbb B(x, y, b) (\be(\bu) +   \be_{y}(\bu_1)), \be(\bu) +  \be_{y}(\bu_1) \big\rangle_{\Omega_{W,T},  Y_x}  + \big\langle \mathbb E_{M_l}(b) \be(\bu) , \be(\bu)\big\rangle_{\Omega_{ML,T}}
\\\leq \liminf_{\ve \to 0 } \, \langle \bbE^\ve(x, b^\ve) \be(\bu^\ve), \be(\bu^\ve) \rangle_{\Omega_{T}}
 \leq \limsup_{\ve \to 0 } \, \langle \bbE^\ve(x, b^\ve) \be(\bu^\ve), \be(\bu^\ve)  \rangle_{\Omega_T}
 \\ =
\lim\limits_{\ve \to 0} \big[\langle  \bff, \bu^\ve  \rangle_{\Gamma_{\cE, T}}  - \langle  P \nu, \bu^\ve  \rangle_{\Gamma_{\cI, T}} \big]
 =
\langle  \bff, \bu \rangle_{\Gamma_{\cE, T}}  - \langle  P \nu, \bu  \rangle_{\Gamma_{\cI, T}} .
\end{aligned}
\end{equation}
Taking $\bu$ and $\bu_1$ as  test functions  in  \eqref{macro_13} and \eqref{macro_12} respectively, and adding the resulting equations,  yields
 \begin{equation*}
\begin{aligned}
 \langle  \bff, \bu \rangle_{\Gamma_{\cE, T} }  - \langle  P \nu, \bu \rangle_{\Gamma_{\cI, T}} =
\big \langle |Y_x|^{-1} \mathbb B(x, y, b) (\be(\bu) +  \be_y(\bu_1)), \be(\bu) + \be_y(\bu_1)\big \rangle_{\Omega_{W,T},  Y_x} \\
+ \big\langle \mathbb E_{M_l}(b) \be(\bu) , \be(\bu)  \big\rangle_{\Omega_{ML,T}}.
 \end{aligned}
\end{equation*}
Hence we obtain
\begin{equation}\label{strong_ts_1}
\begin{aligned}
 & \lim_{\ve \to 0 } \langle \bbE^\ve(x, b^\ve) \be(\bu^\ve), \be(\bu^\ve) \rangle_{\Omega_T} = \big\langle \mathbb E_{M_l}(b) \be(\bu) , \be(\bu)  \big\rangle_{\Omega_{ML,T}}\\
&  \quad
+ \big\langle |Y_x|^{-1}  \mathbb B(x, y, b) (\be(\bu) +  \be_{y}(\bu_1)), \be(\bu) +  \be_{y}(\bu_1) \big\rangle_{\Omega_{W,T},  Y_x}.
 \end{aligned}
\end{equation}
Combining the positive definiteness   and boundedness of $\bbE^\ve$ and $\mathbb B$,   the strong convergence of $b^\ve$,   convergence~\eqref{estim_non_local_1}, and  the definitions of $\bbE^\ve$ and $\mathbb B$,  we obtain
$$
\begin{aligned}
& \big\||Y_x|^{-\frac 12}\big[ \be(\bu^\ve) - (\be(\bu) +  \be_{y}(\bu_1))\big]\big \|^2_{L^2(\Omega_{W,T}; L^2(Y_x))} \!\! + \|\be(\bu^\ve) - \be(\bu)\|^2_{L^2(\Omega_{ML, T})}\\
& \leq \sup\limits_{\Omega_{W,T}}\!\! \big|\bbE^\ve(x, b^\ve)^{-1}\big| \Big(\big\| |Y_x|^{-\frac 12}\big[\bbE^\ve(x, b^\ve)^{\frac 12} \be(\bu^\ve) - B(x,y, b)^{\frac 12} (\be(\bu) +  \be_{y}(\bu_1)) \big] \big\|^2_{L^2(\Omega_{W,T}; L^2(Y_x))}\\
& \qquad   + \big\| |Y_x|^{-\frac 12} \big(\bbE^\ve (x, b^\ve)^{\frac 12}  - B(x, y, b)^{\frac 12} \big)(\be(\bu) +  \be_{y}(\bu_1))\big \|^2_{L^2(\Omega_{W,T}; L^2(Y_x))} \Big)
\\
& \qquad + \sup\limits_{\Omega_{ML, T}}\!\!\big| E_{M_l}(b^\ve)^{-1}\big| \Big( \big\|E_{M_l}(b^\ve)^{\frac 12}\be(\bu^\ve) - E_{M_l}(b)^{\frac 12}  \be(\bu)\big\|^2_{L^2(\Omega_{ML, T})}\\
& \qquad \qquad \qquad  \qquad \quad + \big\|\big(E_{M_l}(b^\ve)^{\frac 12} -E_{M_l} (b)^{\frac 12}\big) \be(\bu)\big\|^2_{L^2(\Omega_{ML, T})}\Big) \\
& \leq C \Big(\big \| |Y_x|^{-\frac 12} \big[\bbE^\ve(x, b^\ve)^{\frac 12} \be(\bu^\ve) - B(x,y, b)^{\frac 12} (\be(\bu) +  \be_{y}(\bu_1))\big] \big\|^2_{L^2(\Omega_{W,T}; L^2(Y_x))} \\
& \qquad  \qquad  \qquad +  \big\|E_{M_l}(b^\ve)^{\frac 12}\be(\bu^\ve) - E_{M_l}(b)^{\frac 12}  \be(\bu)\big\|^2_{L^2(\Omega_{ML, T})} \Big) + \kappa(\ve),
\end{aligned}
$$
where $\kappa(\ve) \to 0 $ as $\ve \to 0$.
 Then, using \eqref{strong_ts_1},   we obtain   the strong l-p two-scale convergence of $\be(\bu^\ve)$ in $(0,T)\times \Omega_{W}$ and strong convergence  in $L^2((0,T)\times \Omega_{ML})$.
 The properties of the l-p unfolding operator, the weak convergence of $\cT_\mL^\ve(\be(\bu^\ve))$, and the strong l-p two-scale convergence of $\be(\bu^\ve)$, give
 \begin{equation}
 \begin{aligned}
 &  0 \leq \| \cT_\mL^\ve(\be(\bu^\ve)) - \be(\bu) -  \be_y(\bu_1(\cdot, \cdot, D_x \cdot))\|^2_{L^2(\Omega_{W,T}\times Y)}
\\
& \leq  |Y|\big\| \be(\bu^\ve)\big\|^2_{L^2(\Omega_{W,T})}  -2\big \langle \cT_\mL^\ve(\be(\bu^\ve)),  \be(\bu) + \be_y(\bu_1(\cdot, \cdot, D_x \cdot)) \big\rangle_{\Omega_{W,T},  Y}\\
& \qquad \qquad \qquad \qquad + |Y| \, \big \| |Y_x|^{-\frac 12}(\be(\bu) + \be_y(\bu_1))\big\|^2_{L^2(\Omega_{W,T}; L^2(Y_x))} \to 0,
  \end{aligned}
 \end{equation}
 as $\ve \to 0$, and, hence, the strong convergence of $\cT_\mL^\ve(\be(\bu^\ve))$ in $L^2((0,T)\times \Omega_{W}\times  Y)$.
\end{proof}

As next we derive the macroscopic equations for the systems of  reaction-diffusion  and ordinary differential equations~ \eqref{main_2}.
\begin{theorem}
A sequence  of solutions of the microscopic problem  \eqref{main_2}-\eqref{BC} converges to a solution $p_\alpha, c, b \in L^2(0,T; H^1(\Omega))$, with $\partial_t p_\alpha, \partial_t c, \partial_t b \in L^2(0,T; H^1(\Omega)^\prime)$,   $b_s \in H^1(0,T; L^2(\Omega_W; L^2(Y_x)))$, and $b_s \in H^1(0,T; L^2(\Omega_{ML}))$  of the macroscopic equations
  \begin{eqnarray}\label{macro_macro}
  \begin{aligned}
   & \theta(x) \partial_t p_E - {\rm div} (\mathcal D_E(x) \nabla p_E) = 0, \\
 & \theta(x) \partial_t  p_e - {\rm div} (\mathcal D_e(x) \nabla p_e) =  - \theta(x)  R_e(p_e, p_E),   \\
 & \theta(x) \partial_t p_d - {\rm div} (\mathcal D_d(x) \nabla p_d) = \; \theta(x) R_e(p_e, p_E) \\
& \hspace{ 3.5 cm}  + 2 \theta(x) \big[\mathcal R_{s}(b, b_s, \be(\bu)) - R_d(p_d, c)\big], \\
&\theta(x)  \partial_t c - {\rm div} ( \mathcal D_c(x) \nabla c) =  \theta(x)\big[\mathcal R_{s}(b, b_s, \be(\bu))-  R_d(p_d, c)\big] ,  \\
&\theta(x) \partial_t b -  {\rm div} (\mathcal D_b(x) \nabla b) =  \theta(x)\big[ R_d(p_d, c) - \mathcal R_{b}(b, \be(\bu))  -  d_b  b\big] ,
 \end{aligned}
 \end{eqnarray}
in~$(0,T)\times \Omega$  and  the two-scale equation
 \begin{equation}\label{macro_bs}
 \begin{aligned}
 \partial_t b_s =& R_b(b, \mathbb W(t,x,y) \be(\bu)) 
 - R_s(b, b_s, \mathbb W(t,x,y) \be(\bu)) - d_s b_s  && \text{for } (t,x) \in \Omega_{W, T}, y\in  Y_x^\ast, \\
 \partial_t b_s =& R_b(b,  \be(\bu)) - R_s(b, b_s,  \be(\bu)) - d_s b_s && \text{for  } (t,x) \in  \Omega_{ML, T},
 \end{aligned}
\end{equation}
 with the boundary conditions
\begin{eqnarray}\label{macro_macro_BC}
\begin{aligned}
& \mathcal D_\alpha(x) \nabla p_\alpha \cdot \nu =  J_\alpha(p_e, p_d, b) - \gamma_\alpha p_\alpha && \text{ on } (0,T)\times \Gamma_\cI,\\
& \mathcal D_c(x) \nabla c \cdot \nu = J_c(\bu) - \gamma_c c  &&  \text{ on } (0,T)\times \Gamma_\cI,\\
& \mathcal D_\alpha(x) \nabla p_\alpha \cdot \nu =   0, \qquad  \mathcal D_c(x) \nabla c \cdot \nu =   0   && \text{ on } (0,T)\times \Gamma_\cE,  \\
& \mathcal D_b(x) \nabla b \cdot \nu =   0   &&\text{ on } (0,T)\times \partial \Omega,
\end{aligned}
\end{eqnarray}
for  $\alpha=e,E,d$, where $J_d(p_e,p_d, b) \equiv 0$,
and   the initial conditions
\begin{equation}\label{init_macro}
\begin{aligned}
&p_\alpha(0) = p_{\alpha,0},\;   c(0) = c_0,\;  b(0)=b_0  && \text{in } \Omega, \;\; \text{ for } \alpha = e,E,d, \\
&b_s(0) = b_{s,0}  &&\text{in } \Omega_W\times  Y_x^\ast \, \text{ and }\, \Omega_{ML}.
\end{aligned}
\end{equation}
 Here $\theta(x) = |Y^\ast_{x}|/|Y_x|$ for $x \in \Omega_W$,  $\theta(x) = 1$ for $x \in \Omega_{ML}$, where  $ Y_x = D_x Y$ and   $Y_{x}^\ast= Y_x \setminus  R_x\overline Y_0$, with $D_x= R_x W(x)$, $R_x = R(\gamma(x_3))$. The reaction terms are defined  as
 $$
 \begin{aligned}
 \mathcal R_{b}(b, \be(\bu))& = \ddashinttt_{Y_x^\ast}\!\!  R_b(b, \mathbb W(t,x,y) \be(\bu))  dy &&\text{ for } (t,x) \in  \Omega_{W, T}, \\
 \mathcal R_{s} (b, b_s, \be(\bu))&=\ddashinttt_{ Y_x^\ast}\!\! R_s(b, b_s, \mathbb W(t,x,y) \be(\bu))  dy \;\; &&\text{ for } (t,x) \in  \Omega_{W, T}, \\
 \mathcal R_{b}(b, \be(\bu)) & = r_b(b) Q_b(b,  \be(\bu))  &&   \text{ for } (t,x) \in \Omega_{ML, T}, \\
 \mathcal R_{s} (b, b_s, \be(\bu))&=r_s(b_s)Q_s(b,  \be(\bu)) &&   \text{ for } (t,x) \in \Omega_{ML, T},
\end{aligned}
$$
where
$$
\begin{aligned}
R_b(b, \mathbb W(t,x,y) \be(\bu)) & = r_b(b) Q_b(b, \mathbb W(t,x,y) \be(\bu)), \\
R_s(b, b_s, \mathbb W(t,x,y) \be(\bu)) & = r_s(b_s)Q_s(b, \mathbb W(t,x,y) \be(\bu)),
\end{aligned}
$$ with
$\mathbb W_{ijkl}(t,x,y) = \delta_{ik}\delta_{jl} + \big( \be_y({\tilde \omega}_{ij} (t,x,R^{-1}_{x}y)) \big)_{kl}$ and   ${\tilde \omega}_{ij}$ are  solutions of~\eqref{unit_loc_per}. The macroscopic diffusion coefficients $\mathcal D_\alpha(x)$ are given by
 $$
 \begin{aligned}
 \mathcal D_{\alpha, ij}(x) &= D_\alpha\frac 1{|Y_x|}\!\!\int_{Y^\ast_{x}}\!\! \Big[ \delta_{ij} +  \partial_{y_i} w^j(x,y) \Big] dy \;   &&\text{ for }  x\in \Omega_W, \\
 \mathcal D_\alpha(x) &= D_\alpha {\bf I} \;  &&\text{ for } x \in \Omega_{ML},
 \end{aligned}
 $$
 for   $\alpha=e, E, d, c, b$, where $w^j$, for $j=1,2,3$, are solutions of the  cell problems
 \begin{eqnarray}\label{unit_cell}
 \begin{aligned}
 \nabla_y \cdot (\nabla_y w^j + l_j)& = 0 \quad && \text{ in }   Y^\ast_{x}, \quad \\
  (\nabla_y w^j + l_j)\cdot \nu &=0 \quad && \text{ on } \Gamma_x, \quad w^j \; \; Y_x-\text{periodic},
 \end{aligned}
 \end{eqnarray}
for   the canonical basis  $\{l_j\}_{1\leq j \leq 3}$  in $\mathbb R^3$ and $ \Gamma_x = R_x \Gamma$.
\end{theorem}
\begin{proof}
 In order to pass to the limit in the nonlinear reaction functions, we  prove the strong  convergence of~$\T_\mL^{\ve} (b^{\ve}_s)$ in $L^2((0,T)\times \Omega_W\times Y)$ and of $b^\ve_s$ in $L^2((0,T)\times \Omega_{ML})$,  by showing the Cauchy property. Considering~\eqref{weak_sol_bs} for $b_s^{\ve_n}$ and $b_s^{\ve_m}$,   using an extension of $b_s^\ve$ from $\Omega^\ve_M$ into $\Omega$, and   applying the l-p unfolding operator, we obtain
\begin{equation}\label{conv:strong1}
 \begin{aligned}
 &   \int_0^\tau \!\!\!\big\|[\cT_\mL^{\ve_n}(Q_s(b^{\ve_n}, \be(\bu^{\ve_n})))]^{\frac 12} \big|\cT_\mL^{\ve_n}(b^{\ve_n}_s) -\cT_\mL^{\ve_m}( b^{\ve_m}_s) \big|\big\|^2_2 d\tau\\
&  + \|\cT_\mL^{\ve_n}(b^{\ve_n}_s)(\tau) - \cT_\mL^{\ve_m}(b_s^{\ve_m})(\tau) \|^2_2
\leq  C_1 \!\!\int_0^\tau \!\! \Big[ \|\cT_\mL^{\ve_n}(b^{\ve_n}) -\cT_\mL^{\ve_m}(b^{\ve_m}) \|^2_2
\\
& \qquad +
  \|\T_\mL^{\ve_n} (b^{\ve_n}_s) - \T_\mL^{\ve_m} (b^{\ve_m}_s)\|^2_2
  + \|\cT_\mL^{\ve_n}(\be(\bu^{\ve_n})) - \cT_\mL^{\ve_n}(\be(\bu^{\ve_m})) \|^2_2 \Big] dt \\
&   +
 C_2\!\! \int_0^\tau \!\!\!\int_{\Omega_W\times  Y}\!\!|\T_\mL^{\ve_n} (b^{\ve_n}) - \T_\mL^{\ve_m} (b^{\ve_m})|
  |\T_\mL^{\ve_n} (\be(\bu^{\ve_n}))| \times \\
&\qquad \qquad \qquad   \times   |\T_\mL^{\ve_n} (b_s^{\ve_n}) - \T_\mL^{\ve_m} (b_s^{\ve_m})|dydx dt  =  \int_0^\tau \!\! \big[I_1(t)+ I_2(t)\big] dt,
 \end{aligned}
\end{equation}
for $\tau \in (0, T]$, where $\|v\|_2 = \|v\|_{L^2(\Omega_W\times   Y)}$.
The boundedness of $b^\ve_s$ and the estimates for $\bu^\ve$, yields
$$
\begin{aligned}
 I_2(t) \leq  \| \T_\mL^{\ve_n} (\be(\bu^{\ve_n}))\|_{2} \|\T_\mL^{\ve_n} (b^{\ve_n}) - \T_\mL^{\ve_m} (b^{\ve_m})\|_{4} \|\T_\mL^{\ve_n} (b_s^{\ve_n}) - \T_\mL^{\ve_m} (b_s^{\ve_m})\|_{4} \\
 \leq C_1  \|\T_\mL^{\ve_n} (b^{\ve_n}) - \T_\mL^{\ve_m} (b^{\ve_m})\|^{\frac 43}_{4}
 + C_2\|\T_\mL^{\ve_n} (b_s^{\ve_n}) - \T_\mL^{\ve_m} (b_s^{\ve_m})\|^2_{2}.
\end{aligned}
$$
Using in~\eqref{conv:strong1} the Gr\"onwall inequality, together with  the   strong convergence of $b^\ve$ and  $\T_\mL^{\ve} (\be(\bu^{\ve}))$,  we obtain
\begin{equation}
 \|\T_\mL^{\ve_n} (b^{\ve_n}_s(\tau)) - \T_\mL^{\ve_m}(b_s^{\ve_m}(\tau)) \|^2_{L^2(\Omega_W\times  Y)} \leq C \sigma(n,m), \quad \text{ for } \; \tau \in (0,T],
\end{equation}
 where $\sigma(n,m) \to 0$ as $n,m \to \infty$, and  hence the strong convergence of $\T_\mL^{\ve} (b^{\ve}_s)$ in $L^2(\Omega_{W,T} \times  Y)$. Similar calculations, considering equations for $b_s^{\ve_n}$ and $b_s^{\ve_m}$ in $\Omega_{ML, T}$ and using the strong convergence of $b^\ve$ and $\be(\bu^\ve)$ in $L^2(\Omega_{ML, T})$, yield the strong convergence of $b_s^\ve$ in $L^2(\Omega_{ML, T})$.

To derive the macroscopic equations~\eqref{macro_macro}, we  first consider the equation for $p_d^\ve$ in~\eqref{weak_sol_n1} and rewrite it  as
 \begin{equation}\label{micro_pd_11}
 \begin{aligned}
  - \langle p^\ve_d, \partial_t \phi_d \chi_{\tilde \Omega_M^\ve}\rangle_{\Omega_T}  + \langle D_d \nabla p^\ve_d, \nabla \phi_d \chi_{\tilde \Omega_M^\ve}\rangle_{\Omega_T}  - \langle  R_{e}(p^\ve_e, p^\ve_E), \phi_d \chi_{\tilde \Omega_M^\ve} \rangle_{\Omega_T} \qquad \\
 \qquad \quad  + 2  \langle  R_{d}(p_d^\ve, c^\ve)- R_s(b^\ve, b_s^\ve, \be(\bu^\ve)) , \phi_d \chi_{\tilde \Omega_M^\ve} \rangle_{\Omega_T} +  \langle \gamma_d p_d^\ve, \phi_d \rangle_{\Gamma_{\cI,T}}
\\
 - \langle p^\ve_d, \partial_t \phi_d (\chi_{\Omega_M^\ve}- \chi_{\tilde \Omega_M^\ve})\rangle_{\Omega_T}
+ \langle D_d \nabla p^\ve_d, \nabla \phi_d (\chi_{ \Omega_M^\ve}- \chi_{\tilde \Omega_M^\ve})\rangle_{\Omega_T} \qquad \qquad \quad
\\
 + \langle 2[ R_{d}(p_d^\ve, c^\ve)-R_s(b^\ve, b_s^\ve, \be(\bu^\ve))] -  R_{e}(p^\ve_e, p^\ve_E), \phi_d (\chi_{\Omega_M^\ve}- \chi_{\tilde \Omega_M^\ve}) \rangle_{\Omega_T}=0,
 \end{aligned}
 \end{equation}
for $\phi_d \in C^1_0(0,T; C^1(\overline \Omega))$,  where
$\tilde \Omega_M^\ve = \Omega \setminus \big(\bigcup_{n=1}^{N_\ve} \hat \Omega_{0,n}^{\ve}\big)$  with $\hat \Omega_{0,n}^{\ve} = \bigcup_{\xi \in \hat \Xi_n^\ve} \ve R_{x_n^\ve}(\overline Y_0 + W_{x_n^\ve}\xi)$, for
$ \hat \Xi_n^\ve = \{ \xi \in \mathbb Z^3 : \ve  D_{x_n^\ve}(\overline Y + \xi) \subset \Omega_n^\ve \cap \Omega_W\}$. 
For  $\chi_{\Omega^\ve_M}-\chi_{\widetilde \Omega^\ve_M}$, using~\eqref{estim_non_local_1}, we obtain
$$
\begin{aligned}
& \|\chi_{\Omega^\ve_M} - \chi_{\widetilde \Omega^{\ve}_M}\|^2_{L^2(\Omega)}\\
 & \leq
 \!\sum_{n=1}^{N_\ve} \!\int_{\Omega_n^\ve} \! \sum_{j =1}^{I_n^\ve}   \big| \vartheta_\ve(R^{-1}_{x_{k_j}^\ve}(x-x_{k_j}^\ve))
-   \vartheta_\ve(R^{-1}_{x_n^\ve}(x- x_n^\ve)-  \ve W_{x_n^\ve} m_j)\big |^2 dx
\leq  C \ve^{3r-2} \!\!\to 0,
\end{aligned}
$$
as $\ve \to 0$,
for $r \in (2/3,1)$, where $x^\ve_n=R_{x^\ve_n} \ve \kappa_n$ and  $x^\ve_{k_{j}}= R_{x^\ve_{k_j}} \ve k_{j}$,  with $k_j=\kappa_n+ m_j$ and $m_j \in \mathbb Z^3$.

The local Lipschitz continuity of $Q_l$ and extension of $b^\ve$ from $\Omega_M^\ve$ into $\Omega$   ensure
\begin{equation*}
\begin{aligned}
& \|\cT^\ve_\mL(Q_l(b^\ve, \be(\bu^\ve)) )  -Q_l(b, \be(\bu) +  \be_y(\bu_1)) \|_{L^1(\Omega_{W,T} \times  Y)}
\\ & \quad  \leq  C\|\be(\bu^\ve)\|_{L^2(\Omega_{W,T})} \|\cT_\mL^\ve(b^\ve) - b\|_{L^2(\Omega_{W,T}\times  Y)}
 \\ & \qquad  +  C \|b\|_{L^2(\Omega_{W,T})} \|\cT_\mL^\ve(\be(\bu^\ve)) - \be(\bu) -  \be_y(\bu_1)\|_{L^2(\Omega_{W,T}\times Y)},
\end{aligned}
\end{equation*}
for $l=b,s$. Then  the   strong  convergence of $\cT_\mL^\ve(\be(\bu^\ve))$ and $\cT_\mL^\ve(b^\ve)$  implies
\begin{equation*}
\begin{aligned}
&\lim\limits_{\ve \to 0} \big \langle \cT_\mL^\ve(Q_l(b^\ve,  \be(\bu^\ve)) ),   \psi \big \rangle_{\Omega_{W,T},  Y} =
\big \langle  Q_l(b, \be(\bu) +  \be_y(\bu_1)),   \psi \big \rangle_{\Omega_{W,T},  Y},
 \end{aligned}
\end{equation*}
for all $\psi \in C_0(\Omega_{W,T} \times  Y)$, and we conclude, using a priori estimates for $b^\ve$ and $\bu^\ve$, that
$$
\cT_\mL^\ve(Q_l(b^\ve, \be(\bu^\ve))) \rightharpoonup Q_l(b, \be(\bu) +  \be_y(\bu_1)) \quad \text{ weakly  in } L^2(\Omega_{W,T}\times   Y), \quad \text{ for } \; l=b,s.
$$
Due to the relation between the l-p two-scale convergence  and  the weak convergence of the unfolded sequence~\cite{Ptashnyk_2013}, we have
\begin{equation*}\label{conv_Q}
Q_l(b^\ve, \be(\bu^\ve)) \rightharpoonup Q_l(b, \be(\bu) +  \be_y(\bu_1)) \qquad \text{l-p two-scale}, \quad \text{ for }  l=b,s.
\end{equation*}
Similarly
\begin{equation*}
\begin{aligned}
& \|\cT^\ve_\mL(R_s(b^\ve, b^\ve_s,\be(\bu^\ve)) )  -R_s(b, b_s, \be(\bu) + \be_y(\bu_1)) \|_{L^1(\Omega_{W,T}\times  Y)}
\\
&\leq  C_1\|\be(\bu^\ve)\|_{2}\big[\|b_s^\ve\|_{L^\infty} \|\cT_\mL^\ve(b^\ve) - b\|_2 + \|b\|_{L^\infty} \|\cT_\mL^\ve(b^\ve_s) - b_s\|_2\big]\\
 & \qquad +  C_2 \|b^\ve_s\|_{2}\|b\|_{L^\infty} \|\cT_\mL^\ve(\be(\bu^\ve)) - \be(\bu) -  \be_y(\bu_1)\|_2,
\end{aligned}
\end{equation*}
where $\|v\|_2 = \|v\|_{L^2(\Omega_{W,T}\times  Y)}$.
Thus, the   strong  convergence of $\cT_\mL^\ve(\be(\bu^\ve))$, $\cT^\ve_\mL(b^\ve)$, and $\cT_\mL^\ve(b^\ve_s)$ implies  the l-p two-scale convergence of $R_s(b^\ve, b^\ve_s,\be(\bu^\ve)) $ to $R_s(b, b_s,\be(\bu)+ \be(\bu_1))$.
The a priori estimates  and the strong convergence of $b^\ve$, $b_s^\ve$, and $\be(\bu^\ve)$ ensure   $R_b(b^\ve,\be(\bu^\ve)) \rightharpoonup  R_b(b,\be(\bu))$ and
$R_s(b^\ve, b^\ve_s,\be(\bu^\ve)) \rightharpoonup  R_s(b, b_s,\be(\bu))$ in~$L^2(\Omega_{ML, T})$.

Taking   $\phi_d^\ve(t,x) = \phi_d^1(t,x) + \ve \mathcal L_\rho^\ve(\phi_d^2)(t,x)$,  with
$\phi_d^1 \in C_0^1(0,T; C^1(\overline \Omega))$,  $\phi_d^2 \in C_0^1(\Omega_{W,T}; C^1_{\rm per}(Y_x))$, as a test function in~\eqref{micro_pd_11}, passing to   the l-p two-scale limit  and using the convergence results in~Lemma~\ref{converg_2} and the  l-p two-scale convergence of $R_s(b^\ve, b^\ve_s,\be(\bu^\ve)) $, we obtain
 \begin{equation}\label{macro_pd_11}
 \begin{aligned}
  - \big\langle |Y_x|^{-1}p_d, \partial_t \phi_d^1 \big\rangle_{\Omega_{W,T}, Y_x^\ast}
 \! + \big\langle | Y_x|^{-1} D_d (\nabla p_d + \nabla_y p_d^1), \nabla \phi_d^1 +\nabla_y \phi_d^2\big\rangle_{\Omega_{W,T},  Y_x^\ast}
  \\
  - \langle p_d, \partial_t \phi_d^1 \rangle_{\Omega_{ML,T}}
  + \langle  D_d \nabla p_d , \nabla \phi_d^1\rangle_{\Omega_{ML,T}}
     +  \langle \gamma_d p_d, \phi_d^1 \rangle_{\Gamma_{\cI,T}} \qquad
 \\
  +  \big\langle |Y_x|^{-1} \big[2R_{d}(p_d, c)- 2 R_s(b, b_s, \be(\bu)+\be(\bu_1)) -  R_{e}(p_e, p_E)\big] , \phi_d^1  \big\rangle_{\Omega_{W,T}, Y^\ast_x}\\
  +  \big\langle  2R_{d}(p_d, c)- 2 R_s(b, b_s, \be(\bu)) -  R_{e}(p_e, p_E), \phi_d^1  \big\rangle_{\Omega_{ML,T}}=0.
 \end{aligned}
 \end{equation}
Choosing  $\phi_\alpha^\ve(t,x) = \phi_\alpha^1(t,x) + \ve \mathcal L_\rho^\ve(\phi_\alpha^2)(t,x)$,  with
$\phi_\alpha^1 \in C_0^1(0,T; C^1(\overline \Omega))$,   $\phi_\alpha^2 \in C_0^1(\Omega_{W,T}; C^1_{\rm per}( Y_x))$,  for $\alpha=e, E,c,b$,
 as  test functions in  the equations for $p_E^\ve$, $p_e^\ve$, $c^\ve$, $b^\ve$ in~\eqref{weak_sol_n1},   integrating by parts in time,  and considering the l-p two-scale convergence in the same way as in~\eqref{micro_pd_11},     we derive the macroscopic equations~\eqref{macro_macro}.
 The strong convergence of $\bu^\ve$ in $L^2((0,T)\times \Gamma_\cI)$, see Lemma~\ref{converg_2}, ensures convergence in the boundary conditions for $c^\ve$ on~$(0,T)\times \Gamma_\cI$.
 Standard calculations, see e.g.~\cite{CDG_book}, considering first $\phi_\alpha^1=0$ and then $\phi_\alpha^2 =0$ in the limit equations, imply the  cell problems~\eqref{unit_cell} and the macroscopic diffusion coefficients $\mathcal D_\alpha$, for $\alpha = E,e,d,c,b$.
 To derive the boundary conditions~\eqref{macro_macro_BC} we consider $\phi_\alpha^1 \in C^1_0(0,T;  C^1(\overline \Omega))$, that is  zero  first at $\Gamma_\cI$ and then at $\Gamma_{\cE}$, for $\alpha = e, E, d, c,b$.
 Since $0<\theta_0 \leq \theta(x) \leq 1$ for  $x \in \Omega$, the standard arguments yield the regularity of the time derivatives $\partial_t p_\alpha, \partial_t c, \partial_t b \in L^2(0,T; H^1(\Omega)^\prime)$.
 Choosing  $\phi_\alpha^1 \in C^1([0,T]\times \overline \Omega)$, with $\phi_\alpha^1(T) =0$, for $\alpha = e, E, d, c,b$,  and integrating by parts in time imply the initial conditions~\eqref{init_macro}.
Testing~\eqref{weak_sol_bs} with
 $\phi^\ve(t,x) = \mathcal L^\ve(\phi_s^2)(t,x)$ for $(t,x) \in \Omega_{W,T}$ and $\phi^\ve(t,x) = \phi_s^1(t,x)$ for $(t,x) \in \Omega_{ML, T}$, where $\phi_s^2 \in C_0(\Omega_{W,T}; C_{\rm per}( Y_x))$ and $\phi_s^1 \in C_0(\Omega_{ML,T})$,
 and using the convergence of the reaction terms  we obtain the macroscopic equations~\eqref{macro_bs}. The time-regularity of $b_s$ implies that the associated  initial conditions  in~\eqref{init_macro} are satisfied.

 The uniqueness of a solution to the macroscopic model~\eqref{eq:macro_u}  and \eqref{macro_macro}-\eqref{init_macro} is shown by contradiction. Assuming that there are two solutions and performing calculations similar to those  in Theorem~\ref{th:existence_micro}  we derive the contraction inequality, which then  implies the uniqueness of the solution.
%
%
\end{proof}

\section{Conclusion}\label{conclusion}
This work presents the derivation and multiscale analysis of a new model for plant tissue biomechanics. In
comparison to previous results~\cite{PS_1_2016, PS_2_2016, PS_2017}, the main novelty lies in  a more detailed modelling of  the dynamics of calcium-pectin cross-links, where we distinguish between the stretching of newly created  and the breakage of stretched cross-links. Another novel aspect is the analysis  of non-periodic microscopic structure of plant cell walls, characterised by rotated planes of parallel-aligned microfibrils. The derived microscopic model consists of a strongly coupled system of reaction-diffusion  and  ordinary differential equations, with reaction terms depending on the displacement gradient,
coupled with the equations of linear elasticity, whose elastic moduli depend on the density of calcium–pectin cross-links.
The main challenge and novelty of the analysis include  the derivation of  a priori estimates  and the corresponding contraction inequalities in $L^\infty$ in the  proof of existence of solutions to the microscopic model, as well as the derivation of the equicontinuity property with respect to the time variable for the displacement gradient. Since the dynamics of stretched cross-links is modelled by an ordinary differential equation
with reaction terms depending on the deformation gradient, new ideas were applied in the derivation of a priori estimates, compared to the previous work~\cite{PS_1_2016}.

For the multiscale analysis and the derivation of the macroscopic model, the non-periodic microscopic structure of plant cell walls was approximated by a locally periodic microstructure and the methods of l-p two-scale convergence and the l-p unfolding operator were applied. In order to pass to the limit in the nonlinear reaction terms,  strong l-p two-scale convergence for the displacement gradient and for the density of stretched cross-links was shown. To show the strong l-p two-scale convergence of the displacement gradient a standard approach, that consists of   considering the microscopic and  macroscopic equations and  showing the convergence of the associated norms, was applied.
To prove the strong l-p two-scale convergence of the density of  stretched cross-links, which satisfies an ordinary differential equation defined in a perforated domain,  an extension argument  and the l-p unfolding operator are used. This approach allows us to  establish the Cauchy property for the corresponding unfolded sequence, which then implies   strong  l-p two-scale convergence of the original sequence.

The multiscale modelling and analysis of plant tissue biomechanics considered in this work provide a framework for the detailed modelling of interactions between mechanical properties, microscopic structure, and molecular processes in biological tissues. They also enable efficient numerical simulations of the underlying processes by means of the corresponding macroscopic model, which provides  a rigorous connection between microscopic properties at the cellular level  and macroscopic behaviour at the tissue level. Accounting for the non-periodic microstructure of plant cell walls allows us to analyse the impact of  a rotated distribution of microfibrils on mechanical properties and deformation of biological tissues~\cite{Anderson, McCulloch, PS_2_2016}.

\section*{Acknowledgments}
MP would like to thank Erwin Schr\"odinger International Institute for Mathematics and Physics (ESI)  for support and hospitality during the Thematics programme ``Free Boundary Problems'' when work on this paper was undertaken.

\section*{Appendix}
The   elasticity tensor $\bbE_F$ for transversal isotropic microfibres  is determined by the Young's modulus $E_F$ associated with the directions  perpendicular to the microfibril, the Poisson's ratio $\nu_{F1}$ characterizing the transverse reduction of the plane perpendicular to the microfibril for stress lying in this plane, the ratio $n_F$ between $E_F$ and the Young's modulus associated with the direction of the axis of the microfibril, the Poisson's ratio $\nu_{F2}$ governing the reduction in the plane perpendicular to the microfibril for stress in the direction of the microfibril, and the shear modulus $Z_F$ for planes parallel to the microfibril.
In the  Voigt notation $\bbE_F$ can be written as
$$
\bbE_F = \begin{pmatrix}
\alpha_2+\alpha_5 & \alpha_2-\alpha_5 & \alpha_3 & 0 & 0 & 0\\
\alpha_2-\alpha_5 & \alpha_2+\alpha_5 & \alpha_3 & 0 & 0 & 0\\
\alpha_3 & \alpha_3 & \alpha_1 & 0 & 0 & 0\\
0 & 0 & 0 & \alpha_4 & 0 & 0\\
0 & 0 & 0 & 0 & \alpha_4 & 0\\
0 & 0 & 0 & 0 & 0 & \alpha_5
\end{pmatrix},
$$
where $\alpha_i$, for $i=1,2,3,4,5$, are related to the five parameters  through
\begin{align*}
\alpha_1&=\frac{E_F(1-\nu_{F1})}{n_F(1-\nu_{F1})-2\nu_{F2}^2}, &
\alpha_2&=\frac{E_Fn_F}{2n_F(1-\nu_{F1})-4\nu_{F2}^2},\\
\alpha_3&=\frac{E_F\nu_{F2}}{n_F(1-\nu_{F1})-2\nu_{F2}^2},&
\alpha_4&=Z_F,&
\alpha_5&=\frac{E_F}{2(1+\nu_{F1})}.
\end{align*}

Zsivanovits et al.~\cite{ZMSR} measured the Young's modulus of a pectin network made from mixing de-estrified apple pectin and various concentrations of calcium.  In particular, they found that for calcium concentrations of $10$ mM and $30$ mM, the network has a Young's modulus of $15$ MPa and $33$ MPa, respectively.
Assuming that all of the calcium is converted into cross-links, $E$ can be given as a function of the concentration of cross-links
$
E(b) = 0.9\, b\,  \text{MPa}/\text{mM} + 10.5\, \text{MPa}.
$
\end{document}